\documentclass[final,12pt]{colt2024}

\usepackage{color}
\usepackage{url}
\usepackage{comment}
\usepackage{tablefootnote}
\usepackage{bbm}
\usepackage{float}
\usepackage{amsfonts}
\usepackage{amsmath}
\usepackage{wrapfig}
\usepackage{caption}
\usepackage{outlines}
\usepackage{amssymb}
\usepackage{MnSymbol}
\usepackage{graphicx}

\def\ip#1#2{{\left\langle#1,#2\right\rangle}}
\def\lb#1#2{{\left[#1,#2\right]}}
\def\norm#1{{\left\|#1\right\|}}
\newcommand\ad{\operatorname{ad}}

\newcommand\tr{\operatorname{tr}}
\newcommand\divg{\operatorname{div}}
\newcommand\scfunc{\operatorname{sc}}
\newcommand\rcfunc{\operatorname{rc}}
\newcommand\Law{\operatorname{Law}}
\def\lt#1{T_{#1}\mathsf{L}_{{#1}^{-1}}}
\def\lte#1{T_e\mathsf{L}_{#1}}
\def\M{\mathcal{M}}
\def\d{\mathrm{d}}
\def\G{\mathsf{G}}
\def\g{\mathfrak{g}}

\def\KL#1#2{{\operatorname{KL}\left(#1\|#2\right)}}

\newtheorem{assumption}{Assumption}

\graphicspath{{./figs/}}

\title[Convergence of Kinetic Langevin Monte Carlo on Lie groups]{Convergence of Kinetic Langevin Monte Carlo on Lie groups}

\coltauthor{
\Name{Lingkai Kong} \Email{lkong75@gatech.edu}\\
\Name{Molei Tao} \Email{mtao@gatech.edu}\\
\addr Georgia Institute of Technology
}

\begin{document}
\maketitle
\begin{abstract}
    Explicit, momentum-based dynamics for optimizing functions defined on Lie groups was recently constructed, based on techniques such as variational optimization and left trivialization. We appropriately add tractable noise to the optimization dynamics to turn it into a sampling dynamics, leveraging the advantageous feature that the trivialized momentum variable is Euclidean despite that the potential function lives on a manifold. We then propose a Lie-group MCMC sampler, by delicately discretizing the resulting kinetic-Langevin-type sampling dynamics. The Lie group structure is exactly preserved by this discretization. Exponential convergence with explicit convergence rate for both the continuous dynamics and the discrete sampler are then proved under $W_2$ distance. Only compactness of the Lie group and geodesically $L$-smoothness of the potential function are needed. To the best of our knowledge, this is the first convergence result for kinetic Langevin on curved spaces, and also the first quantitative result that requires no convexity or, at least not explicitly, any common relaxation such as isoperimetry.
\end{abstract}


\section{Introduction}
Sampling is a classical field that is nevertheless still rapidly progressing, with new results on quantitative and nonasymptotic guarantees, new algorithms, and appealing machine learning applications. One type of methods is constructed by, or at least interpretable as discretizing a continuous time SDE whose equilibrium is the target distribution. Some famous examples are Langevin Monte Carlo (LMC), which originates from overdamped Langevin SDE, and kinetic Langevin Monte Carlo (KLMC), which originates from kinetic Langevin SDE.  Both LMC and KLMC are widely used gradient-based methods. Similar to the fact that momentum helps gradient descent \citep{nesterov2013introductory}, KLMC can be interpreted as a momentum version of LMC. 

Nonasymptotic error analysis of momentumless Langevin algorithm in Euclidean spaces have be established by a collection of great works \citep[e.g.,][]{dalalyan2017theoretical,cheng2018convergence,durmus2019analysis,vempala2019rapid,li2021sqrt}. Samplers based on kinetic Langevin are more difficult to analyze, largely due to the degeneracy of noise as it is only added to the momentum but not the position, but many remarkable progress  \cite[e.g.,][]{shen2019randomized, dalalyan2017theoretical, ma2021there, zhang2023improved, yuan2023markov, altschuler2023faster} have been made too, showcasing the benefits of momentum. More discussion of sampling in Euclidean spaces can be found in Sec. \ref{sec_related_work_Euclidean}.

Sampling on a manifold that does not admit a global coordinate is much harder. Even the exact numerical implementation of Brownian motion on it is difficult for general manifolds \citep{hsu2002stochastic}. Seminal results for general Riemannian manifolds with nonasymptotic error guarantees include \cite{cheng2022efficient,wang2020fast,gatmiry2022convergence}, leveraging either a discretization of the Brownian motion or assuming an oracle that exactly implements Brownian motion (which is feasible for certain manifolds). However, all these results focused on the momentumless case.

For a problem closely related to sampling, namely optimization though, several great results already exist for momentum accelerated optimization on Riemannian manifold, including both in continuous time \citep[e.g.,][]{alimisis2020continuous} and in discrete time \citep[e.g.,][]{ahn2020nesterov}. It is possible to formulate sampling as an optimization problem, where the convergence of a sampling dynamics/algorithm is characterized as conducting optimization in the infinite-dimensional space of probability densities \citep[e.g.,][]{vempala2019rapid}. However, when momentum is introduced to the manifold sampling problem, we have to be considering the convergence of densities defined on the tangent bundle of the manifold. This makes the analysis much more challenging. To the best of our knowledge, no convergence result exists for kinetic Langevin on curved spaces, in neither continuous or discrete time cases. See Sec. \ref{sec_momentum_meets_curved_spaces} for more discussions of those difficulties.

This paper considers a special class of manifolds, known as the Lie groups. A Lie group is a differential manifold with group structure. Many widely used manifolds in machine learning indeed have Lie group structures; one example is the set of orthogonal matrices (with det=1), i.e. $\mathsf{SO}(n)$; see e.g., \citet{lezcano2019cheap, tao2020variational, kong2022momentum} for optimization on $\mathsf{SO}(n)$. However, fewer results have been established for sampling on Lie groups, especially with momentum. The interesting work by \cite{arnaudon2019irreversible} might be the closest to ours in this regard, but the kinetic Langevin dynamics that we are proposing was not explicitly worked out, nor its discretization, and neither the uniqueness of invariant distribution or a convergence guarantee was provided, let alone convergence rate. In fact, similar to the general manifold case, we are unaware of prior construction of kinetic-Langevin-type samplers for Lie groups, and certainly not performance quantification.
\subsection{A brief summary of main results}
Our main contributions are:
\begin{enumerate}
    \item 
        We provide the first momentum version of Langevin-based algorithm for sampling on the curved spaces of Lie groups, with rigorous and quantitative analysis of the geometric ergodicities of both the sampling dynamics (i.e. in continuous time) on which it is based, and the sampling algorithm per se (i.e. in discrete time) 
    \item
        The exponential convergence on Lie groups is proved under weaker assumptions than considered in the literature for general manifolds, as we leverage their additional group structure. Only compactness of the manifold and smoothness of the log density are used. No (geodesic-)convexity or explicit isoperimetric inequalities are needed.
    \item
        Our algorithm is fully implementable in the sense that it requires no implementation of Brownian motion on curved spaces. It is also computationally efficient as it is based on explicit numerical discretization that preserves the manifold structure; that is, unlike commonly done, no extra projection back to the manifold (which can be computationally expensive) is needed. Note that our discretization is different from exponential integrators commonly used in Euclidean cases, and requires new analysis.
\end{enumerate}

\noindent
More specifically, consider sampling from a target distribution with density $\propto \rho(g)$. Letting $U=-\log\rho$ and calling it potential function ($U: \G\rightarrow \mathbb{R}$, where $\G$ is the Lie group), the sampling dynamics we construct is:
\begin{equation}
\label{eqn_sampling_SDE_ad_star_vanish}
    \begin{cases}
        \dot{g}=\lte{g} \xi\\
        \d \xi=-\gamma\xi \d t -\lt{g}(\nabla U(g))\d t+\sqrt{2\gamma}\d W_t
    \end{cases}
\end{equation}
It admits the following invariant distribution 
\begin{equation}
    \nu_*(g, \xi)=\frac{1}{Z}\exp \left(-U(g)-\frac{1}{2}\ip{\xi}{\xi}\right) \d g \d \xi
    \label{eqn_Gibbs_TG}
\end{equation}
where $\d g$ is left Haar measure and $\d \xi$ is Lebesgue. Its $g$ marginal is the target distribution.

The sampling algorithm we propose is:

\SetKwInOut{KwParameter}{Parameter}
\SetKwInOut{KwInitialization}{Initialization}
\SetKwInOut{KwOutput}{Output}
\begin{algorithm2e}[H]
    \KwParameter{step size $h>0$, friction $\gamma>0$, number of iterations $N$}
    \KwInitialization{$g_0\in \G$, $\xi_0= 0$}
    \KwOutput{A sample from $Z^{-1} \exp(-U(g))$}
    \For{$k=0,\cdots,N-1$}{
    $\xi_{k+1}=\exp(-\gamma h)\xi_k -\frac{1-\exp(-\gamma h)}{\gamma} \lt{g} \nabla U(g_k)+\mathcal{N}(0, 1-\exp(-2\gamma h))$\\
    $g_{k+1}=g_k\exp(h\xi_k)$\\
    }
     \Return{$g_N$}
     \caption{Kinetic Langevin Monte Carlo (KLMC) sampler on Lie groups}
     \label{algo_sampling_Lie_group}
\end{algorithm2e}

The convergence guarantee for the sampling dynamics (eq.\ref{eqn_sampling_SDE_ad_star_vanish}; continuous time) is 
\begin{theorem}[Convergence of Kinetic Langevin dynamics on Lie groups (Informal version)]
    Suppose the Lie group $\G$ is compact, finite-dimensional and the potential function $U$ is $L$-smooth, then we have the exponential convergence of sampling dynamics \eqref{eqn_sampling_SDE_ad_star_vanish} to the Gibbs distribution $\nu_*$, i.e.,
    \begin{align*}
    W_2(\nu_t, \nu_*)\le C_\rho e^{-ct} W_\rho(\nu_0, \nu_*)
    \end{align*}
    where $\nu_t$ is the joint density of $g(t)$ and $\xi(t)$. For notations and more details, see Thm. \ref{thm_SDE_error_W2}.
\end{theorem}

The nonasymptotic error bound for our sampler (Alg.\ref{eqn_sampling_SDE_ad_star_vanish}; discrete time) is 
\begin{theorem}[Convergence of Kinetic Langevin Sampler on Lie groups (Informal version)]
    Suppose the Lie group $\G$ is compact, finite-dimensional and the potential function $U$ is $L$-smooth, then we have the exponential convergence of our KLMC algorithm to the Gibbs distribution, i.e.,
    \begin{align*}
        W_2(\tilde{\nu}_k, \nu_*)\le C_\rho e^{-ckh} W_\rho(\nu_0, \nu_*)+\mathcal{O}(h^{1/2})
    \end{align*}
    where $\tilde{\nu}_k$ is the density for the sampler at step $k$, $\nu_*$ is that target distribution in Eq. \ref{eqn_Gibbs_TG} and $c$ is the contraction rate for the continuous dynamics. For notations $W_\rho$, $C_\rho$, a more explicit expression of $\mathcal{O}(h^{1/2})$, and more technical details, see Thm. \ref{thm_global_error_W2}.
\end{theorem}

\section{Preliminaries}
\label{sec_preliminaries}

\subsection{Lie group and Lie algebra}
\label{sec_Lie_group}
A \textbf{Lie group}, denoted as $G$, is a differentiable manifold with a group structure. A \textbf{Lie algebra} is a vector space with a bilinear, alternating binary operation that satisfies the Jacobi identity, known as Lie bracket. The tangent space at $e$ (the identity element of the group) is a \textbf{Lie algebra}, denoted as $\g:=T_e \G$. The dimension of the Lie group $\G$ will be denoted by $m$. Lie groups considered when proving the convergence of kinetic Langevin dynamics and kinetic Langevin sampler will be assumed to satisfy the following:

\begin{assumption}[General assumptions on geometry]
\label{assumption_general}
    We assume the Lie group $\G$ is finite-dimensional, connected, and compact.
\end{assumption}

The property we will use to handle momentum on Lie groups is called \textbf{left-trivialization}, which is an operation from the group structure and does not exist on a general manifold. Left group multiplication $\mathsf{L}_g: \hat{g}\to g\hat{g}$ is a smooth map from the Lie group to itself and its tangent map $T_{\hat{g}} \mathsf{L}_g: T_{\hat{g}} \G\to T_{g\hat{g}} \G$ is a one-to-one map. As a result, for any $g\in \G$, we can represent the vectors in $T_g \G$ by $\lte{g} \xi$ for $\xi \in T_e \G$. This operation gives us an optimization dynamics on Lie groups in the next section.
\subsection{Optimization dynamics based on left trivialization}
\label{sec_optimization_dynamics}
\cite{tao2020variational} proposed a Lie group optimizer based on constructing the following ODE, which performs optimization in continuous time on a general Lie group:
\vspace{-0.2cm}
\begin{equation}
    \begin{cases}
        \dot{g}=\lte{g} \xi\\
        \dot{\xi}=-\gamma\xi+\ad_\xi^*\xi-\lt{g}(\nabla U(g))
    \end{cases}
    \label{eqn_optimization_ODE}
\end{equation}
$\xi$ is the left-trivialized momentum (intuitively it could be thought of as angular momentum), and $\lte{g}\xi$ is the momentum (intuitively, think it as $g\xi$, i.e. position times angular momentum gives momentum\footnote{Technically, momentum is the dual of velocity, and these quantities should called velocity instead for rigor, but we will stick to the word `momentum' by convention.}). Potential $U:\G\to \mathbb{R}$ is the objective function of optimization. $\nabla U$ is the Riemannian gradient of the potential, and $\lt{g}(\nabla U(g))$ is its left-trivialization. This dynamics essentially models a damped mechanical system, where the total energy (sum of some kinetic energy term and potential energy $U$) is drained by the frictional forcing term $-\gamma \xi$, and $U$ is minimized at $t\to\infty$. Indeed, this is how \citet{tao2020variational} proved this ODE converges to a local minimum of $U$. The kinetic energy needs more discussion. In the Euclidean space, we have a global inner product, which is not true in curved spaces. In the Lie group case, we first define an inner product $\ip{\cdot}{\cdot}$ on $\g$, which is flat, and then move it around by the differential of left multiplication, i.e., the inner product at $T_g \G$ is for $\xi_1,\xi_2 \in T_g \G$, $\ip{\xi_1}{\xi_2}:=\ip{\lt{g} \xi_1}{\lt{g} \xi_2}$. As a result, $\frac{1}{2}\ip{\xi}{\xi}$ is the \textbf{kinetic} energy. $\gamma$ provides dissipation to the total energy (the sum of kinetic energy and potential energy). In general, $\gamma$ can be a positive function depending on time (e.g., NAG-C). However, we only consider the case $\gamma$ as a constant for simplicity.

For curved space, an additional term $\ad_\xi^*\xi$ that vanishes in Euclidean space shows up in Eq. \eqref{eqn_optimization_ODE}. It could be understood as a generalization of Coriolis force that accounts for curved geometry and is needed for free motion, see Sec. \ref{sec_discussion_ad_star} for more discussion. The \textbf{adjoint operator} $\ad:\g \times \g\to \g$ is defined by $\ad_{X} Y:=\lb{X}{Y}$. Its dual, known as the \textbf{coadjoint operator} $\ad^*:\g\times \g \to \g$, is given by $\ip{\ad^*_X Y}{Z}=\ip{Y}{\ad_X Z}, \forall Z \in \g$.

\subsection{Choice of inner product on $\g$}
\label{sec_inner_product}
The term $\ad^*_\xi \xi$ in the optimization ODE \eqref{eqn_optimization_ODE} (and also the later sampling SDE Eq.\ref{eqn_sampling_SDE}) is a quadratic term and it will make the numerical discretization that will be considered later difficult. Another more intrinsic drawback of this term is, it depends on the Riemannian metric, and indicates an inconsistency between the Riemannian structure and the group structure, i.e., the exponential map from the Riemannian structure is different from the exponential map from the group structure. See Sec. \ref{sec_about_ad_self_adjoint} for more details. Fortunately, the following lemma shows a special metric on $\g$ can be chosen to make the term $\ad^*_\xi \xi$ vanish.
\begin{lemma}[$\ad$ skew-adjoint \citep{milnor1976curvatures}]
    \label{lemma_ad_self_adjoint}
    Under Assumption \ref{assumption_general}, there exists an inner product on $\g$ such that the operator $\ad$ is skew-adjoint, i.e., $\ad^*_\xi=-\ad_\xi$ for any $\xi \in \g$. 
\end{lemma}
Rmk. \ref{rmk_inner_product_explicit_expression} in Sec. \ref{sec_about_ad_self_adjoint} gives an explicit expression for this inner product. Under this inner product, $\ad^*_\xi \xi=-\ad_\xi \xi=\lb{\xi}{\xi}=0$ because of the skew-symmetricity of the Lie-bracket. As a result, we can choose this inner product to ensure $\ad^*_\xi \xi$ vanishes for any $\xi\in \g$. This choice will be used throughout the rest of this paper. The diameter of $\G$ under such metric will be denoted as $D$.

Under the left-invariant metric induced by this inner product on $\mathsf{g}$, the Lie group always has non-negative sectional curvature (Sec. \ref{sec_Riemannian_structure}). Interestingly, different opinions exist about whether positive or negative curvatures help optimization or sampling. Positive curvatures help the quantification of discretization error via a modified cosine rule \citep{alimisis2020continuous}, while negative curvatures make it easier to have geodesically convex potential.
\vspace{-0.2cm}

\subsection{Assumptions on the potential function}
\vspace{-0.2cm}

\label{sec_assumptions_potential}
After selecting our inner product on $\g$ and using it to induce a left-invariant metric on $\G$, the Riemannian structure on $\G$ enables us to define distance and gradient on $\G$. With these, we can make the following smoothness assumption on the potential function $U$:
\begin{assumption}[$L$-smoothness]
\label{assumption_L_smooth}
Under the left-invariant metric induced by the inner product in Lemma \ref{lemma_ad_self_adjoint}, there exist constants $L\in (0,\infty )$, s.t.
    \begin{align*}
    \norm{\lt{g}\nabla U(g)-\lt{\hat{g}}\nabla U(\hat{g})}&\le Ld(g,\hat{g}) \quad\forall g, \hat{g}\in \G
    \end{align*} 
\end{assumption}
Although here we compare gradients using left-trivialization, Lemma \ref{lemma_L_smooth_equivalence} shows the $L$-smoothness defined here is the same as the commonly used geodesic $L$-smoothness based on parallel transport (Def. \ref{def_geodesic_L_smooth}), given the left-invariant metric in Lemma \ref{lemma_ad_self_adjoint}.

\vspace{-0.4cm}

\section{The SDE for sampling dynamics on Lie groups}
\vspace{-0.2cm}
\label{sec_sampling_dynamics}
In Euclidean space, when noise is added to momentum gradient descent, we have kinetic Langevin sampling SDE. Thanks to that our $\xi$ is in a flat space, to obtain a Lie group generalization we add noise analogously to the optimization ODE Eq.\eqref{eqn_optimization_ODE} to obtain the following SDE on $\G\times\g$ ,whose invariant distribution is the Gibbs distribution (Thm. \ref{thm_invariant_distribution_Gibbs}):
\begin{equation}
\label{eqn_sampling_SDE}
    \begin{cases}
        \dot{g}=\lte{g} \xi\\
        \d \xi=-\gamma\xi \d t+\ad_\xi^*\xi \d t
        -\lt{g}(\nabla U(g))\d t+\sqrt{2\gamma}\d W_t
    \end{cases}
\end{equation}
Here $\d W_t $ means $\sum_i e_i \d W_t^i$, with $W^i$ being i.i.d. Brownian motion and $\{e_i\}_{i=1}^m$ being a set of orthonormal basis of $\g$ under $\ip{\cdot}{\cdot}$. The Brownian motion in both cases is not on a manifold but simply in a finite-dimensional vector space with an inner product, thus is well-defined.

The term $\ad^*_\xi \xi$ is again due to the curved space, see Sec.\ref{sec_symplectic_structure}. It ensures the invariant distribution \eqref{eqn_Gibbs_TG} for any choice of inner product. See Thm. \ref{thm_invariant_distribution_Gibbs}. However, if we use the inner product in Lemma \ref{lemma_ad_self_adjoint}, this term vanishes and the sampling SDE reduces to Eq.\ref{eqn_sampling_SDE_ad_star_vanish}.

\vspace{-0.1cm}
\section{Convergence of sampling dynamics in continuous time}
\label{sec_convergence_SDE}
\vspace{-0.1cm}
We now prove the convergence of sampling SDE \eqref{eqn_sampling_SDE} under Wasserstein-2 distance. To do so, we first construct a coupling scheme by mixing synchronous coupling with reflection coupling. We also design a semi-distance. Using semi-martingale decomposition, we then prove the contractivity under this semi-metric, which induces convergence to the invariant distribution. Finally, this will be used to obtain convergence in $W_2$. The coupling scheme is inspired by \cite{eberle2019couplings}, but both our semi-distance and detailed coupling scheme are different, which are specially designed to handle the curved space and to better utilize the compactness of the Lie group.
\vspace{-0.1cm}
\subsection{Construction of coupling}
\label{sec_coupling}
Coupling is a powerful probabilistic technique for, e.g., studying the convergence of a diffusion process. An easy example is synchronous coupling, i.e. consider $\d x=b(x)\d t+\d W_t$ and $\d y=b(y)\d t+\d W_t$ with the same noise and $y$ initialized at the invariant distribution; when the drift term $b$ provides contractivity, one can show $x$ converges to $y$ and hence the invariant distribution. For Euclidean Langevin SDE, it works, for example, when $b=-\nabla U$ with a strongly convex potential $U$. However, when contractivity from the drift term is not enough or noise is more complicated, we need more advanced coupling techniques:

Specifically, we consider $(g_t, \xi_t)$ and $(\hat{g}_t, \hat{\xi}_t)$, both evolving in law as 
\eqref{eqn_sampling_SDE}, given by:
\begin{flalign}
    \begin{cases}
        \d g_t=\lte{g_t} \xi_t \d t\\
        \d\xi_t=-\gamma\xi_t \d t-\lt{g_t}(\nabla U(g_t))\d t+\sqrt{2\gamma}\rcfunc(Z_t,\mu_t) \d W_t^{\rcfunc} +\sqrt{2\gamma}\scfunc(Z_t,\mu_t) \d W_t^{\scfunc}
    \end{cases}
    \label{eq:SDEsCoupled} &&
\end{flalign}
\vspace{-0.25cm}
\begin{flalign*}
    \begin{cases}
        \d\hat{g}_t=\lte{\hat{g}_t} \hat{\xi}_t \d t\\
        \d\hat{\xi}_t=-\gamma\hat{\xi}_t \d t-\lt{\hat{g}_t}(\nabla U(\hat{g}_t))\d t+\sqrt{2\gamma}(I_d-2e_te_t^\top)\rcfunc(Z_t,\mu_t) \d W_t^{\rcfunc} +\sqrt{2\gamma}\scfunc(Z_t,\mu_t) \d W_t^{\scfunc}
    \end{cases}
\end{flalign*}
where $Z_t:=\log \hat{g}_t^{-1} g_t$, $\mu_t:=\xi_t-\hat{\xi}_t$ and $Q_t:=Z_t+\frac{1}{\gamma}\mu_t$. $e_t= Q_t/\norm{Q_t}$ if $Q_t\neq 0$, otherwise $e_t=0$.

In all cases, we use $\log$ to define the inverse of the exponential map corresponding to the minimum geodesic (any one will do when the minimum geodesic is not unique). Moreover, $\rcfunc, \scfunc:\mathbb{R}^{2m}\to [0,1]$ are Lipschitz continuous functions such that $\rcfunc^2+\scfunc^2\equiv 1$, and

\begin{equation}
\label{eqn_rc}
\rcfunc(z,\mu) =
\begin{cases}
0,\,\mbox{if }z+  \gamma^{-1}\mu=0\,\,\mbox{or}\,\, \norm{\mu}\ge \gamma R+\epsilon ,\\
1,\,\mbox{if }\norm{z+ \gamma^{-1}\mu}\ge\epsilon\,\,\mbox{and}\,\, \norm{\mu}\le \gamma R
\end{cases}
\end{equation}

The parameters $R\in(0, \infty)$ will be chosen later in Sec. \ref{sec_convergence_rate}. $\epsilon$ is a fixed positive constant, and will go to 0 eventually (See the proof for Thm. \ref{thm_contractivity_SDE_Wrho}).

The intuition of our design of coupling is illustrated in Fig. \ref{fig_coupling_region}. The nonlinearly transformed phase space is partitioned into four regions where different couplings are used. More precisely -

\begin{wrapfigure}{r}{0.25\textwidth}
\captionsetup{justification=raggedright}
\centering
    \includegraphics[width=0.25\textwidth]{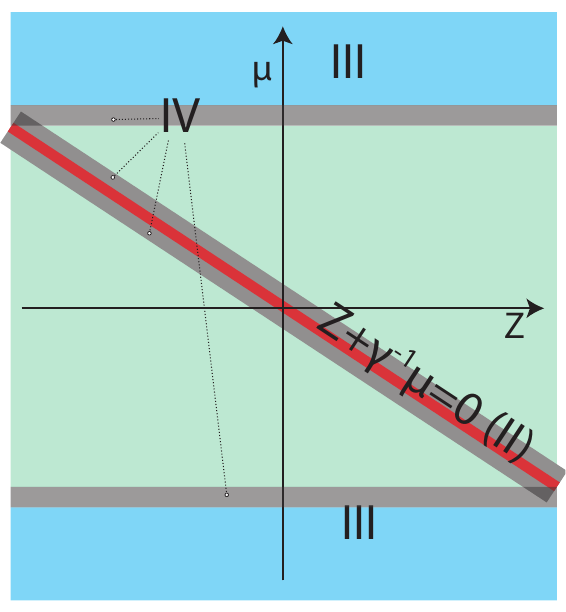}
    \caption{Partition of phase space using different couplings \vspace{-0.5cm}}
    \label{fig_coupling_region}
\end{wrapfigure}

Region (III) is where two momenta are very different. Their difference is so large that it is the main difference between $(g, \xi)$ and $(\hat{g}, \hat{\xi})$. In this case, friction $-\gamma\xi$ provides contractility and synchronous coupling is sufficient for bringing the two copies together.

Region (II) is the `good' part, which means the two points $(g, \xi)$ and $(\hat{g}, \hat{\xi})$ will eventually stop at the same point under friction if the potential does nothing. Here, the momenta of the two copies are pulling their position variables together, thus providing contractivity. Synchronous coupling is enough for this region too.

Region (I) is the tricky part. None of the two ways above works, and we use reflection coupling instead to obtain contractivity. An intuition is the following: when reflection coupling is used, the difference between trivialized momenta is no longer differentiable in time but a diffusion process. Combined with the concavity of $f$ in the semi-distance, which will be introduce in Sec. \ref{sec_semi_distance}, a negative Ito correction term shows up in the adapted finite-variation process of a semi-martingale decomposition of the semi-distance, which provides an extra negative drift and leads to contractivity if the parameters are chosen carefully. More specifically, the helping term we are referring to is $4\gamma^{-1} \rcfunc(Z,\mu)^2 f''(r)$ in $K$ (defined in Sec. \ref{sec_semi_martingale_decomposition}) in the semi-martingale decomposition of $e^{ct}\rho_t$ (Lemma \ref{lemma_martingale_decomposition}).

Region (IV) is the transitional region between synchronous coupling and reflection coupling, colored in gray in Fig. \ref{fig_coupling_region}. Its has width $\epsilon$ and will eventually vanish when $\epsilon\to 0$.

\vspace{-0.3cm}
\subsection{Design of semi-distance}
\label{sec_semi_distance}
Besides designing the coupling, we also need a quantification of how far two points are on $\G\times \g$. One of our innovations is, we do not necessarily need a distance but only a semi-distance, i.e., no triangle inequality. We will construct the semi-distance carefully so that the coupling leads to contractivity. More precisely, let the semi-distance $\rho$ be given by
\begin{align}
\label{eqn_rho}
    \rho((g,\xi),(\hat{g} ,\hat{\xi})):= f(r((g,\xi),(\hat{g} ,\hat{\xi})))G(\xi, \hat{\xi}), \quad \forall g, \hat{g}\in \G, \forall \xi, \hat{\xi} \in \g
\end{align}
where $f:[0, \infty)\to [0, \infty)$ is a continuous, non-decreasing, concave function satisfing $f(0)=0$, $f_+'(0)=1$, $f_-'(R)>0$ \footnote{Here $f_-$ and $f_+$ denote the left and right derivative of $f$, respectively.} and $f$ is constant on $[R, \infty)$. $r$ and $G$ are defined as
\begin{align}
r((g,\xi),(\hat{g} ,\hat{\xi}))&:=\alpha d(g, \hat{g})+\norm{\log \hat{g}^{-1}g+\gamma^{-1} (\xi-\hat{\xi})} \label{eqn_r}\\
G(\xi, \hat{\xi})&:=1+\beta\norm{\hat{\xi}-\xi}^2 \label{eqn_G}
\end{align}
An explicit expression of $f$ will be chosen later in Sec. \ref{sec_choose_parameter_f}. The parameters $\alpha, \beta\in(0, \infty)$ will be given in Sec. \ref{sec_choose_parameter_theta_beta_R}, whose values are chosen carefully together with other parameters ($f$ and $R$) to achieve contractivity.

As shown, the semi-distance $\rho$ has a complicated form, and the triangle inequality is sacrificed, which may lead to difficulties when analyzing numerical error from discretization later. But before diving into more details, let's provide some motivation for this complicated design.  and how this design of semi-distance is related to our coupling.

$\rho$ is the product of two parts, $f(r)$ and $G$. The function $G$ is one plus the square of Euclidean distance between $\xi$ and $\hat{\xi}$. It is designed for handling region (III), where friction ensures the decrease of $G$. Unlike \cite{eberle2019couplings}, our $G$ does not depend on position ($g$ or $\hat{g}$). Besides the intuition we mentioned about friction, other reasons are: 1) In \cite{eberle2019couplings}, $G$ also provide contractivity when the positions $g$ and $\hat{g}$ are far from each other utilizing strong dispativity under synchronous coupling. However, as discussed in Sec. \ref{sec_compare_conditions}, in our case, we can not have convexity assumptions on the potential and including $g$ and $\hat{g}$ in $G$ does not help; 2) compactness ensures the distance between $g$ and $\hat{g}$ is bounded, and in this case having $G$ that only depends on the difference of momentum is enough to ensure $\rho$ is stronger than $d^2$ (Lemma \ref{lemma_equivalence_distance}).

The other part of $\rho$ is $f(r)$. It is by design concave because, like mentioned earlier, Ito's correction will give an additional term based on $f''$, and concavity of $f$  will provide a negative sign and thus contractivity in region (I). Moreover, we make it first increasing and then constant starting from $r=R$ on. This way, $\min \{r\in\mathbb{R}: f(r)=\sup f\} = R$. Meanwhile, the bound for the synchronous coupling (III) is defined by $\norm{\xi-\hat{\xi}}\ge \gamma R$. It is not a coincidence that they both scale with $R$; instead, it is because the coupling and semi-metric are carefully designed together: As mentioned earlier in Sec. \ref{sec_coupling}, in region (III), the friction term is large enough and synchronous coupling suffices, i.e., when $\mu\ge \gamma R$, we have $r\ge R$. Therefore, the concavity of $f$ in the interval of $r<R$ is sufficient for creating contractivity, and thus we simply set $f$ to be constant on $[R, \infty)$.

Intuitively, $r$ in Eq. \eqref{eqn_r} measures the `distance' between points on $\G\times\g$. The first parts $d(g, \hat{g})$ is directly the distance on the Lie group. However, we do not measure the difference between momenta directly, but by $\norm{\log \hat{g}^{-1}g+\gamma^{-1}(\xi-\hat{\xi})}$, i.e. a twisted version with position distance also leaked in. The reason is similar to why we use region (I): we are measuring how far they will travel before eventually stopping due to friction without potential; this is a manifold generalization of the existing Euclidean treatment \citep[e.g.,][]{dalalyan2020sampling,eberle2019couplings}.

As a result, the design of $\rho$ is a combination of functions that ensures contractivity in different regions. Later, we will carefully select parameters $\alpha$ and $\beta$ to perfectly balance them. We will also discuss how to remedy the loss of triangle inequality and provide a substitute formula in Lemma \ref{lemma_triangular_ineq_rho} when we consider the discretization error of our SDE.

For now, we first show $\rho$ can control the standard geodesic distance. The standard geodesic distance $d$ in the product space $\G\times \g$ is defined by
\begin{align}
\label{eqn_d_TG}
    d((g, \xi), (\hat{g}, \hat{\xi}))=\sqrt{d^2(g, \hat{g})+\norm{\hat{\xi}-\xi}^2}
\end{align}
where $d$ on the right-hand side is the distance on $\G$ given by the minimum geodesic length. Since both the distance on $\G$ and the distance on $\G\times \g$ are derived from the inner product on $\g$, we will use the same notation when there is no confusion. We have the following theorem showing $d^2$ on $\G\times \g$ is controlled by $\rho$.
\begin{lemma}[Control of $d$ by $\rho$]
\label{lemma_equivalence_distance}
Define $C_\rho$ as Eq. \eqref{eqn_C_rho}, we have     \begin{align*}
        C_\rho d^2((g, \xi), (\hat{g}, \hat{\xi}))\le \rho((g, \xi), (\hat{g}, \hat{\xi})), \qquad \qquad
        \forall (g, \xi), (\hat{g}, \hat{\xi}) \in \G\times \g
    \end{align*}
\end{lemma}

Here are more intuitions about $\rho$. For simplified notation, use $d$ to denote $d((g, \xi), (\hat{g}, \hat{\xi}))$. When $d$ is small, we have $r\sim d$ and $d^2\lesssim d$. Moreover, we have $f_{-}'$ lower bounded from 0, and $\rho\equiv f(r) \sim d$. However, when $d$ is large, we have $r$ is also large and $f(r)=f(R)$. In this case, $\rho\sim f(R) d^2$. To summarize, $\rho$ is similar to $d$ when $d$ is small but similar to $d^2$ when $d$ is large. This is why $d^2$ but not $d$ can be bounded by $\rho$ in Lemma \ref{lemma_equivalence_distance}.

\begin{remark}[Comparison with \cite{eberle2019couplings}]
    The coupling and semi-distance used by \cite{eberle2019couplings} inspired our choices, but their version is not suitable for us. The first reason is that their $G$ (Eq. 3.10 in \cite{eberle2019couplings}) is specially designed for the condition `convex outside a ball', which is not available for Lie groups (Sec. \ref{sec_compare_conditions}). Also, technical issues appear for their semi-distance: $\d\log$ (differential of logarithm) does not exist on a zero-measured set $N$ (Eq.\ref{eqn_N}) on Lie groups, and on a small neighbour of $N$, the operator norm of $\d\log$ can be unbounded, which leads to difficulties in calculation when using their sem-distance. However, our design of coupling and semi-distance utilizes the boundness of our Lie group. Consequently,  in all our proof, what we only need for $\d\log$ is the properties in Cor. \ref{cor_log_property} and the fact that $N$ is zero-measured is enough for our approach.
\end{remark}

\subsection{Contractivity of sampling dynamics under $W_\rho$ distance}
\label{sec_convergence_rate}
We define the Wasserstein semi-distance between distributions on $\G\times \g$ as $W_\rho$, i.e., $W_\rho(\nu_1, \nu_2):=\inf_{\pi\in\Pi(\nu_1, \nu_2)}\int \rho d\pi$ where $\Pi$ is the set of all distributions on $(\G\times \g)^2$ with marginal distributions $\nu_1$ and $\nu_2$. We hope to prove the contractivity of densities under $W_\rho$ distance. The proof uses the following idea:

For any pair of points $(g, \xi)$ and $(\hat{g}, \hat{\xi})$ that are coupled together in the way stated in Sec. \ref{sec_coupling}, we now construct a martingale decomposition bound of $e^{ct}\rho((g, \xi), (\hat{g}, \hat{\xi}))$, where $c$ is our target contraction rate that will be chosen later. More precisely, as will be shown by Lemma \ref{lemma_martingale_decomposition}, $e^{ct}\rho((g, \xi), (\hat{g}, \hat{\xi}))$ is decomposed as the sum of an adapted finite-variation process  and a continuous local martingale, where the former is upper bounded by $\int_0^t e^{cs} K((g_s, \xi_s), (\hat{g}_s, \hat{\xi}_s)) \d s$ with $K:(\G\times\g)^2\to \mathbb{R}$ defined later in Sec. \ref{sec_semi_martingale_decomposition}.
\begin{lemma}\label{lemma_martingale_decomposition}
    Let $c,\epsilon \in (0,\infty )$, and suppose that $f:[0,\infty )\to [0,\infty )$
    is continuous, non-decreasing, concave, and $C^2$ except for
    finitely many points. Then
    \begin{equation}
    \label{eqn_decomposition_rho}
    e^{ct}\rho_t\le\rho_0 + \int_0^te^{cs}K_s ds + M_t\quad \forall t\ge 0\quad a.e.,
    \end{equation}
    where $(M_t)$ is a continuous local martingale, and $K_s:=K((g_s, \xi_s), (\hat{g}_s, \hat{\xi}_s))$
\end{lemma}
By taking expectation, we only need $K((g_s, \xi_s), (\hat{g}_s, \hat{\xi}_s))\le 0, a.e., \forall s$ when $\epsilon\to 0$, where $(g_s, \xi_s)$ and $(\hat{g}_s, \hat{\xi}_s)$ are coupled as in Sec. \ref{sec_coupling}. This will infer that $\mathbb{E}e^{ct}\rho((g_t, \xi_t), (\hat{g}_t, \hat{\xi}_t))$ is non-increasing and further gives us the convergence rate under $W_\rho$ distance. We summarize the conditions needed for $K((g_s, \xi_s), f(\hat{g}_s, \hat{\xi}_s))\le 0, a.e.$ in Sec. \ref{sec_conditions_for_contractivity} and later choose all the parameters $f$, $\alpha$, $\beta$, $R$ and $c$ in Sec. \ref{sec_choose_parameters} such that these conditions are met. The choice of $\gamma$ is also given to establish an explicit order of convergence rate.
This gives our contractivity result for sampling SDE Eq. \eqref{eqn_sampling_SDE} under semi-distance $\rho$:
\begin{theorem}
\label{thm_contractivity_SDE_Wrho}
For any probability distributions $\nu_0$ and $\hat{\nu}_0$ on $\G\times \g$ that is absolute continuous w.r.t. $\d g \d \xi$, if we evolve them by the sampling SDE \eqref{eqn_sampling_SDE}, then, for some $\alpha$, $\beta$, $R$, $f$, and $c>0$ (specified in Thm. \ref{lemma_parameters}), we have
\begin{equation*}
W_\rho (\nu_t, \hat{\nu}_t)\le e^{-ct} W_\rho (\nu_0, \hat{\nu}_0)
\end{equation*}
\end{theorem}

As we mentioned, $W_\rho$ is only a semi-distance and lacks triangle inequality since $\rho$ is only a semi-distance. However, what we only need for now is it controls $W_2$ distance:

\subsection{Error bound for sampling SDE under $W_2$ distance}

Upon choosing $\hat{\nu}_0$ in Thm.\ref{thm_contractivity_SDE_Wrho} as the invariant distribution, obviously $\hat{\nu}_t = \hat{\nu}_0$ and Thm.\ref{thm_contractivity_SDE_Wrho} thus quantifies the convergence speed of the sampling dynamics in $W_\rho$. Since $W_\rho$ is what we invented and not a distance, we control $d^2$ by $\rho$, which infers $W_\rho$ controls $W_2^2$ (Cor. \ref{cor_equivalence_W2}),  so that we can have the following theorem for convergence in a more standard distance:
\begin{theorem}[Error of sampling SDE under $W_2$]
\label{thm_SDE_error_W2}
Suppose the initial condition $(g_0, \xi_0)\sim \nu_0$, where we assume $\nu_0$ is absolute continuous w.r.t. $\d g\d \xi$. Denote by $\nu_*$ the Gibbs distribution Eq. \eqref{eqn_Gibbs_TG} and by $\nu_t$ the distribution evolved by SDE Eq. \eqref{eqn_sampling_SDE}. Then we have the $W_2$ distance between $\nu_t$ and $\nu_*$ is bounded by
\begin{align*}
    W_2(\nu_t, \nu_*)\le e^{-\frac{c}{2}t}\sqrt{C_\rho W_\rho(\nu_0, \nu_*)}
\end{align*}
\end{theorem}
The reason why we need absolute continuity of the initial condition is because this gives us absolute continuity at any time $t$, which further enables us to ignore a bad set $N$ where $\log$ is not differentiable. However, this condition will not lead to infeasibility of our discrete algorithm, which will be discussed later in Rmk. \ref{rmk_absolute_continuity_initialization}.

The contraction rate $c$ is not explicitly expressed in the above theorem only because it is lengthy. A detailed expression is given by Eq. \eqref{eqn_c_upper_bound} in Lemma \ref{lemma_parameters} and Eq. \eqref{eqn_c_order} gives an estimation of the order of $c$. How these results help choose $\gamma$ is also discussed in Sec. \ref{sec_discussion_c}.

\section{Convergence of sampling algorithm in discrete time}
\label{sec_convergence_splitting}
This section will first construct a sampler based on a time discretization of our sampling SDE. Thanks to an operator splitting technique, this discretization will render the iterations exactly satisfying the geometry of the curved space, which is a pleasant property referred to as structure-preservation (Thm. \ref{thm_structure_preserving}). Then we will establish its error bound in $W_2$ by: 1) quantifying local integration error in $d$ (Thm. \ref{thm_local_error_W2}); 2) developing a modified triangle inequality for our semi-distance $\rho$ (Lemma \ref{lemma_triangular_ineq_rho}) and estimating sampling error propagation for the discrete sampler under $\mathbb{E}\rho$ (Thm. \ref{thm_error_propagation}); 3) quantifying how local error accumulates to establish a global nonasymptotic error estimate under $\mathbb{E}\rho$ (Cor. \ref{cor_error_Erho}), and then turning it into a nonasymptotic sampling error bound under $W_\rho$ (Thm. \ref{thm_global_error_Wrho}); 4) bounding sampling error in $W_2$ (Thm. \ref{thm_global_error_W2}) by the fact that $W_2^2$ can be controlled by $W_\rho$ (Cor. \ref{cor_equivalence_W2}).

\subsection{Sampler based on splitting discretization}
\label{sec_splitting}
We consider sampling SDE \eqref{eqn_sampling_SDE} with vanishing $\ad^*_\xi \xi$ (due to Lemma \ref{lemma_ad_self_adjoint}). To obtain a time discretization that respects the geometry, we write \eqref{eqn_sampling_SDE} as the sum of the following two SDEs, each of which can be solved explicitly, and alternatively evolve them to approximate the solution of \eqref{eqn_sampling_SDE}, whose closed-form solution does not exist.
\vspace{-0.5cm}
\begin{center}
\begin{minipage}[t]{.65\linewidth}
\begin{equation}
\label{eqn_splitting_SDE_xi}
    \begin{cases}
        \dot{g}=0\\
        \d\xi=-\gamma\xi \d t-\lt{g}(\nabla U(g))\d t+\sqrt{2\gamma}\d W_t
    \end{cases}
\end{equation}
\end{minipage}
\hfill
\begin{minipage}[t]{.3\linewidth}
\begin{equation}
\label{eqn_splitting_SDE_g}
    \begin{cases}
        \dot{g}=\lte{g} \xi\\
        \d\xi=0
    \end{cases}
\end{equation}
\end{minipage}
\end{center}
To implement our algorithm, we embed the Lie group in an ambient Euclidean space, and the algorithm automatically keeps the iterates on the Lie group. Most algorithms in curved spaces similarly rely on this embedding, but they need extra work to correct the deviation from manifold after each iteration, e.g., \cite{cisse2017parseval}. In contrast, thanks to our specific splitting discretization, no matter how large the step size $h$ is, our algorithm ensures $g$ stays on the Lie group, and no artificial step that pulls the point back to the curved space is needed. Specifically, by first evolving Eq. \eqref{eqn_splitting_SDE_xi} for time $h$ and then Eq. \eqref{eqn_splitting_SDE_g} for time $h$, we obtain the following one step update:
\begin{equation}
\label{eqn_discrete_splitting}
    \begin{cases}
        \tilde{g}_h=g_0\exp(h\tilde{\xi}_h)\\
        \tilde{\xi}_h=\exp(-\gamma h)\xi_0-\frac{1-\exp(-\gamma h)}{\gamma} \lt{g}\nabla U(g_0)+\sqrt{2\gamma}\int_0^h \exp(-\gamma (h-s)) \d W_s
    \end{cases}
\end{equation}
Iterating this update gives Algorithm \ref{algo_sampling_Lie_group}.

Its general structure preservation is summarized below; an example is given in Rmk.\ref{rmk_example_structure_preserving}.

\begin{theorem}
\label{thm_structure_preserving}
    The splitting discretization Eq. \eqref{eqn_discrete_splitting} is structure-preserving, i.e., for any step size $h$ and any initial point $(g_0, \xi_0)$, the iteration $(g_k, \xi_k), \forall k\ge 0$ has the property that $g_k$ stays exactly on the Lie group and $\xi_k$ stays exactly on the Lie algebra.
\end{theorem}
\vspace{-0.1cm}

We remark that the commonly used exponential integrator (see Eq.\ref{eqn_exp_integrator_Euclidean}) doesn't work here due to the nonlinear geometry. For example, the matrix-group-embedded $g$ dynamics is $\dot{g}=g\xi$ where both $g$ and $\xi$ are matrices; because $\xi$ is time-dependent, $g$ admits no analytical solution. Existing tools for analyzing exponential-integrator-based samplers unfortunately have technical difficulties to be generalized here too (see Sec. \ref{sec_more_about_splitting_discretization}).

\vspace{-0.2cm}
\subsection{Quantification of discretization error}
\vspace{-0.1cm}
\underline{Notation}: \emph{starting from initial conditions $(g_0, \xi_0)$ and $(\hat{g}_0, \hat{\xi}_0)$, we use $(g_t, \xi_t)$ and $(\hat{g}_t, \hat{\xi}_t)$ to denote the exact solutions of two SDEs that are coupled together \eqref{eq:SDEsCoupled} correspondingly. $(\tilde{g}_h, \tilde{\xi}_h)$ is for one iteration by our sampler (i.e. splitting discretization \eqref{eqn_discrete_splitting}) from initial condition $(g_0, \xi_0)$ with fixed step size $h$. $(\tilde{g}_k, \tilde{\xi}_k)$ is the result of $k$ iterations.}

It is challenging to quantify discretization error directly in $\rho$, because $\rho$ has a complicated expression and lacks triangle inequality. We bypass the difficulty by using the natural distance $d$ on $\G\times \g$ instead. 
The following theorem will quantify the one-step mean square error of our numerical scheme. The proof uses a new technique, namely to view the numerical solution after just one-step, with step size $h$, as the exact solution of some shadowing SDE at time $h$ (Lemma \ref{lemma_SDE_splitting}), and then we only need to quantify the difference between the sampling SDE \eqref{eqn_sampling_SDE} and the shadow SDE \eqref{eqn_SDE_splitting}.
\vspace{-0.1cm}
\begin{theorem}[Local numerical error]
\label{thm_local_error_W2}
    Suppose the initial condition  $(g_0, \xi_0)\sim \nu_0$ is absolute continuous w.r.t. $\d g\d \xi$. $(g_t, \xi_t)$ is the solution of SDE \eqref{eqn_sampling_SDE} at time $h$ and $(\tilde{g}_h, \tilde{\xi}_h)$ is the solution of the numerical discretization Eq. \ref{eqn_discrete_splitting} with time step $h$. Then we have
    \begin{align*}
    \mathbb{E}d^2((g_h, \xi_h), (\tilde{g}_h, \tilde{\xi}_h))\le x(h)+y(h)
    \end{align*}
    where $x(t)$ and $y(t)$ are the solutions to an ODE given by Eq.\ref{eqn_xy_ode}.
\end{theorem}
If one wants a more explicit bound, the following Lemma gives a less tight but more succinct one: Basically the local strong error is $3/2=1.5$ order.

\begin{lemma}[Order of local error]
\label{lemma_order_W2_error}
    \begin{align*}
        \mathbb{E}d^2((g_h, \xi_h), (\tilde{g}_h, \tilde{\xi}_h))=\mathcal{O}(h^3), \quad h\to 0
    \end{align*}
    i.e., $\exists h_0, A_{h_0}, s.t. \forall h\le h_0, x(h)+y(h)\le A_{h_0} h^3$ for $x$ and $y$ defined by Eq. \ref{eqn_xy_ode}.
\end{lemma}
Then, in the next subsection, we will combine this local error in $d^2$ and the contractivity of the continuous dynamics to establish the contractivity of our sampler.

\subsection{Local error propagation}
A standard technique for analyzing the sampling error of a sampler based on the discretization of a continuous dynamics is to transfer its infinite long time numerical integration error to a Wasserstein control of the difference between the continuous and discrete dynamics. To do so, one typically first seeks the contractivity of the continuous dynamics (Thm. \ref{thm_contractivity_SDE_Wrho}), and then use that to both control the accumulation of local integration errors into global integration error, and bound the distance between the continuous dynamics and the target distribution so that the distance between the discrete dynamics and the target can also be bounded. For both tasks, triangle inequality is leveraged. However, as mentioned earlier in Sec. \ref{sec_semi_distance}, $\rho$ is only a semi-distance and does not satisfy triangle inequality, and we develop  the following lemma as an alternation.

\begin{lemma}[Modified triangle inequality for $\rho$]
\label{lemma_triangular_ineq_rho}
    For any points $(g, \xi)$, $(\hat{g}, \hat{\xi})$ and $(\tilde{g}, \tilde{\xi})$,
    \begin{align*}
        \rho((\hat{g}, \hat{\xi}), (\tilde{g}, \tilde{\xi}))\le \left[1+A_1d((g, \xi), (\tilde{g}, \tilde{\xi}))\right]\rho((\hat{g}, \hat{\xi}), (g, \xi))+A_2d^2((g, \xi), (\tilde{g}, \tilde{\xi}))
    \end{align*}
    for some constants $A_1, A_2>0$ (their detailed expressions are given in Eq. \ref{eqn_A}).
\end{lemma}

Combining contractivity of continuous SDE (Thm.\ref{thm_contractivity_SDE_Wrho}), local discretization error (Thm.\ref{thm_local_error_W2}) and modified triangle inequality for $\rho$ (Lemma \ref{lemma_triangular_ineq_rho}), we have the following result quantifying how error propagates after one step (more precisely, given two initial conditions, one evolving under the continuous dynamics for time $h$ and the other iterated by the discrete algorithm \eqref{eqn_discrete_splitting} for one $h$-step, how the difference of the results changes from the initial difference):

\begin{theorem}[Propagation of error under $\mathbb{E}_\rho$]
\label{thm_error_propagation}
Suppose $\Law(g_0, \xi_0)$ and $\Law(\hat{g}_0, \hat{\xi}_0)$ are absolute continuous w.r.t. $\d g\d \xi$. After one iteration from 
$(\tilde{g}_0, \tilde{\xi}_0)=(g_0, \xi_0)$ to $(\tilde{g}_h, \tilde{\xi}_h)$ by Eq.\eqref{eqn_discrete_splitting} with step size $h$, and exact evolution of $\hat{g}_0, \hat{\xi}_0$ by time $h$, we have, $\forall h>0$,
    \begin{align*}
        \mathbb{E}\rho((\hat{g}_h, \hat{\xi}_h), (\tilde{g}_h, \tilde{\xi}_h))\le\exp(-ch)\mathbb{E}\rho((\hat{g}_0, \hat{\xi}_0), (g_0, \xi_0))+E(h)
    \end{align*}
    where $E(h)$ is given by
    \begin{equation}
    \label{eqn_E}
        E(h):=\left(\frac{A_1}{2h^\frac{3}{2}}+A_2\right)\left(x(h)+y(h)\right)+\frac{A_1}{2}f(R)^2(1+8C_2+16C_4)h^\frac{3}{2}
    \end{equation}
    using the notation $x(h)$ and $y(h)$ from Lemma \ref{thm_local_error_W2}.
\end{theorem}

\subsection{Mixing in $W_\rho$}
By applying this local error propagation recurrently, we can obtain a global result:
\begin{corollary}[Nonasymptotic error bound under $\mathbb{E}_\rho$]
\label{cor_error_Erho}
Under the same condition as Thm. \ref{thm_error_propagation}, we have, for $k=0,1,\cdots$,
    \begin{align*}
        \mathbb{E}\rho((\hat{g}_{kh}, \hat{\xi}_{kh}), (\tilde{g}_k, \tilde{\xi}_k))\le e^{-ckh} \mathbb{E}\rho((\hat{g}_0, \hat{\xi}_0), (\tilde{g}_0, \tilde{\xi}_0))+\frac{E(h)}{1-\exp(-ch)}
    \end{align*}
\end{corollary}

The two semi-distances $W_\rho$ and $\mathbb{E}\rho$ are closely related, and we can derive the following error bound in $W_\rho$, for measuring the difference between distributions.
\begin{theorem}[Nonasymptotic error bound under $W_\rho$]
\label{thm_global_error_Wrho}
    Given Assumption \ref{assumption_general}, \ref{assumption_L_smooth} and the inner product on $\mathfrak{g}$ chosen as Lemma \ref{lemma_ad_self_adjoint}, if the initial condition $(g_0, \xi_0)\sim \nu_0$ satisfies $W_\rho(\nu_0, \nu_*)< \infty$ and $\nu_0$ is absolute w.r.t. $\d g\d \xi$, then $\forall k=1,2,\dots$, the density of scheme Eq. \eqref{eqn_discrete_splitting} starting from $\nu_0$ has the following $W_2$ distance from the target distribution:
    \begin{align*}
        W_\rho(\tilde{\nu}_k, \nu_*)\le e^{-ckh} W_\rho(\nu_0, \nu_*)+\frac{E(h)}{1-\exp(-ch)},
    \end{align*}
    where $E(h)$ is defined in Eq.\ref{eqn_E}. Note this holds $\forall h>0$, but $E(h)$ can grow exponentially.
\end{theorem}

\subsection{Global sampling error in $W_2$}
Via the global $W_\rho$ error bound (Thm.\ref{thm_global_error_Wrho}) and the property that $d^2$ is controlled by $\rho$ (Lemma \ref{lemma_equivalence_distance}), we can also have a nonasymptotic control of the sampling error in a more common way of measurement, namely $W_2$, which is a true distance between distributions this time:
\begin{theorem}[Nonasymptotic error bound under $W_2$]
\label{thm_global_error_W2}
    Under the same assumption as Thm. \ref{thm_global_error_Wrho}, we have
    \begin{align*}
        W_2(\tilde{\nu}_k, \nu_*)\le C_\rho\left(e^{-ckh} W_\rho(\nu_0, \nu_*)+\frac{E(h)}{1-\exp(-ch)}\right)
    \end{align*}
\end{theorem}
The requirement that initial condition is absolute continuous w.r.t. $\d \g \d \xi$ is only for technical reasons for the proof but not needed in implementation. See Rmk. \ref{rmk_absolute_continuity_initialization} for details. 
\begin{remark}
    Notice that $E(h)$ in Eq. \eqref{eqn_E} is of order $h^{\frac{3}{2}}$ and $1-\exp(-ch)$ is of order $h$, which means the bias in the second term created by discretization error converges to 0 when step size is infinitely small, and the bias is asymptotically of order $h^{0.5}$.
\end{remark}

\section{Numerical demonstration}
\label{sec_numerical_experiment}
Consider an example of sampling on the Lie group $\mathsf{SO}(n)$, which we embed in the matrix space $\mathbb{R}^{n\times n}$, i.e., $\mathsf{SO}(n):=\left\{X\in \mathbb{R}^{n\times n}: X^\top X=I, \det(X)=1\right\}$ (Example \ref{example_SOn}). The inner product in Lemma \ref{lemma_ad_self_adjoint} is given by $\ip{A}{B}=\tr(A^\top B)$. Under this innner product, the left-trivialized gradient $\lt{g}\nabla U(g)$ can be calculated as follows: suppose $\mathsf{SO}(n)$ is represented by $n\times n$ matrices, and $U$ is a real valued function defined on a subset of $\mathbb{R}^{n\times n}$. Denote its Euclidean gradient $\nabla_E U$, which can be calculated either by backward propagation or by closed-form solution. In this case, the Riemannian gradient $\nabla U(X)$ is given by $\nabla_E U(X)-X \left(\nabla_E U(X)\right)^\top X$ and the left-trivialized gradient $\lt{g}\nabla U(g)$ is given by $X^{-1}\nabla U(X)=X^\top \nabla_E U(X)-\left(\nabla_E U(X)\right)^\top X$.

Now we numerically experiment on a specific potential $U(X)=-10X_{1,1}^2$. We let $n=10$, and the dimension of the group is thus 45. Note even though the potential might appear like a convex function, the target distribution is actually multimodal, and we recall again there is nonconstant geodesic convex function on this manifold (e.g., Sec.\ref{sec_compare_conditions}).

The parameters for our sampler Algo. \ref{algo_sampling_Lie_group} are $\gamma=1$ and $h=0.1$. All the Markov chains are initialized by a fixed random generated orthogonal matrix. We focus on the variable $X[0,0]$ (the top left element of the matrix).

Fig. \ref{fig_experiment_SO_a} compares the samples from our sampler to the ground truth samples generated by rejection sampling, showing our sampler is sampling the correct Gibbs distribution. Fig. \ref{fig_experiment_SO_b} is another evidence that our sampler is dynamically sampling from a multimodal distribution. Exponential convergence of our sampler can be observed from Fig. \ref{fig_experiment_SO_c}, as theoretically proved in Thm. \ref{thm_SDE_error_W2} and \ref{thm_global_error_W2} despite of the multimodality. Fig. \ref{fig_experiment_SO_d} and \ref{fig_experiment_SO_e} visualize how the density evolves as our discrete process mixes and equilibrates.

\begin{figure}[H]
\label{fig_experiment_SO}
    \centering
    \subfigure[Histogram of values collected along a single trajectory, compared to ground truth]{\label{fig_experiment_SO_a}\includegraphics[width=0.33\textwidth, height=4cm]{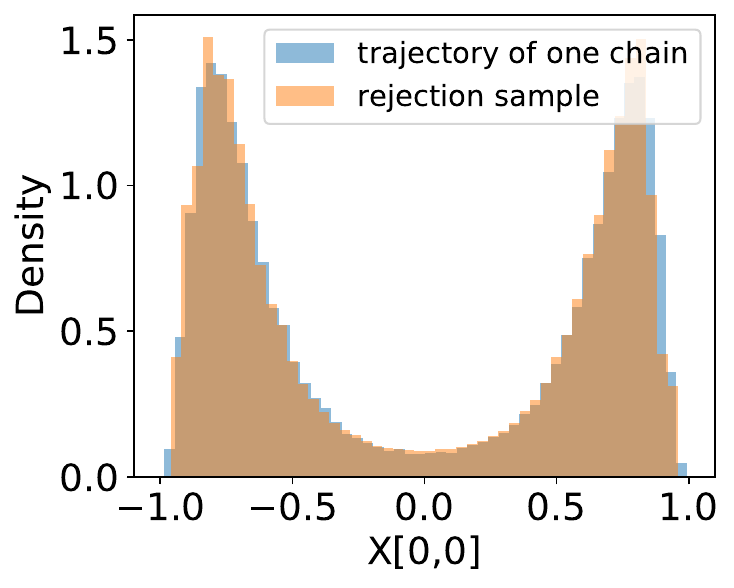}}
    \subfigure[Trajectory of one chain]{\label{fig_experiment_SO_b}\includegraphics[width=0.32\textwidth, height=4cm]{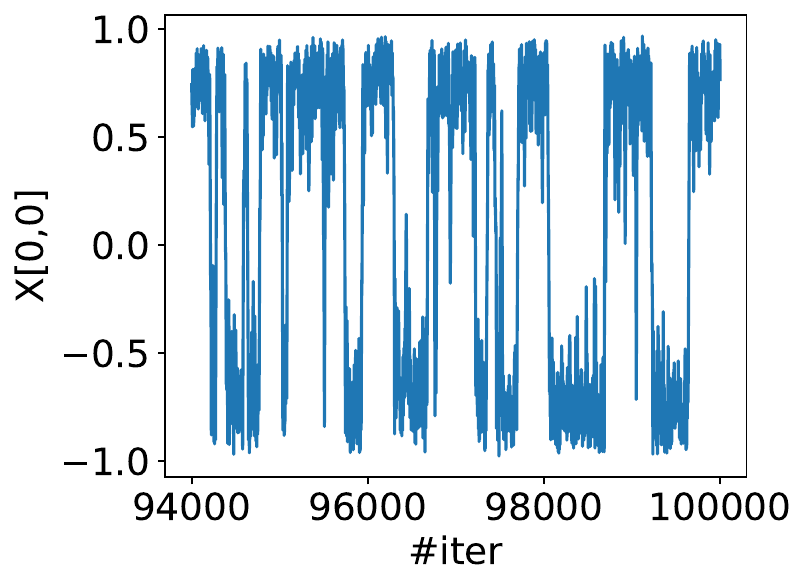}}
    \subfigure[Time evolution of the mean of an ensemble of chains, illustrating speed of convergence]{\label{fig_experiment_SO_c}\includegraphics[width=0.33\textwidth, height=4cm]{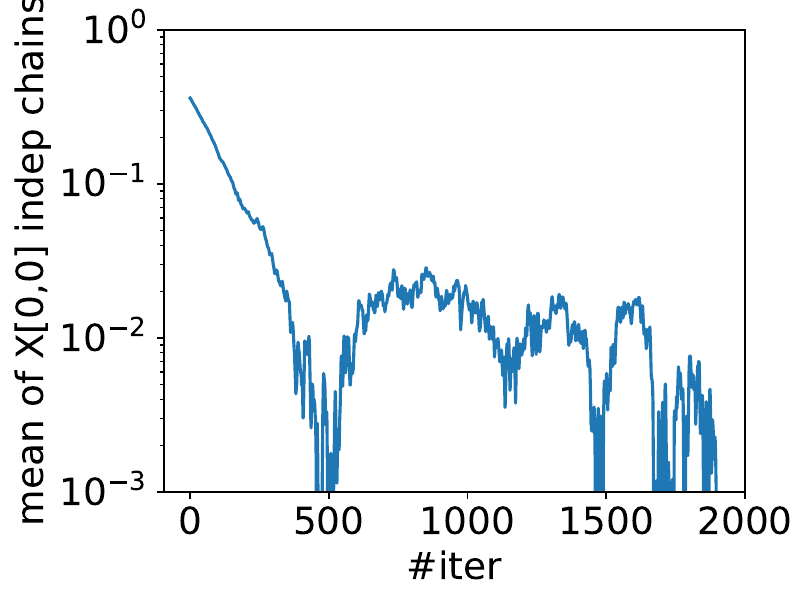}}
    \subfigure[Time evolution of the distribution of an ensemble of independent chains, prior to convergence (iteration \#0-100)
    ]{\label{fig_experiment_SO_d}\includegraphics[width=0.48\textwidth, height=6.5cm]{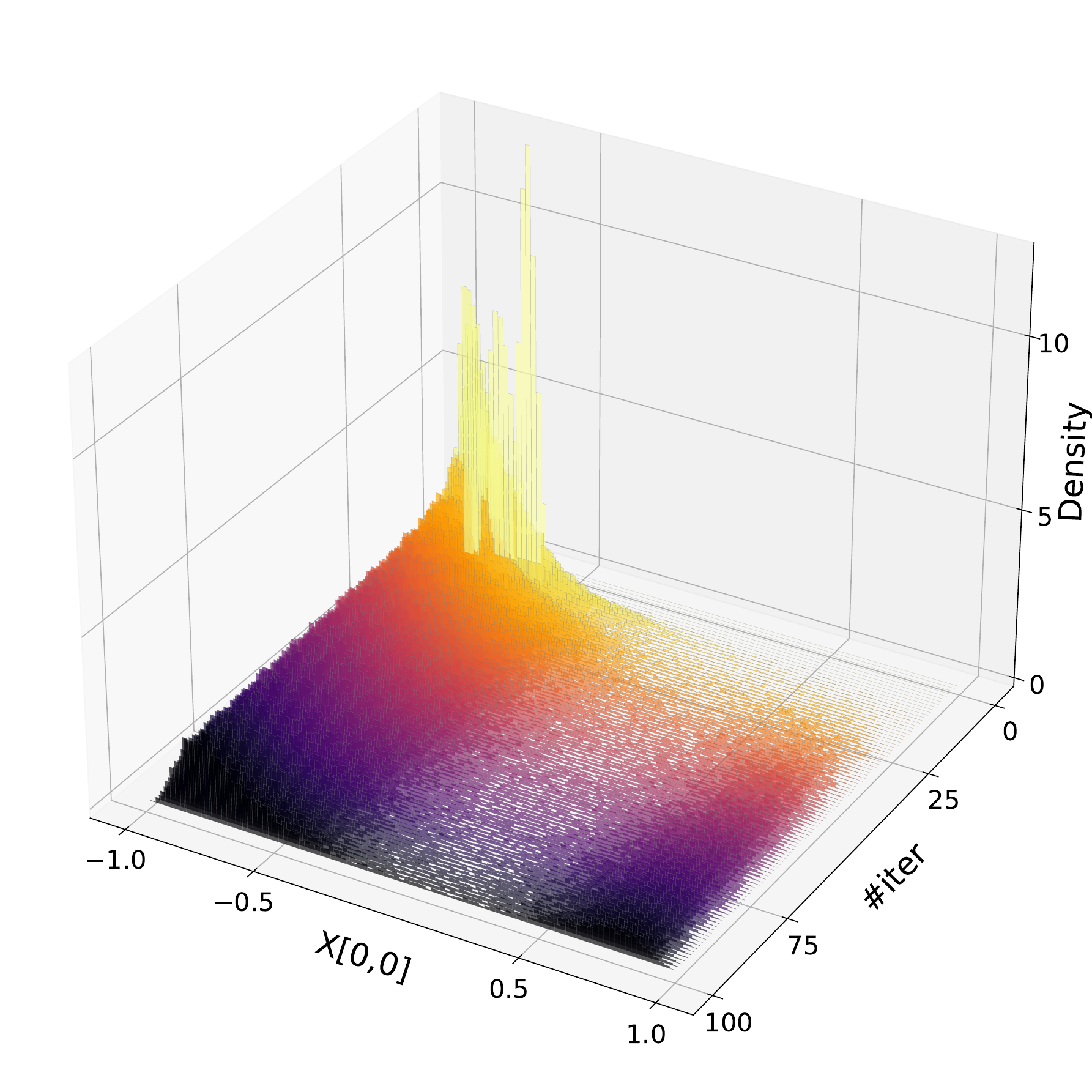}}
    \subfigure[Time evolution of the distribution of an ensemble of independent chains, closer to convergence (iteration \#100-2000)]{\label{fig_experiment_SO_e}\includegraphics[width=0.48\textwidth, height=6.5cm]{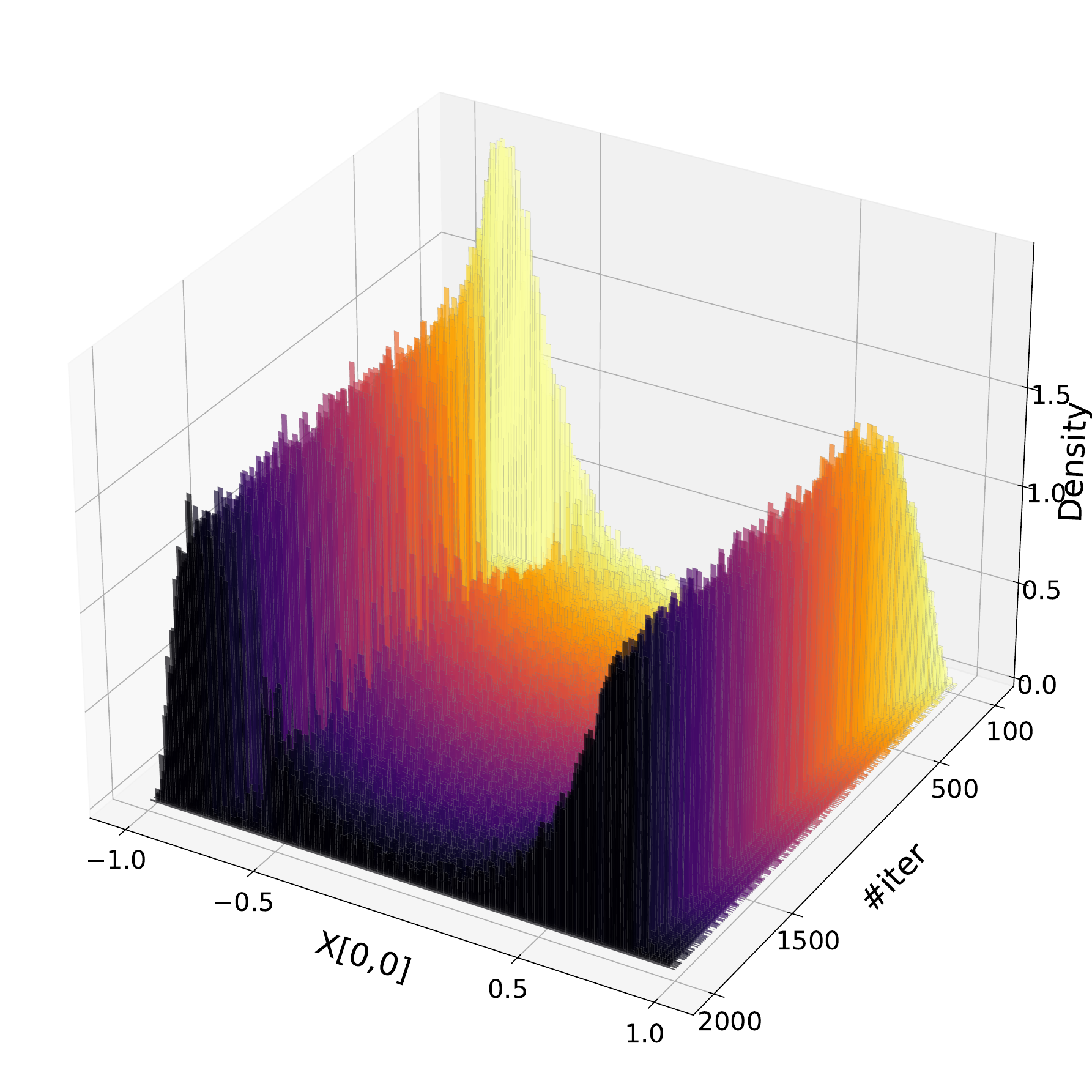}} 
    \caption{Visualizing the exponential convergence of our Lie group sampler (Alg.\ref{algo_sampling_Lie_group}) for sampling a multimodal distribution on $\mathsf{SO}(10)$.}
\end{figure}

\acks{We thank Andre Wibisono, Hongyuan Zha, and Matthew Zhang for inspiring discussions. This research is partially supported by NSF DMS-1847802, Cullen-Peck Scholarship, and GT-Emory Humanity.AI Award. We also thank the reviewers for insightful feedback}.

\bibliography{ref}
\newpage
\appendix
\section{Notation and map}
Please see Table \ref{tab_notation} for a list of notations and Fig. \ref{fig_proof_map} for a graph showing the dependence of theorems.
\begin{table}
\centering
\begin{tabular}{| c|c | c | l |}
\hline
Category &Notation & Location & Description \\ 
\hline
Basic & $\G$ &   & Lie group\\ 
&$\g$ &   & Lie algebra\\ 
&$g$ &  Sec. \ref{sec_Lie_group} & Element in Lie group\\ 
&$\xi$ &   & Element in Lie algebra\\ 
&$m$ &  & dimension of Lie groups /algebra\\ 
\hline
Group  &$\ad$ & Sec. \ref{sec_optimization_dynamics} & Adjoint operator \\
structure&$c_{ij}^k$ & Def. \ref{def_structural_constant} & Structural constant \\
&$C$ & Eq. \eqref{eqn_C} & Operator norm of $\ad$\\
&$N$ & Eq. \eqref{eqn_N} & Set that $\d\log$ do not exist\\
\hline
Riemannian &$\ad^*$ & Sec. \ref{sec_optimization_dynamics} & Coadjoint operator \\
structure&$\Gamma_{ij}^k$ & Sec. \ref{sec_Riemannian_structure} & Christoffel symbol \\
&$\nabla$ & & Gradient or Levi-Civita connection\\
&$d$\tablefootnote{We use $d$ for distance and $\d$ for differential (please note the different fonts).} & Eq. \eqref{eqn_d_TG}  & distance on $\G$ or $\G\times \g$\\
&$D$ & Sec. \ref{sec_inner_product} & Diameter of $\G$\\
\hline
Convergence &$(g_t, \xi_t)$, $(\hat{g}_t, \hat{\xi}_t)$ & Sec. \ref{sec_coupling}  & Coupled r.v. following sampling SDE\\
&$(\tilde{g}_t, \tilde{\xi}_t)$ &  Sec. \ref{sec_splitting} & R.v. from discretization\\
&$C_k$ & Lemma \ref{lemma_bound_xi_power} & Upper bound for $\mathbb{E}\norm{\xi}^k$ \\  
&$x$, $y$ & Eq. \eqref{eqn_xy_ode} & ODE quantifying local numerical error \\
&$h$ &  & Step size for discretization \\
\hline
Others&$\omega$ & Eq. \eqref{eqn_symplectic_2_form} & symplectic 2-form \\
&$\left\{e_i\right\}_{i=1}^m$ & & A set of orthonormal basis in $\g$ \\
&$p$ and $p^{-1}$ & Eq. \eqref{eqn_p}, \eqref{eqn_p_inv} & power series for $\d\log$ and  $\d \exp$\\
&$\alpha$, $\beta$, $R$, $e_t$, $Z_t$, $Q_t$ & Sec. \ref{sec_coupling}, \ref{sec_semi_distance} & coupling and semi-distance \\
&$K$& Sec. \ref{sec_semi_martingale_decomposition}& drift term in decomposition of $\rho$ \\
\hline
\end{tabular}
\caption{A list of notations}
\label{tab_notation}
\end{table}
\begin{figure}[H]
    \centering
    \includegraphics[width=1.0\textwidth]{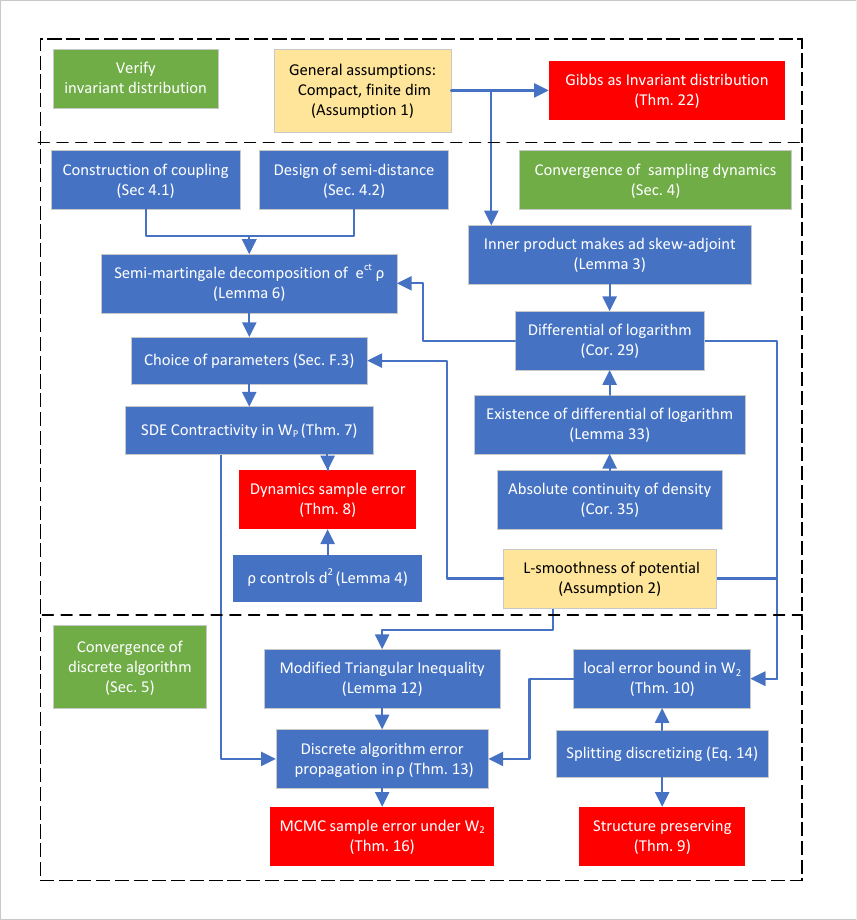}
    \caption{Diagram of the dependence of theorems}
    \label{fig_proof_map}
\end{figure}

\section{More discussion on related works}
\subsection{More discussion on Langevin sampling in Euclidean spaces}
\label{sec_related_work_Euclidean}
In order to analyze the convergence of Langevin-type sampling algorithms, at least two common ways of quantification have been used, based on different conditions and analysis techniques. 1) When convergence is quantified under Wasserstein distance, coupling methods provide useful tools. 2) When convergence is measured via f-divergence, a Lyapunov approach that quantifies convergence in the space of probability densities is often employed. Roughly speaking, 2) requires weaker conditions than 1), but is more challenging to work with when momentum is involved. Here are some more details.

Wasserstein metric is a distance between measures. Coupling methods are usually used for proving the convergence rate due to the definition of Wasserstein distance. For overdamped Langevin, it is a standard proof that synchronous coupling can establish its contractivity under a strong convexity assumption of the potential. Kinetic Langevin is a little more complicated, but it was also known that when its friction coefficient is constant, contractivity still stands after a constant linear coordinate transformation, for both the continuous SDE and its discretization \citep[e.g.,][]{dalalyan2020sampling, cheng2018underdamped}. Contractivity allows the convergence analysis of a continuous dynamics in Wasserstein distance to carry over to the analysis of its time-discretization, i.e. sampling algorithms \citep[e.g.,][]{li2021sqrt}, and it was known that the sampling accuracy can be improved via good numerical discretization \citep[e.g.,][] {shen2019randomized}. In addition, coupling method can be generalized to potential function under weaker conditions, e.g., convexity outside a ball. For example, reflection coupling worked well for overdamped Langevin \citep{eberle2016reflection}, and the momentum case (i.e. kinetic Langevin) is later considered \citep{eberle2019couplings}. The clever but complicated design of coupling and semi-distance function employed in these approaches makes it nontrivial to obtain an explicit convergence rate for the discrete algorithm. 

Another quantification of convergence is based on some f-divergence from one measure to another. For example, Langevin dynamics and its discretization of Langevin Monte Carlo can be shown convergent in various f-divergences under various isoperimetric inequality assumption on the target distribution, such as in KL under Logarithmic Sobolev Inequality (LSI) and in chi-square under Poincare inequality (PI) \citep[e.g.,][]{vempala2019rapid,erdogdu2022convergence,chewi2024log}. 
Since PI is weaker than LSI, and LSI is weaker than convexity outside a ball, in some sense convergence in f-divergence requires a weaker condition than that in Wasserstein distance\footnote{Although better convergence rates can be established under stronger conditions \citep[e.g. ][]{dalalyan2017theoretical, cheng2018convergence, durmus2019analysis, li2021sqrt}.}. The proof of the convergence by \citet{vempala2019rapid}, for instance, is based on splitting the Markov kernel into a deterministic part and a Brownian motion part. The deterministic part keeps the KL/Renyi divergence unchanged and only ensures the invariant distribution is correct. The Brownian motion part mollifies the density and leads to a monotonic decrease in the KL divergence between the current distribution and the target distribution. However, this approach for example is difficult to generalize to the momentum case (i.e. kinetic Langevin, frequently referred to as underdamped Langevin too) due to the degeneracy of noise. Unlike in the \{Wasserstein distance + strong convexity\} case where a constant linear change of coordinates helps recover contractivity, the analysis of convergence in f-divergence under isoperimetric inequalities is often a case-by-case study. For example, \cite{ma2021there} proved the convergence of KLMC in KL divergence under LSI assumption by adding a carefully designed cross term to the joint KL divergence. \cite{zhang2023improved} provided additional tools that allow improved analysis of the convergence of KLMC (e.g., the smoothness assumption on Hessian in \cite{ma2021there} is no longer needed), and convergence in both KL and Renyi is provided. \cite{altschuler2023faster} proposed an innovative technique based on Renyi divergence with Orlicz-Wasserstein shifts and used it to prove many great results, such as the convergence rate with improved dimensional dependence of Metropolis-adjusted KLMC, under a variety of metric including total variance, Chi-square and KL divergence and $W_2$ distance with the assumption that the target distribution satisfies either LSI or PI.

\subsection{When momentum meets curved spaces}
\label{sec_momentum_meets_curved_spaces}
Both curved space and momentum (which renders the underlying Markov process for sampling no longer reversible) lead to extra difficulties. Quantifying numerical error in curved spaces is much harder compared to the cases in flat spaces \citep[e.g.,][]{gatmiry2022convergence,cheng2022efficient}, and momentum leads to the degeneracy of noise, which means extra techniques are needed \citep[e.g.,][]{dalalyan2020sampling, cheng2018convergence, shen2019randomized, ma2021there, zhang2023improved, yuan2023markov, altschuler2023faster}. However, when both curved space and momentum show up together, there is an extra difficulty: when studying kinetic Langevin, we are considering the convergence of measures on the phase space, i.e., the product space of position and momentum. When the space is curved, momentum is a vector in the tangent space\footnote{Rigorously speaking, momentum should be in the cotangent space while velocity is in the tangent space, but we will follow the convention and not distinguish velocity and momentum.} of the current position, and the phase space becomes the tangent bundle. Convergence analysis on the tangent bundle is challenging, due to different reasons in the f-divergence case and the Wasserstein distance case, and we will discuss them separately now.

When quantifying convergence in f-divergence, we are considering the ratio between the current density and the density of the invariant distribution, and the Fokker-Planck equation governs the density evolution, which involves the (spatial) gradient of density. Therefore, we can not bypass the need of calculating the gradient on the tangent bundle. However, a Riemannian metric is required to calculate the gradient, and in this case, we need a Riemannian metric for the entire tangent bundle instead of just the curved space. Although one can induce a Riemannian metric of the tangent bundle from the Riemannian metric of the manifold (e.g., Sasaki metric), how to apply it to construct and analyze a kinetic-type Langevin dynamics or its discretization is an open problem. 

When quantifying convergence in Wasserstein distance, we need to compare how far two points are in the phase space, which is typically then used as the cost function to induce a Wasserstein distance. In the manifold case, one can compare momenta (i.e. vectors in tangent space) via parallel transport, which however could be path-dependent, i.e., different paths connecting the two points give different linear isomorphisms between tangent spaces. In other words, a uniform way to compare momenta at different positions does not exist. This difficulty is not severe in optimization, since we only consider a trajectory in that case. A natural choice is to move the momentum along the position trajectory (e.g.,  \cite{ahn2020nesterov}). However, in sampling tasks, the position as a random variable takes values over the whole manifold, and as a result, we need a uniform way to compare momenta, and so far we do not know how to do so. To the best of our knowledge, there is no known result that extends kinetic Langevin dynamics or KLMC to the manifold case, let alone any analysis of the convergence. 

All those difficulties due to curved spaces and momentum still need to be addressed when we construct and analyze kinetic Langevin Monte Carlo on Lie groups. What further complicates the problem is, the additional group structure is not free lunch. For example, here is one chain of complications: 1) Lie groups with left-invariant metric and sectional curvature negative everywhere are very limited \citep{milnor1976curvatures}. 2) On many Lie groups, there exists two points such that there are more than one geodesics connecting them. This further implies: 3) (geodesic) strong convexity is too strong for many Lie groups as there is no nontrivial convex function along a closed geodesic. As a result, we do not want to make assumptions about the potential function other than smoothness. See Sec. \ref{sec_compare_conditions} for more details. However, despite all those difficulties from Lie groups, we can enjoy a useful technique, namely (left) trivialization, to establish a uniform way of comparing momenta in the tangent spaces at different points. Trivialization provides an alternative tool for bijection between tangent spaces, by related both to a fixed linear space known as the Lie algebra, and plays a key role in this paper. We will build the sampling dynamics by trivialized momentum and then have our structure-preserving discretization based on it. Our proofs will also heavily rely on trivialization.

\subsection{Discussion on commonly used conditions for the potential function}
\label{sec_compare_conditions}
As mentioned earlier, normally some assumptions are required for the potential function. $L$-smoothness is the most commonly assumed one, and our work also assumes it to compare the gradients of potential at different positions. Meanwhile, in most existing results, additional assumptions are used to ensure geometric ergodicity. More precisely - 

In Euclidean space, when $U$ is strongly convex, one can prove the convergence of kinetic Langevin by showing the contractivity of dynamics after a constant coordinate transformation \citep{dalalyan2020sampling}. It is unclear to us how to do the coordinate transformation in curved spaces, because it can no longer be constant due to the nonlinearity of space, but then additional challenges arise (e.g., it no longer induces a metric as essentially done in \citep{dalalyan2020sampling}). In fact, even discussing what happens under convexity is vacuous on Lie groups because there is no nontrivial geodesically convex function on compact Lie group \citep{yau1974non}. An intuition for this is, any convex function on a closed geodesic must be constant.

One relaxation of strong convexity is distant-dissipativity \citep{cheng2022efficient}. In flat space, it is called strong-dissipativity \citep{eberle2016reflection, erdogdu2022convergence}, defined by
\begin{align}
\label{eqn_strong_dissipativity}
    \ip{\nabla U(x)-\nabla U(y)}{x-y}\ge m_1\norm{x-y}^2, \forall \norm{x-y}\ge R_1 \quad \text{for some}\quad m_1,R_1>0
\end{align}
also known as `strongly convexity outside a ball'. Although such condition is widely used and helpful in allowing us to bypass nonconvexity by containing it inside a ball, it is however still a rather strong assumption. The reason is, a function with distant-dissipativity is still convex at a large scale, but no nontrivial (i.e. nonconstant) geodesically convex function exists on manifolds with closed geodesic. Fig.\ref{fig_convex_outside_a_ball} illustrates why this is the case.

\begin{figure}
    \centering
    \subfigure[Parameterization of $\mathsf{SO}(2)\cong S^1$]{\qquad\includegraphics[width=0.3\textwidth]{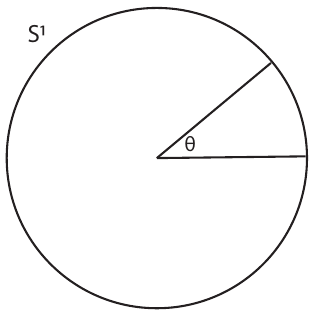} \qquad\label{fig_convex_outside_a_ball_a}}
    \subfigure[Red solid: a function on $\mathsf{SO}(2)$ that might be mistaken as convex outside a ball; Blue solid: its `convexification' which is not convex; Dashed lines: what periodicity and convexity would require, which create inconsistency.]{\quad\quad\includegraphics[width=0.45\textwidth]{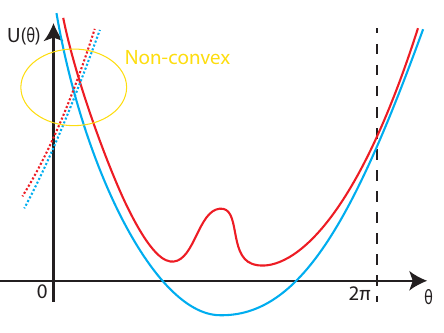}\quad\qquad\label{fig_convex_outside_a_ball_b}}    
    \caption{An illustration of the non-existence of strongly convex or distant-dissipative function on a manifold with closed geodesics. The intuition is a convex function (if any) would grow in both sides, but along a loop, increasing $\theta$ eventually leads to $\theta$ reset to 0 and increasing it again, which means the function needs to be both increasing and decreasing as $\theta$ grows (e.g., the yellow region), which is a contradiction.}
\label{fig_convex_outside_a_ball}
\end{figure}

There are weaker conditions, such as log-Soblev inequality (LSI) and Poincare inequality (PI), but it is unclear how to use them together with coupling methods for manifolds.

To summarize, the commonly used assumptions on the potential function are not necessarily suitable for our compact Lie group case, and this is why we only make the $L$-smooth assumption on the potential function.

\section{More details about Sec. \ref{sec_preliminaries} and \ref{sec_sampling_dynamics}}
\subsection{Examples for Lie groups}
In this section, we give some examples of Lie groups. The most well-known Lie group is the Matrix group:
\begin{example}[General Linear group]
    \label{example_matrix_lie_group}
    Denote \textbf{general linear group} by $\mathsf{GL}(n,\mathbb{R}) := \{X\in\mathbb{R}^{n\times n}: \det{X}\neq 0\}$ with the group multiplication defined by matrix multiplication. The corresponding Lie algebra $\mathfrak{gl}(n,\mathbb{R}):=\mathbb{R}^{n\times n}$ has the matrix commutator: $\lb{A}{B}:=AB-BA$ as its Lie bracket. Consequently, $\lte{g} \xi$ is $g\xi$ (matrix multiplication of two $n$-by-$n$ matrices) and the exponential map is the matrix exponential. 
    
    An natural choice of the inner product on $\mathfrak{gl}$ is $\ip{A}{B}=\tr (A^\top B)$. Under the left-invariant metric induced by this inner product, $\ad^*_A B=B A^\top-A^\top B$. Note that the operator $\ad$ is not skew-adjoint in this case. In fact, the inner product that makes $\ad^*$ skew-adjoint does not exist. A necessary and sufficient condition for the existence of such inner product is provided in \cite{milnor1976curvatures}.
\end{example}
The definition of general Lie groups as manifolds in Sec. \ref{sec_Lie_group} may sound abstract, but in fact, the following remark shows all compact Lie groups can be viewed as a subgroup of a matrix group in Example \ref{example_matrix_lie_group}. 
\begin{remark}[Matrix representation of Lie groups/Lie algebra]
    Peter–Weyl theorem implies that every compact Lie group is a closed subgroup of $\mathsf{GL}(n,\mathbb{R})$ for some $n$. Ado's theorem \citep{ado1947representation} states that every finite-dimensional Lie algebra on $\mathbb{R}$ can be viewed as a Lie subalgebra of $\mathfrak{gl}(n, \mathbb{R})$ for some $n$.
\end{remark}

Although the general linear group is not compact making our assumption \ref{assumption_general} fails, many of its Lie subgroups are compact. A nice example is the following:
\begin{example}[Special Orthogonal group $\mathsf{SO}(n)$]
\label{example_SOn}
    The special orthogonal group is defined as 
    \begin{equation*}
        \mathsf{SO}(n):=\left\{X\in\mathbb{R}^{n\times n}: X^\top X=I, \det(X)=1\right\}   
    \end{equation*}
    The corresponding Lie algebra is $\mathfrak{so}(n):=\left\{\xi\in\mathbb{R}^{n\times n}: \xi^\top+ \xi=0\right\}$. Since it is a Lie subgroup, the group operation, left translation, and exponential map are the same as general linear group (Example \ref{example_matrix_lie_group}). The inner product that makes $\ad$ skew-adjoint in Lemma \ref{lemma_ad_self_adjoint} is given by $\ip{A}{B}:=\tr (A^\top B)$.
\end{example}

We also give an explicit example such that an inner product that makes $\ad$ skew adjoint (Lemma \ref{lemma_ad_self_adjoint}) may not exist when assumption \ref{assumption_general} does not hold:
\begin{example}[Lie group with constant negative curvature]
    We consider a 2-dim Lie group whose Lie algebra has basis $e_1$ and $e_2$. Define a linear operator as $l:\G\to\mathbb{R}$ by $l(x e_1+y e_2):=ax+by, \forall x, y$ for some constants $a, b\neq 0$ and the Lie bracket as $\lb{e_1}{e_2}=l(e_1) e_2-l(e_2)e_1$.  Example 1.7 in \cite{milnor1976curvatures} shows this Lie group has strictly negative sectional curvature for any left-invariant metric.

    We can explicitly express this example as a Lie subgroup of the general linear group using adjoint representation. Denote $\G:=\left\{x e_1+y e_2:x,y\in \mathbb{R}\right\} \subset \mathsf{GL}(2, \mathbb{R})$ where the basis are defined as $e_1=\begin{pmatrix}
        0&0\\
        -b&a
    \end{pmatrix}$
    and 
    $e_2=\begin{pmatrix}
        b&-a\\
        0&0
    \end{pmatrix}$.  Together with Eq. \ref{eqn_sectional_curvature}, this tells us the inner product in Lemma \ref{lemma_ad_self_adjoint} that makes $\ad$ skew-adjoint does not exist.
\end{example}

The flat Euclidean space is a trivial Lie group and our sampling dynamics recover kinetic Langevin dynamics in the Euclidean space. 
\begin{example}[Lie group structure for Euclidean spaces]
\label{example_Euclidean_Lie_group}
    Euclidean space is a trivial (commutative)  Lie group, i.e., when $\G=\mathbb{R}^n$ with the group multiplication defined as $g_1 g_2:=g_1+g_2$ (vector summation). The corresponding Lie algebra is $\g=\mathbb{R}^n$ with vanish Lie bracket. The exponential map is the identity map on $\mathbb{R}^n$.
\end{example}
The left-trivialization in the Euclidean space is the identity map, and Eq. \eqref{eqn_sampling_SDE} in Euclidean spaces becomes
    \begin{equation}
    \begin{cases}
        \dot{q}=p\\
        \d p=-\gamma p \d t-\nabla U(q)\d t+\sqrt{2\gamma}\d W
        \label{eqn_sampling_SDE_Euclidean}
    \end{cases}
\end{equation}
which is the well-known kinetic Langevin dynamics.

Although many elegant results have been established for the Euclidean case (Sec. \ref{sec_related_work_Euclidean}, especially \cite{eberle2019couplings}), our analysis only focuses on compact Lie groups and cannot be applied to the Euclidean sampler (despite that our algorithm does).

\subsection{Symplectic structure on Lie groups}
\label{sec_symplectic_structure}
The tangent bundle of a Lie group $T\G\cong \G\times \g$ has a natural symplectic structure \citep[Prop. 4.4.1]{abraham1978foundations}. More specifically, at $(g, \mu)\in \G\times\g$, the symplectic 2-form is given by
\begin{equation}
\label{eqn_symplectic_2_form}
    \omega_{(g,\mu)}((v,\alpha), (w,\beta))=\ip{\beta}{\lt{g}v}-\ip{\alpha}{\lt{g}w}+\ip{\mu}{[\lt{g}v, \lt{g}w]}
\end{equation}
for any $(v, \alpha), (w, \beta)\in \G\times T_g\G$.

Given the symplectic structure above, we can provide the optimization dynamics with a mechanic view. Setting the Hamiltonian $H:G\times \g\to \mathbb{R}$ as $H(g, \xi):=U(g)+\frac{1}{2}\ip{\xi}{\xi}$ gives a Hamiltonian field $X_H$, defined as the unique vector field on $\G\times \g$ satisfying $\d H=\omega(X_H, \cdot)$. The Hamiltonian flow $\frac{\d}{\d t}(g_t, \xi_t)=X_H$ is exactly Eq. \eqref{eqn_optimization_ODE} with $\gamma=0$. The term $\ad_\xi \xi$ in Eq. \eqref{eqn_optimization_ODE}, which does not show in flat Euclidean spaces, comes from the third term in the symplectic 2-form in Eq. \eqref{eqn_symplectic_2_form}.

On a symplectic manifold, there is always a natural volume form  $\omega^m:=\omega\wedge\dots\wedge\omega$, which provides a base measure. \cite{arnaudon2019irreversible} uses $\omega^m$ as the base measure distribution when proving the invariant distribution. How is this symplecticity-based measure $\omega^m$ related to the group-structure-based measure $\d g \d \xi$ in Thm.\ref{thm_invariant_distribution_Gibbs}? They are identical, as proved in Thm. \ref{thm_equivalence_of_measure}. Note that this theorem do not depend our special choice of inner product that makes $\ad$ skew-adjoint in Lemma \ref{lemma_ad_self_adjoint}.

\begin{theorem}[Equivalence of base measure]
\label{thm_equivalence_of_measure}
    The product measure on $\G\times \g$, i.e., the product of left Haar measure on $\G$ and Lebesgue measure on $\g$, is identical to the measure induced by the volume form $\omega^m$ up to a constant.
\end{theorem}
\begin{proof}[Proof of Thm. \ref{thm_equivalence_of_measure}]
    The outline of this proof is: we will choose $2m$ vector fields on the tangent space of $2m$-dim manifold $\G\times \g$, and we calculate the ratio of the two volume forms corresponding to those two different vector fields when applying to the $2m$ vector fields. We will prove the ratio is constant, showing the two measures are identical up to a constant.
    
    We fix the following $2m$ vector fields, $\{(\lte{g} e_i, 0)\}_{i=1,\dots, m}$ and $\{(0, e_j)\}_{j=1,\dots, m}$, where vector fields $\lte{g} e_i$ on $\G$ are left invariant vector fields generated by $e_i$. In the following, we prove that $\omega^m_{(g, \mu)}\left((\lte{g} e_1, 0), \dots, (\lte{g} e_m, 0), (0, e_1), \dots, (0, e_m)\right)$ is constant.

    By the expression of $\omega$ in Eq. \eqref{eqn_symplectic_2_form},
    \begin{align*}
        &\omega_{(g, \mu)}((\lte{g} e_i, 0), (\lte{g} e_j, 0))=\ip{\lb{e_i}{e_j}}{\mu}\\
        &\omega_{(g, \mu)}((\lte{g} e_i, 0), (0, e_j))=\delta_{ij}\\
        &\omega_{(g, \mu)}((0, e_i), (\lte{g} e_j, 0))=-\delta_{ij}\\
        &\omega_{(g, \mu)}((0, e_i), (0, e_j))=0
    \end{align*}
    The definition of exterior product gives 
    \begin{align*}
        &\omega^m_{(g, \mu)}\left((\lte{g} e_1, 0), \dots, (\lte{g} e_m, 0), (0, e_1), \dots, (0, e_m)\right)\\
        &=\frac{1}{(2m)!}\sum_{\sigma\in S_{2m}} \operatorname{sgn}(\sigma) \omega^m_{(g, \mu)}\sigma\left((\lte{g} e_1, 0), \dots, (\lte{g} e_m, 0), (0, e_1), \dots, (0, e_m)\right)
    \end{align*}
    By the fact that $\omega_{(g, \mu)}((0, e_i), (0, e_j))$ and $\omega_{(g, \mu)}((\lte{g} e_i, 0), (\lte{g} e_j, 0)) + \omega_{(g, \mu)}((\lte{g} e_j, 0), (\lte{g} e_i, 0))$ vanishes, we have that all the non-vanish terms must have the form $\omega_{(g, \mu)}((\lte{g} e_i, 0), (0, e_i))$, and 
    \begin{align*}
        &\omega^m_{(g, \mu)}\left((\lte{g} e_1, 0), \dots, (\lte{g} e_m, 0), (0, e_1), \dots, (0, e_m)\right)\\
        &=\frac{1}{(2m)!}\sum_{\sigma\in S_{m}} 2^m\prod_{i=1, \dots, m} \omega_{(g, \mu)}\sigma\left((\lte{g} e_i, 0), (0, e_i)\right)\\
        &=\frac{m!2^m}{(2m)!}\prod_{i=1, \dots, m} \omega_{(g, \mu)}\sigma\left((\lte{g} e_i, 0), (0, e_i)\right)\\
        &=\frac{m!2^m}{(2m)!}
    \end{align*}
    Since both $\omega^m$ and 
    \begin{align*}
        \d g \d \xi=\d(\lte{g} e_1, 0) \wedge \dots \wedge \d(\lte{g} e_m, 0)\wedge \d(0, e_1)\dots\wedge \d(0, e_m)
    \end{align*}
    applied to our $2m$ independent vector fields gives a non-zero constant function, they are identical up to a constant.
\end{proof}    
\subsection{Discussion on the term $\ad^*_\xi \xi$}
\label{sec_discussion_ad_star}
There are several reasons why the term $\ad^*_\xi \xi$ is required in both optimization dynamics Eq. \eqref{eqn_optimization_ODE} and sampling dynamics Eq. \eqref{eqn_sampling_SDE}: 
\begin{enumerate}
    \item From the view of mechanics (discussed in Sec. \ref{sec_symplectic_structure}), it comes from the third term of the natural symplectic 2-form Eq. \eqref{eqn_symplectic_2_form}.
    \item From the view of Riemannian geometry (will be discussed later in Sec. \ref{sec_about_ad_self_adjoint}), it is a term from the definition of geodesics.
    \item Another technical reason is the term is required to ensure the invariant distribution is correct because the divergence of left-invariant vector fields does not always vanish, see more details in the proof of Thm. \ref{thm_invariant_distribution_Gibbs}. 
\end{enumerate}
However, despite the necessity of the term $\ad^*_\xi \xi$ on a general left-invariant metric, we can carefully choose the inner product to make it vanish (Lemma \ref{lemma_ad_self_adjoint} and Rmk. \ref{rmk_inner_product_explicit_expression}), and all the above still hold.

\subsection{Gibbs as invariant distribution of sampling dynamics}
This section will be devoted to showing that the Gibbs distribution $\G\times \g$ Eq. \eqref{eqn_Gibbs_TG} is an invariant distribution of our sampling dynamics Eq. \eqref{eqn_sampling_SDE} under mild conditions. Unlike the stronger Assumption \ref{assumption_general} we made for convergence, we do not require the special choice of inner product on $\g$ in Lemma \ref{lemma_ad_self_adjoint} nor the compactness of Lie group in this section.

\begin{definition}[Structural constant]
\label{def_structural_constant}
    By denoting $\{e_i\}_{i=1}^m$ as a set of orthonormal basis of $\g$ under inner product $\ip{\cdot}{\cdot}$, the structural constant of Lie algebra $\g$ defined as
    \begin{equation*}
        c_{ij}^k=\ip{\lb{e_i}{e_j}}{e_k}
    \end{equation*}
\end{definition}

\begin{lemma}
\label{lemma_support_for_invariant_distribution}
For any $\xi\in \g$, we have
\begin{align*}
    \divg_g(\lte{g}\xi)+\divg_\xi\left(\ad^*_\xi \xi\right)=0
\end{align*}
where $\lte{g}\xi$ is the left-invariant vector field generated by $\xi$.
\end{lemma}
\begin{proof}[Proof of Lemma \ref{lemma_support_for_invariant_distribution}]
    Consider the local chart on $\G$ at $g$ given by $\{g\exp(x_i e_i)\}_{i=1}^m$ for $x_i\in [-\varepsilon, \varepsilon]^m$. We have 
    \begin{align*}
        \divg_g(\lte{g}\xi)=\sum_i \partial_i \xi_i +\sum_{ij}\Gamma_{ji}^j \xi_i 
    \end{align*}
    By Eq. \eqref{eqn_christoffel_symbol}, $\Gamma_{ji}^j=0, \forall i, j$ , and the term $\sum_{ij}\Gamma_{ji}^j \xi_i$ vanishes consequently. By our choice of the local coordinate,
    \begin{align*}
        \partial_i \xi_i&=\mathcal{L}_{e_i} \ip{\lte{\hat{g}} \xi}{e_i}|_{\hat{g}=e}\\
        &= \ip{\mathcal{L}_{e_i} \lte{\hat{g}} \xi}{e_i}|_{\hat{g}=e}+ \ip{\lte{\hat{g}} \xi}{\mathcal{L}_{e_i} e_i}|_{\hat{g}=e}\\
        &= \ip{\lb{e_i}{\xi}}{e_i}+ \ip{\lte{\hat{g}} \xi}{\lb{e_i}{e_i}}\\
        &=\ip{\lb{e_i}{ \xi}}{e_i}
    \end{align*} 
    where $\mathcal{L}$ is the Lie derivaitve. In the first line of equation, $\lte{\hat{g}} \xi$ is the left-invariant vector field generated by $\xi$. $|_{\hat{g}=e}$ means we evaluate the Lie derivative at $e$.

    After taking summation w.r.t. $i$, we have $\divg_g(\lte{g}\xi)=\sum_i \partial_i \xi_i=\ip{\xi}{\sum_k \ad^*_{e_k} e_k}$.
    
    For $\divg_\xi\left(\ad^*_\xi \xi\right)$, the divergence is taken in the Euclidean space $\g$ and a direct calculation gives
    \begin{align*}
        \divg\ad^*_\xi \xi&=\sum_i\partial_i \left(\ad^*_\xi \xi\right)_i\\
        &=\sum_{ijk}\partial_i \xi_j \xi_k \ip{\ad^*_{e_j} e_k}{e_i}\\
        &=-\sum_{i}\ip{\xi_j}{\ad^*_{e_i} e_i}\\
        &=-\divg_g(\lte{g}\xi)
    \end{align*}
\end{proof}

Now we are ready to prove Thm. \ref{thm_invariant_distribution_Gibbs}. The outline for the proof is: 1) calculate the infinitesimal generator $\mathcal{L}$ by definition; 2) find its adjoint operator $\mathcal{L}^*$ under $L^2$; 3) verify that the Gibbs distribution is a fixed point of the adjoint of the infinitesimal generator. 
\begin{theorem}
\label{thm_invariant_distribution_Gibbs}
    Suppose the Lie group $\G$ is finite dimensional, then the sampling dynamics Eq.\eqref{eqn_sampling_SDE} has invariant distribution $\nu_*$ in Eq. \eqref{eqn_Gibbs_TG}.
\end{theorem}
\begin{proof}[Proof of Thm. \ref{thm_invariant_distribution_Gibbs}]
    We first write down the infinitesimal generator $\mathcal{L}$ for SDE \eqref{eqn_sampling_SDE}. For any $f\in C^2(\G\times \g)$, $\mathcal{L}$ is defined as
    \begin{align*}
        \mathcal{L}f (g, \xi):=&\lim_{\delta\rightarrow 0}\frac{\mathbb{E}\left[f(g_\delta, \xi_\delta)|(g_0, \xi_0)=(g, \xi)\right]-f(g, \xi)}{\delta}\\
        =&\ip{\nabla_g f}{\lte{g}\xi}+\ip{\nabla_\xi f}{-\gamma\xi+\ad^*_\xi \xi-\lt{g}(\nabla U(g))}+\gamma \Delta_\xi f\\
        =&\ip{\xi} {\lt{g}\nabla_g f}-\gamma\ip{\nabla_\xi f}{\xi}+\ip{\nabla_\xi f}{\ad^*_\xi \xi}-\ip{\nabla_\xi f}{\lt{g}(\nabla U(g))}+\gamma \Delta_\xi f
    \end{align*}

    We denote the adjoint operator of $\mathcal{L}$ by $\mathcal{L}^*$, i.e., $\mathcal{L}^*:C^2\rightarrow C^2$ satisfying $\int_{\G\times \g}\nu \mathcal{L}f \d g \d \xi=\int_{\G\times \g}f\mathcal{L}^*\nu \d g \d \xi$ for any $f, \nu\in C_0^2(\G\times \g)$. By the divergence theorem, we have
    \begin{align*}
        &\int_{\G\times \g}\nu \mathcal{L}f\, \d g \d \xi\\
        &=\int_{\G\times \g}-f\divg_g (\nu \lte{g}\xi)-f\divg_\xi\left(\nu\ad^*_\xi \xi\right)+f\divg_\xi\left(\nu \lt{g}(\nabla U(g)) \right)\d g \d \xi\\
        &+\gamma\int_{\G\times \g}f\divg_\xi (\nu\xi)+f\Delta_\xi \nu \d g \d \xi
    \end{align*}
    Here $\lte{g}\xi$ stands for the left-invariant vector filed on $\G$ generated by $\xi$. As a result, we have 
    \begin{align*}
        \mathcal{L}^*\nu&=-\divg_g (\nu \lte{g}\xi)-\divg_\xi\left(\nu\ad^*_\xi \xi\right)+\divg_\xi\left(\nu \lt{g}(\nabla U(g)) \right)+\gamma\left(\divg_\xi (\nu\xi)+\Delta_\xi \nu\right)\\
        &=-\ip{\nabla_g \nu}{\lte{g}\xi}-\ip{\nabla_\xi \nu}{\ad^*_\xi \xi}+\ip{\nabla_\xi \nu}{\lt{g}(\nabla U(g))}+\gamma\divg_\xi (\nu\xi+\nabla_\xi \nu)\\
        &-\nu\divg_g(\lte{g}\xi)-\nu\divg_\xi\left(\ad^*_\xi \xi\right)
    \end{align*}
    We emphasize that the divergence of the left-invariant vector does not necessarily vanish, and Lemma \ref{lemma_support_for_invariant_distribution} shows $\divg_g(\lte{g}\xi)+\divg_\xi\left(\ad^*_\xi \xi\right)=0$, which means the term $\ad^*_\xi \xi$ is necessary to cancel with the divergence of left-invariant vector field and ensure the invariant distribution is correct.
    \begin{align*}
        \mathcal{L}^*\nu&=-\ip{\nabla_g \nu}{\lte{g}\xi}-\ip{\nabla_\xi \nu}{\ad^*_\xi \xi}+\ip{\nabla_\xi \nu}{\lt{g}\nabla U(g)}+\gamma\divg_\xi (\nu\xi+\nabla_\xi \nu)
    \end{align*}
    The last step is verifying that $\nu_*$ given in Eq. \eqref{eqn_Gibbs_TG} is a fixed point of $\mathcal{L}^*$. By direct calculation, we have $\nabla_g \nu_*=\nu_*\nabla U$ and $\nabla_\xi \nu_*=-\nu_*\xi$. Together with expression $\mathcal{L}^*$, we have $\mathcal{L}^* \nu_*=0$, which means $\nu_*$ given in Eq. \eqref{eqn_Gibbs_TG} is an invariant distribution.
\end{proof}
The following remark gives the left Haar measure an intuition.
\begin{remark}[Left Haar measure]
\label{rmk_left_Haar_measure}
The base measure $\d g$ we used is called the `left Haar measure'. Roughly speaking, if we have a `measure' at $e$ (rigorously speaking, a volume form at $e$), we can expand it to the whole Lie group by left multiplication, thanks to the group structure. This measure is the left Haar measure, rigorously defined as the measure that is invariant under the pushforward by left multiplication, which is unique up to a constant scaling factor. The reason why it is `left' is because our SDE Eq. \eqref{eqn_sampling_SDE_ad_star_vanish} and later \eqref{eqn_sampling_SDE} depends on left trivialization and our metric on $\G$ is left-invariant.
\end{remark}
\section{Riemannian structure on Lie groups with left-invariant metric}
\subsection{More discussion on Lemma \ref{lemma_ad_self_adjoint}}
\label{sec_about_ad_self_adjoint}
In the beginning of Sec. \ref{sec_inner_product}, we mentioned that Lemma \ref{lemma_ad_self_adjoint} means the Riemannian structure and the group structure are `compatible'. Here we provide more details: it means the exponential map from the group structure and the exponential map from the Riemannian structure are the same.

\begin{definition}[Exponential map (group structure)]
\label{def_exp_group}
    The exponential map $\exp:\g\rightarrow \G$ is given by $\exp(\xi) = \gamma(1)$ where $\gamma: \mathbb R \to \G$ is the unique one-parameter subgroup of $\G$ whose tangent vector at the group identity $e$ is equal to $\xi$. 

\end{definition}

\begin{definition}[Exponential map (Riemannian structure)]
\label{def_exp_geodesic}
    The exponential map $\exp:T_g \G\rightarrow \G$ is given by $\exp_g(\lte{g} \xi) = \gamma(1)$ where $\gamma: \mathbb R \to \G$ is the unique geodesic satisfying $\gamma(0) = g$ with initial tangent vector $\gamma' (0) = \lte{g} \xi$.
\end{definition}

To compare the two exponential maps, we can write down the ODEs characterising these two exponential maps. Starting from $g_0$ with initial direction $\lte{g_0} \xi_0$, the trajectory of two exponential maps $\exp_{g_0}(t\xi_0), t\in[0,1]$ are given by the ODEs with initial condition $g(0)=g_0$ and $\xi(0)=\xi_0$:

The exponential map from the group structure (Def. \ref{def_exp_group}) is given by
\begin{equation*}
    \begin{cases}
        \dot{g}=g\xi\\
        \dot{\xi}=0
    \end{cases}
\end{equation*}

The exponential map from the Riemannian structure (Def. \ref{def_exp_geodesic}) is given by
\begin{equation*}
    \begin{cases}
        \dot{g}=g\xi\\
        \dot{\xi}=\ad^*_\xi \xi
    \end{cases}
\end{equation*}

If we choose the inner product that makes $\ad$ skew-adjoint (Lemma \ref{lemma_ad_self_adjoint}), the two exponential maps are identical by comparing their ODEs. This means the Riemannian structure and the group structure are compatible and will be our default choice in the following. As a result, we will no longer differentiate the two exponential.

\begin{remark}[Explicit expression of the inner product in Lemma \ref{lemma_ad_self_adjoint}]
\label{rmk_inner_product_explicit_expression}
    The condition that $\ad$ is skew adjoint (Lemma \ref{lemma_ad_self_adjoint}) is equal to the requirement that the metric is bi-invariant, i.e., a metric that is both left-invariant and right-invariant. A left-invariant metric is not always bi-invariant because of the non-commutativity of the group structure. However, a connected Lie group admits such a bi-invariant metric if and only if it is isomorphic to the Cartesian product of a compact group and a commutative group \citep{milnor1976curvatures}. On a compact Lie group,  an explicit expression for a bi-invariant metric is
    \begin{equation*}
        \ip{u}{v}=\int_\G\left(\operatorname{Ad}_g(u), \operatorname{Ad}_g(v)\right)\vartheta(\d g)
    \end{equation*}
    where $(\cdot,\cdot)$ is an arbitrary inner product on $\g$ and $\vartheta$ is the \textbf{right} Haar measure. See \cite{milnor1976curvatures, lezcano2019cheap} for more details.
\end{remark}

\subsection{Connection and curvature of Lie groups with left-invariant metric}
\label{sec_Riemannian_structure}
Given left-invariant vector fields $X$ and $Y$, the Levi-Civita connection on Lie groups with a left-invariant metric is given by \citep{milnor1976curvatures}:
\begin{equation*}
    \nabla_X Y=\frac{1}{2}\left([X,Y]-\ad_X^* Y-\ad_Y^* X\right)
\end{equation*}
which leads to the following Christoffel symbols
\begin{equation}
\label{eqn_christoffel_symbol}
    \Gamma_{ij}^k:=\ip{\nabla_{e_i} e_j}{e_k}=\frac{1}{2}\left(c_{ij}^k-c_{ik}^j-c_{jk}^i\right)
\end{equation}
Here $e_i$ stands for the left-invariant vector field generated by $e_i$.

Under the inner product in Lemma \ref{lemma_ad_self_adjoint}, Levi-Civita connection has a simpler expression given by $\nabla_X Y=\frac{1}{2}[X,Y]$, i.e., Levi-Civita connection is half of the Lie derivative. The Christoffel symbols becomes $\Gamma_{ij}^k=\frac{1}{2}c_{ij}^k$.

The Riemannian curvature tensor is defined by $R(X,Y)Z:=\nabla_X\nabla_Y Z - \nabla_Y\nabla_X Z-\nabla_{[X,Y]} Z$. The sectional curvature on a general manifold is defined as $\kappa(X,Y)=\frac{\ip{R(X,Y)Y}{X}}{\ip{X}{X}\ip{Y}{Y}-\ip{X}{Y}^2}$

Under the condition that the operator $\ad^*$ is skew-adjoint (Lemma \ref{lemma_ad_self_adjoint}), and $X,Y$ are orthonormal, the simplified Levi-Civita expression $\nabla_X Y=\frac{1}{2}[X,Y]$ gives us
\begin{equation}
\label{eqn_sectional_curvature}
    \kappa(e_i, e_j)=\frac{1}{4}\norm{\lb{e_i}{e_j}}^2
\end{equation}
Therefore the sectional curvature is non-negative. If we denote $K$ as the upper bound for the absolute value of sectional curvature, i.e., $K=\sup |\kappa|$, and let\begin{equation}
\label{eqn_C}
    C:=\max_{\norm{X}_2=1}\norm{\ad_X}_{op}
\end{equation}
Then Eq. \eqref{eqn_sectional_curvature} leads to $K=\frac{1}{4}C^2$. Note $C=0$ in the flat Euclidean space.

\subsection{About $L$-smoothness on Lie groups}
The commonly used geodesic $L$-smooth on a manifold $\M$ is given by the following definition \citep[e.g., ][Def. 5]{zhang2016first}:
\begin{definition}[Geodesically $L$-smooth]
\label{def_geodesic_L_smooth}
$U : \G \to\mathbb{R}$ is geodesically $L$-smooth if for any $g, \hat \in \M$,
\begin{align*}
    \norm{\nabla U(g)-\Gamma_{\hat{g}}^g \nabla U(\hat{g})}\le L d(g, \hat{g})
\end{align*}
where $\Gamma_{\hat{g}}^g$ is the parallel transport from $\hat{g}$ to $g$.
\end{definition}
The following lemma, quoted from \cite{kong2024quantitative} just for completeness, shows it is equal to the commonly used geodesic $L$-smoothness under the left-invariant metric in Lemma \ref{lemma_ad_self_adjoint}:
\begin{lemma}[Equivalence of trivialized smoothness and geodesic smoothness]
\label{lemma_L_smooth_equivalence}
    Under the assumption of $\ad^*_X$ is skew-adjoint $\forall X\in\g$, Assumption \ref{assumption_L_smooth} is identical to Def. \ref{def_geodesic_L_smooth}.
\end{lemma}
\begin{proof}[Proof of Lemma \ref{lemma_L_smooth_equivalence}, from \cite{kong2024quantitative}]
    For any $g, \hat{g}\in\G$, consider the shortest geodesic $\phi:[0,1]\to \G$ connecting $g$ and $\hat{g}$ and denote $\xi=\lt{g}\nabla U(g)$. Using the condition $\ad$ is skew-adjoint, we have $\dot \phi(t)=0$ and $\lte{\phi(t)} \xi$ is parallel along $\phi$ by checking the condition for parallel transport \citep[Thm. 1]{guigui2021reduced}:
    \begin{align*}
        \frac{\d}{\d t}\xi=0=-\frac{1}{2}\left[\lt{\phi(t)^{-1}} \dot \phi(t),\xi\right]
    \end{align*} 

    This tells that 
    \begin{align*}
        \lt{g}\Gamma_{g}^{\hat{g}} \nabla f(g)=\lt{\hat{g}}\nabla f(\hat{g})
    \end{align*}
    Together with the metric is left-invariant, we have
    \begin{align*}
        \norm{\lt{g}\nabla U(g)-\lt{\hat{g}} \nabla U(\hat{g})}\le L d(g, \hat{g})
    \end{align*}
    which is identical to Assumption \ref{assumption_L_smooth}.
\end{proof}
\subsection{Non-commutativity of Lie groups}
Comparing with the Euclidean space, Lie groups lack of commutativity, i.e., for $g, \hat{g}\in \G$, $g\hat{g}$ and $\hat{g}g$ are not necessarily equal. This can also be characterized by the non-trivial Lie bracket. This non-commutativity leads to the fact that $\exp(X)\exp(Y) \neq \exp(X+Y)$. An explicit expression for $\log(\exp(X)\exp(Y))$ is given by Dynkin's formula \citep{dynkin2000calculation}. In the matrix group case, $\exp$ is matrix exponentiation and multiplication is matrix multiplication, and for matrices $A, B$, we no longer have $\exp(A)\exp(B) = \exp(A+B)$.

We will not use Dynkin's formula directly but only 2 corollaries of it, Cor. \ref{cor_d_exp} and \ref{cor_d_log}. By defining power series
    \begin{equation}
    \label{eqn_p}
        p(x):=\frac{x}{1-\exp(-x)}=\sum_{k=0}^\infty\frac{B_k(1)}{k!}x^k
    \end{equation}
    and its inverse 
    \begin{equation}
    \label{eqn_p_inv}
        p^{-1}(x):=\frac{1-\exp(-x)}{x}=\sum_{k=0}^\infty\frac{(-1)^k}{(k+1)!}x^k
    \end{equation}
    where $B_n$ are Bernoulli polynomials, we have

\begin{corollary}[Differential of exponential]
    \label{cor_d_exp}
    The differential of matrix exponential is given by
    \begin{equation}
        \label{eqn_d_exp}
        \d \exp_X (Y)=\lte{\exp{X}}\left[p^{-1}(\ad_X)Y\right], \quad \forall X, Y\in \g
    \end{equation}
\end{corollary}
Thm. 2.2 in \cite{brocker2013representations} shows that the group exponential is surjective on compact connected Lie groups. Thus, We can define \textbf{logarithm} $\log_g: \G\rightarrow T_g \G$ as the inverse of an exponential map. Since the exponential map is not injective in general cases, its inverse, logarithm is not uniquely defined. In this paper, we use all the $\log$ to denote the one corresponding to the minimum geodesic, i.e., $d(g, h)=\norm{\log_g h}, \forall g, h$. In the cases that the minimum geodesic is not unique, we use $\log$ to denote any one of them. We omit the subscript when the base point is identity, i.e., $\log:=\log_e$. Note that we do not require the uniqueness of the geodesic connecting 2 points.

We also need to consider the differential of the logarithm. However, since the exponential map is not an injection, the logarithm is not always differentiable. We denote the set that the differential of logarithm does not exist as 
\begin{equation}
\label{eqn_N}
    N:=\left\{g\in \G: \d \log_g \text{ is not well defined}\right\}
\end{equation} 

For $g\in \G\backslash N$, we have the following expression for $\d \log_g$
\begin{corollary}[Differential of logarithm]
    \label{cor_d_log}
     For any $g\notin N$, any $\xi \in \g$, we have the differential of $\log$ is given by
    $\d \log_g: T_g \G\to \g$ is given by
    \begin{align*}
        \d \log_g (\lte{g} \xi)&=p(\ad_{\log_g})\xi
    \end{align*}
\end{corollary}
The main property of $\d\log$ we are going to use is the following corollary, which is useful later in the proof of Lemma \ref{lemma_martingale_decomposition}.
\begin{corollary}
    \label{cor_log_property}
    Under the choice of inner product in Lemma \ref{lemma_ad_self_adjoint}, for any $g\notin N$ and $\xi \in \g$, we have
    \begin{equation}
    \label{eqn_log_property_1}
        \ip{\d \log_g (\lte{g}\xi)}{\log g}=\ip{\log g}{\xi}
    \end{equation}
    \begin{equation}
    \label{eqn_log_property_2}
        \ip{\d \log_g (\lte{g}\xi)}{\xi}\le \norm{\xi}^2
    \end{equation}
\end{corollary}
\begin{proof}[Proof of Cor. \ref{cor_log_property}]
    Proof for Eq. \eqref{eqn_log_property_1}: By Cor. \ref{cor_d_log}, we have
    \begin{align*}
        \ip{\d \log_g(\lte{g}\xi)}{\log g}=\ip{p(\ad_{\log_g})\xi}{\log g}
    \end{align*}
    Using the fact that $\ad^*$ is skew-adjoint, $\ip{\left(\ad_{\log_g}\right)^k \xi}{\log_g}= \ip{\left(\ad_{\log_g}\right)^{k-1} \xi}{\ad_{\log_g}^*\log_g} =0$ for any $k\ge 1$. Since the 0-order term in $p$ is 1,
    \begin{align*}
        \ip{\d \log_g(\lte{g}\xi)}{\log g}=\ip{\xi}{\log g}
    \end{align*}

    Proof for Eq. \eqref{eqn_log_property_2}: 
    By Cor. \ref{cor_d_log}, we have
    \begin{align*}
        \ip{\d \log_g(\lte{g}\xi)}{\xi}=\ip{p(\ad_{\log_g})\xi}{\xi}
    \end{align*}
    $\ad_{\log_g(g^*)}$ is a linear map from $\g$ to itself. Since the operator is skew-adjoint, it has all eigenvalues pure imagine or $0$. By the assumption that $p(\ad_{\log_g(g^*)})$ converges, we have all the eigenvalues in the interval $(-2\pi i, -2\pi i)$. Since $\Re[p(x i)]=\frac{x \sin(x)}{2(1-\cos(x))}\le 1$  for $x\in (-2\pi, 2\pi)$, the eigenvalues of $p(\ad_{\log_g(g^*)})$ are smaller or equal to 1, which gives us Eq. \eqref{eqn_log_property_2}.
\end{proof}

By Cor. \ref{cor_d_log}, the set $N$ in Eq. \eqref{eqn_N} is also equal to the set where the power series $p(\ad_{\log_g})$ does not converge. We have the following lemma showing this set is zero-measured under the left Haar measure.
\begin{lemma}
\label{lemma_exp_pushforward_absolute_continuous}
By choosing the inner product in Lemma \ref{lemma_ad_self_adjoint}, the pushforward measure of $\d \xi$ on $\g$ by $\exp$ is absolute continuous w.r.t. $\d g$, i.e.,
\begin{align*}
    \d g\ll \exp_\sharp(\d \xi)
\end{align*}
\end{lemma}
\begin{proof}
     The pushforward measure by $\exp:\g\to \G$ is given by $\d g=\det(p^{-1}(\ad_X)) \d \xi$ where $p^{-1}$ is defined in Eq. \ref{eqn_p_inv}. Since we have $\norm{p^{-1}(x i)}\le 1, \forall x\in\mathbb{R}$, we have $\det(p^{-1}(\ad_\xi))$ is upper bounded, which means $\exp$ pushes any zero-measured set in $\g$ to a zero-measured set in $\G$, i.e., $\d g\ll\exp_\sharp(\d \xi)$.
\end{proof}
\begin{lemma}\label{lemma_d_log}
    $p(\ad_{\log g})$ is well defined for $g\in \G$ almost everywhere under measure $\d g$.
\end{lemma}
\begin{proof}[Proof of Lemma \ref{lemma_d_log}]
    By the expression of $p$ in Eq. \eqref{eqn_p}, we have $p(x)$ diverges when $x\neq 0$ and $\exp(-x)=1$, i.e., $x=2k\pi i$ for $k\neq 0$. As a result, $p(\ad_\xi)$ is well defined if $\xi\in P$,
     \begin{equation*}
         P:=\left\{\xi\in\g: \lambda\neq 2k\pi i, \forall \lambda\in \sigma(\ad_\xi), \forall k\neq 0\right\}
     \end{equation*}
     where $\sigma(\ad_\xi)$ denotes the set of all eigenvalues of the linear operator $\ad_\xi$.

    The Lebesgue measure of $P$ is given by
    \begin{align*}
        &\int_\g \mathbbm{1}_P (\xi) \d \xi\\
        &=\operatorname{Area}(\partial B_1(0))\int_{\partial B_1(0)} \int_0^\infty \mathbbm{1}_P (te) t^{d-1}\d  t \d  e
    \end{align*}
    Since we have the operator $\ad_\xi$ is linear in $\xi$, which gives us that $\sigma(\ad_{t\xi})=t\sigma(\ad_{\xi})$ ($\sigma$ is the set of eigenvalues for operators). Using the assumption that $\g$ is finite dimensional, we have for a fixed $e\in \g$, only countable many $t$ makes $te\in P$, which means $\int_0^\infty \mathbbm{1}_P (te) t^{d-1}\d t=0$, this gives $P$ has zero measure. By Lemma \ref{lemma_exp_pushforward_absolute_continuous}, $\exp(P)$ has measure $0$ under left Haar measure on $\G$.
\end{proof}
This lemma proves that we have $\d\log$ well defined almost everywhere on $G$. Later in our proof of convergence using coupling method, the lemma ensures that we can ignore the pairs $(g, \hat{g})$ that $\d\log \hat{g}^{-1} g$ do not exist when we consider the coupling of 2 trajectories in $\G\times \g$ since a zero-measured set will be negligible when taking expectation. See the proof of Thm. \ref{thm_contractivity_SDE_Wrho}.

\section{Absolute continuity w.r.t. $\d g \d \xi$}
As shown in Lemma \ref{lemma_d_log}, the set that $\d\log$ is not well-defined is zero-measured under the left Haar measure. To rule out this pathological part, we need to show that the density for sampling SDE Eq. \eqref{eqn_sampling_SDE} ($\nu_t$) and the numerical scheme Eq. \eqref{eqn_discrete_splitting} ($\tilde{\nu}_k$) are both absolute continuous w.r.t. $\d g \d \xi$, and then $N$ is also zero-measured under $\nu_t$ and $\tilde{\nu}_k$.

The KL divergence between two probability measures $\nu_1$ and $\nu_2$ is defined by
\begin{align*}
    \KL{\nu_1}{\nu_2}:=\int_{\G\times \g} \ln\left(\frac{\nu_1}{\nu_2}\right) d \nu_1
\end{align*}
A direct conclusion is that $\nu_1$ is absolute continuous w.r.t. $\nu_2$ if $\KL{\nu_1}{\nu_2}<\infty$. Using the data processing inequality, we are going to show $\KL{\nu_t}{\nu_*}< \infty$ and $\KL{\tilde{\nu}_k}{\nu_*}< \infty$, which further induces $\nu_t$ and $\tilde{\nu}_k$ is absolute continuous w.r.t. $\nu_*$ (and also $\d g\d \xi$), respectively.

\begin{lemma}[Data processing inequality]
\label{lemma_data_processing}
For any Markov transitional kernel $P$, and any probability distributions $\nu_1$ and $\nu_2$,
\begin{align*}
    \KL{\nu_1 P}{\nu_2 P}\le \KL{\nu_1}{\nu_2}
\end{align*}
\end{lemma}
\begin{corollary}
\label{cor_absolute_continuous}
    Given initial condition $(g_0, \xi_0)$ is absolute continuous w.r.t. $\d g \d \xi$,  $(g_t, \xi_t)$ (evolution of SDE Eq. \eqref{eqn_sampling_SDE}) and $(\tilde{g}_k, \tilde{\xi}_k)$ (evolution of numerical scheme Eq. \eqref{eqn_discrete_splitting} with step size $h$) are both absolute continuous w.r.t. $\d g \d \xi$ for any $t\ge 0$.
\end{corollary}

Although Cor. \ref{cor_absolute_continuous} requires the absolute continuity of $(g_0, \xi_0)$ to ensure the density $(\tilde{g}_k, \tilde{\xi}_k)$ is absolutely continuous in later steps, this can be challenging in practice and we do not require our initialization to be absolute continuous w.r.t. $\d g\d \xi$ in our proposed Algo. \ref{algo_sampling_Lie_group} for the following reason:
\begin{remark}[About absolute continuity of initialization]
\label{rmk_absolute_continuity_initialization}
    In our implementation of Algo. \ref{algo_sampling_Lie_group}, we use the simplest initialization that $g_0$ is an arbitrary point and $\xi_0=0$ but did not require the absolute continuouity as in the proof of convergence rate in Thm. \ref{thm_global_error_W2}. This is because the absolute continuity will be automatic after several iterations: In the first iteration, we first update $\xi_0$ by Eq. \eqref{eqn_splitting_SDE_xi}, which is an Ornstein–Uhlenbeck process and at time $h$, the density of $\tilde{\xi}_1$ is absolute continuous w.r.t. $\d \xi$. As a result, since $g_1=g_0\exp(h\xi_1)$, we have the density of $\tilde{g}_1$ is densities absolute continuous w.r.t. $\d g$ (Lemma \ref{lemma_exp_pushforward_absolute_continuous}). In the second iteration, since noise is added to $\tilde{\xi}_1$ again, we have the joint distribution $(g_1, \xi_2)$ is already absolute continuous w.r.t. $\d g\d \xi$. And by the data processing inequality (Lemma \ref{lemma_data_processing}), we have from then on, the density will always be absolute continuous w.r.t. $\d g\d \xi$. As a result, the requirement for initialization does not matter in practice and will be satisfied in two iterations automatically.
\end{remark}

\section{Proof of Thm. \ref{thm_contractivity_SDE_Wrho}}

\subsection{Semi-martingale decomposition}
\label{sec_semi_martingale_decomposition}
In this section, we use the following shorthand notation
\begin{align*}
&r_t:=r((g_t,\xi_t),(\hat{g}_t ,\hat{\xi}_t))\\
&\rho_t:=\rho ((g_t,\xi_t),(\hat{g}_t ,\hat{\xi}_t))\\
&G_t:=G(\xi_t, \hat{\xi}_t)\\
&\delta_t:=\delta((g_t,\xi_t),(\hat{g}_t ,\hat{\xi}_t))\\
&K_t:=K((g_t,\xi_t),(\hat{g}_t ,\hat{\xi}_t))
\end{align*}
where $(g_t,\xi_t)$ and $(\hat{g}_t ,\hat{\xi}_t)$ are coupled using reflection coupling in Sec. \ref{sec_coupling}.    

The function $K:(\G\times\g)^2\to \mathbb{R}$ is defined by
\begin{align*}
    K&:= \left[c f(r)+\left(\alpha\gamma\norm{Q}-\theta L\gamma^{-1}\norm{Z}+\mathbbm{1}_{\{\hat{g}_t^{-1} g_t\in N\}}\delta\right)  f'_{-}(r)+4\gamma^{-1} \rcfunc(Z,\mu)^2 f''(r)\right]\left(1+\beta\norm{\mu}^2\right)\\
    &+\beta f(r)\left[-2\gamma\norm{\mu}^2+2L\norm{\mu}\norm{Z}+8\gamma \rcfunc(Z,\mu)^2\norm{\mu}^2\right]\\
    &+16\beta\norm{\mu} \rcfunc(Z,\mu)^2 f'_{-}(r)
\end{align*}
where $\delta:(\G\times\g)^2\to \mathbb{R}$ is defined by $\delta:=(\alpha+2)\norm{\mu}+\alpha\gamma\norm{Z}-\alpha\gamma\norm{Q}$ and $\theta\in \mathbb{R}$ is defined by $\theta:=\alpha\gamma^2L^{-1}-1$. This definition for $K$ comes from the calculation of semi-martingale decomposition in Lemma \ref{lemma_martingale_decomposition}.

\begin{proof}[Proof of Lemma \ref{lemma_martingale_decomposition}]
    \textbf{Bound the time derivative for $\norm{Z_t}$: }
    For $\hat{g}_t^{-1} g_t\notin N$, the paths of the process $(Z_t)$ are almost surely continuously differentiable with derivative
    \begin{align*}
        \frac{\d}{\d t}Z_t&=\d \log_{\hat{g}_t^{-1} g_t} \frac{\d}{\d t}g -\d \log_{\hat{g}_t^{-1} g_t} \frac{\d}{\d t} \hat{g}\\
        &=\d \log_{\hat{g}_t^{-1} g_t} (\xi_t-\hat{\xi}_t)
    \end{align*}
    Using the property of $\d\log$ in Cor. \ref{cor_log_property},
    \begin{align*}
        \frac{\d}{\d t}\norm{Z_t}&=\frac{1}{\norm{Z_t}}\ip{Z_t}{\frac{\d}{\d t} Z_t}=\frac{1}{\norm{Z_t}}\ip{Z_t}{\d \log_{\hat{g}^{-1} g} (\xi_t-\hat{\xi}_t)}\\
        &=\frac{1}{\norm{Z_t}}\ip{Z_t}{\mu_t}=\frac{1}{\norm{Z_t}}\ip{Z_t}{-\gamma Z_t+\gamma Q_t}\\
        &\le -\gamma\norm{Z_t}+\gamma\norm{Q_t}
    \end{align*}
    When $\hat{g}_t^{-1} g_t\in N$, we won't have $\frac{\d}{\d t}Z_t$. However, a bound for $\frac{\d}{\d t}\norm{Z_t}$ can still be derived by the triangle inequality:
    \begin{align*}
        \frac{\d}{\d t}\norm{Z_t}\le \norm{\mu_t}
    \end{align*}
    In summary, \begin{align*}
        \frac{\d}{\d t}\norm{Z_t}\le -\gamma\norm{Z_t}+\gamma\norm{Q_t}+\mathbbm{1}_{\{\hat{g}_t^{-1} g_t\in N\}} \left(\norm{\mu_t}+\gamma\norm{Z_t}-\gamma\norm{Q_t}\right)
    \end{align*}
    \textbf{Semimartingale decomposition for $\norm{Q_t}$: }
    Suppose $\norm{Q_t}$ satisfies the following inequality almost surely
    \begin{equation*}
        \norm{Q_t}\le \norm{Q_0} + A_t^Q + \tilde{M}_t^Q\qquad\text{for all } t\ge 0,
    \end{equation*}
    where $(A_t^Q)$ and $(\tilde{M}_t^Q)$ are the absolutely continuous process and the martingale to be determined.

    By our choice of coupling in Sec. \ref{sec_coupling}, $\mu_t$ follows the SDE
    \begin{equation}
    \label{eqn_d_mu}
    \d\mu_t = -\gamma \mu_t \d t-\left(\lt{g_t}\nabla U(g_t)-\lt{\hat{g}_t}\nabla
    U(\hat{g}_t)\right)\d t+\sqrt{8\gamma} \rcfunc(Z_t,\mu_t)e_te_t^\top \d W_t^{\rcfunc}
    \end{equation}
    When $\hat{g}_t^{-1} g_t\in N$, by $\frac{\d}{\d t}\norm{Z_t}\le \norm{\mu_t}\d t$, we have 
    \begin{align*}
        \frac{\d}{\d t}A_t^Q \le 2\norm{\mu_t} \d t+ \gamma^{-1}  \norm{\lt{g_t}\nabla U(g_t)-\lt{\hat{g}_t}\nabla U(\hat{g}_t)}
    \end{align*}
    
    When $\hat{g}_t^{-1} g_t\notin N$, by the fact that $Q=Z+\gamma^{-1}\mu$, The process $(Q_t)$ satisfies the following SDE:
    \begin{equation*}
        \d Q_t=d Z_t-\mu_t \d t- \gamma^{-1}  \left(\lt{g_t}\nabla U(g_t)-\lt{\hat{g}_t}\nabla U(\hat{g}_t)\right) \d t+\sqrt{8\gamma^{-1}} \rcfunc(Z_t,\mu_t)e_te_t^\top \d W_t^{\rcfunc}.
    \end{equation*}
    Notice that by \eqref{eqn_rc},
    the noise coefficient vanishes if $Q_t=0$. Cor. \ref{cor_log_property} gives us
    \begin{align*}
        \ip{\frac{\d}{\d t} Z_t-\mu_t}{Q_t}&=\ip{\d \log_{\hat{g}_t^{-1} g_t} \mu_t-\mu_t}{\log_{\hat{g}_t^{-1} g_t}+\gamma^{-1}\mu_t}\\
        &=\ip{\d \log_{\hat{g}_t^{-1} g_t} \mu_t}{\log_{\hat{g}_t^{-1} g_t}}+\ip{\d \log_{\hat{g}_t^{-1} g_t} \mu_t}{\gamma^{-1}\mu_t}-\ip{\mu_t}{\log_{\hat{g}_t^{-1} g_t}}-\ip{\mu_t}{\gamma^{-1}\mu_t}\\
        &\le 0
    \end{align*}
    and we have in this case
    \begin{align*}
        \frac{\d}{\d t}A_t^Q \le \gamma^{-1} \norm{\lt{g_s}\nabla U(g_s)-\lt{\hat{g}_s}\nabla U(\hat{g}_s)}
    \end{align*}
    
    Therefore, combining the 2 cases and apply It\^o's formula to get
    \begin{align*}
    A_t^Q &=\gamma^{-1} \int_0^t\norm{\lt{g_s}\nabla U(g_s)-\lt{\hat{g}_s}\nabla U(\hat{g}_s)} +2\mathbbm{1}_{\{\hat{g}_s^{-1} g_s\in N\}}\norm{\mu_s} \d s\\
    \tilde{M}_t^Q &=\sqrt{8\gamma^{-1}} \int_0^t \rcfunc(Z_s,\mu_s) e_s^\top \d W_s^{\rcfunc}
    \end{align*}

    Notice that there is no It\^o correction, because $\partial^2_{q/\norm{q}} \norm{q}=0$ for $q\not=0$ in both cases and the noise coefficient vanishes for $Q_t=0$. 

    By the $L$-smoothness of $U$, we have
    \begin{align*}
        \frac{\d}{\d t}A_t^Q\le L\gamma^{-1}\norm{Z_t}+2\mathbbm{1}_{\{\hat{g}_s^{-1} g_s\in N\}}\norm{\mu_s}
    \end{align*}

    \textbf{Semimartingale decomposition for $f(r_t)$: }
    By the definition of $r_t$ in Eq. \eqref{eqn_r}
    \begin{align*}
    r_t :=\alpha\norm{Z_t}+\norm{Q_t}= \norm{Q_0} + \alpha \norm{Z_t} + A_t^Q  +\tilde{M}_t^Q
    \end{align*}
    where $t\mapsto \alpha \norm{Z_t} + A_t^Q$ is almost surely absolutely continuous
    with time derivative upper bounded by
    \begin{align*}
    \frac{\d}{\d t} (\alpha \norm{Z_t} + A_t^Q) &\le -(\alpha\gamma-L\gamma^{-1})\norm{Z_t}+\alpha\gamma\norm{Q_t}+\mathbbm{1}_{\{\hat{g}^{-1} g\in N\}}\delta_t\\
    &=-\theta L\gamma^{-1}\norm{Z_t}+\alpha\gamma\norm{Q_t}+\mathbbm{1}_{\{\hat{g}_t^{-1} g_t\in N\}}\delta_t
    \end{align*}
    Since by assumption, $f$ is concave and $C^2$, we can now apply the Ito-Tanaka formula to $f(r_t)$ and obtain a semi-martingale decomposition
    \begin{align*}
        e^{ct}f(r_t) \le f(r_0) + \tilde{A}_t^f +  \tilde{M}_t^f
    \end{align*}
    with the absolute continuous process $(\tilde{A}_t)$ and the martingale $\tilde{M}_t^f$ given by
    \begin{align*}
        \d\tilde{A}_t^f &= e^{ct}\left(c f(r_t) +\left(\alpha\gamma\norm{Q_t}-\theta L\gamma^{-1}\norm{Z_t}+\mathbbm{1}_{\{\hat{g}_t^{-1} g_t\in N\}}\delta_t \right)  f'_{-}(r_t)+4\gamma^{-1} \rcfunc(Z_t,\mu_t)^2 f''(r_t) \right)  dt\\
        \tilde{M}_t^f&=\sqrt{8\gamma^{-1}} \int_0^t e^{cs} f'_{-}(r_s)
        \rcfunc(Z_s,W_s) e_s^\top dW_s^{\rcfunc}
    \end{align*}
    
\textbf{Semimartingale decomposition for $G_t$: }
By the SDE for $\mu_t$ in Eq. \eqref{eqn_d_mu}, the SDE for $\mu_t$ gives the decomposition for $G_t$:
    \begin{align*}
        G_t=G_0 + A_t^G + \tilde{M}_t^G\qquad\text{for all } t\ge 0,
    \end{align*}
    where $(A_t^G)$ and $(\tilde{M}_t^G)$ are the absolutely continuous process and
    the martingale given by
    \begin{align*}
    &A_t^\mu =\int_0^t -2\gamma\norm{\mu_s}^2-2\mu_s^\top\cdot\left(\lt{g_s}\nabla U(g_s)-\lt{\hat{g}_s}\nabla U(\hat{g}_s)\right) +8\gamma \ip{\mu_s}{e_s}^2 \rcfunc(Z_s,\mu_s)^2 ds, \\
    &\tilde{M}_t^\mu = \sqrt{32\gamma} \int_0^t \ip{\mu_s}{e_s} \rcfunc(Z_s,\mu_s) e_s^\top dW_s^{\rcfunc}.
    \end{align*}
    Using the $L$-smoothness of $U$, we have
    \begin{align*}
        A_t^\mu\le -2\gamma\norm{\mu_t}^2+2L\norm{Z_t}\norm{\mu_t}+8\gamma \rcfunc(Z_t,\mu_t)^2 \norm{\mu_t}^2 
    \end{align*}
    \textbf{Semimartingale decomposition for $e^{ct}\rho_t$: }
    We combine the semimartingale decomposition for $e^{ct} f(r_t)$ and $\norm{\mu_t}$ to have the following semimartingale decomposition for $e^{ct}\rho_t$
    \begin{align*}
        e^{ct}\rho_t= e^{ct}f(r_t)\left(1+\beta\norm{\mu_t}^2\right)=\rho_0+M_t+A_t
    \end{align*}
    where the continuous process $A_t$ satisfies $\frac{\d}{\d t}A_t\le  e^{ct} K_t$.
    \begin{align*}
        K_t&= \left[c f(r_t)+\left(\alpha\gamma\norm{Q_t}-\theta L\gamma^{-1}\norm{Z_t}+\mathbbm{1}_{\{\hat{g}_t^{-1} g_t\in N\}}\delta_t\right)  f'_{-}(r_t)+4\gamma^{-1} \rcfunc(Z_t,\mu_t)^2 f''(r_t)\right]\left(1+\beta\norm{\mu_t}^2\right)\\
        &+\beta f(r_t)\left[-2\gamma\norm{\mu_t}^2+2L\norm{\mu_t}\norm{Z_t}+8\gamma \rcfunc(Z_t,\mu_t)^2\norm{\mu_t}^2\right]\\
        &+16\beta\norm{\mu_t} \rcfunc(Z_t,\mu_t)^2 f'_{-}(r_t)
    \end{align*}
    This gives us the definition for $K:(\G\times\g)^2\to \mathbb{R}$ at the beginning of Sec. \ref{sec_semi_martingale_decomposition}
    \end{proof}

\subsection{Conditions for contractivity}
\label{sec_conditions_for_contractivity}
In the following lemma, we find the sufficient condition for $f$, $\eta$ ($\alpha$), $\beta$, $R$ to ensure $K_t\le 0$. Note that we ignore the part $\mathbbm{1}_{\{\hat{g}^{-1} g\in N\}}\delta_t$ because we has shown it has zero measure at any time (Cor. \ref{cor_absolute_continuous}) and will not affect the contraction under $W_\rho$ in Thm. \ref{thm_contractivity_SDE_Wrho} later.

\begin{lemma}
\label{lemma_conditions}
    \begin{align*}
        K_t\le \epsilon\alpha\gamma  \left(1+\beta \norm{\mu_t}^2\right)
    \end{align*}
    when $\hat{g}^{-1} g\notin N$ if the following conditions Eq. \eqref{eqn_cond_K1}, \eqref{eqn_cond_K4}, \eqref{eqn_cond_K23}are satisfied:
    \begin{equation}
    \label{eqn_cond_K1}
        A_0 f(r)+A_1 r f'_{-}(r)+ A_2 f''(r)\le 0, \quad \forall r\in[0, R]
    \end{equation}
    for
    \begin{equation}
    \label{eqn_A012}
    \begin{cases}
        A_0:=c+\beta \left[6\gamma^3 R^2+2\gamma LRD\right]\\
        A_1:= \alpha \gamma+16\beta\gamma \max\{1, \alpha^{-1}\}\\
        A_2:=4\gamma^{-1}
    \end{cases}
    \end{equation}

\begin{equation}
\label{eqn_cond_K23}
    c -\beta(2\gamma-c-2LD\gamma^{-1}R^{-1})\gamma^2R^2\le 0
\end{equation}

\begin{equation}
\label{eqn_cond_K4}
     \frac{\theta}{1+\theta} \gamma rf'_{-}(r)\ge A_0 f(r), \quad \forall r\in[0, R]
\end{equation}

\end{lemma}
\begin{proof}    
    We prove this by dividing $(\G\times \g)^2$ in to 3 regions (Fig. \ref{fig_coupling_region}).
    
    \textbf{Region I: } $\norm{\mu}\le \gamma R$ and $\norm{Q}\ge\epsilon$ (Reflection coupling) (Corresponding to Condition Eq. \eqref{eqn_cond_K1})
    
    $\rcfunc\equiv 1$ in this case and we have
    \begin{align*}
        K_t&= \left[c f(r_t)+L\gamma^{-1}\left((1+\theta)\norm{Q_t}-\theta \norm{Z_t}\right)  f'_{-}(r_t)+4\gamma^{-1} f''(r_t)\right]\left(1+\beta\norm{\mu_t}^2\right)\\
        &+\beta f(r_t)\left[-2\gamma\norm{\mu_t}^2+2L\norm{\mu_t}\norm{Z_t}+8\gamma \norm{\mu_t}^2\right]+16\beta\norm{\mu_t} f'_{-}(r_t)\\
        &\le \left(c +\beta \left[2L\norm{\mu_t}\norm{Z_t}+6\gamma \norm{\mu_t}^2\right]f(r_t)\right)\left(1+\beta\norm{\mu_t}^2\right)\\
        &+ \left(L\gamma^{-1}\left((1+\theta)\norm{Q_t}-\theta \norm{Z_t}\right)  +16\beta\norm{\mu_t}\right)f'_{-}(r_t)\left(1+\beta\norm{\mu_t}^2\right)\\
        &+4\gamma^{-1} f''(r_t)\left(1+\beta\norm{\mu_t}^2\right)
    \end{align*}
    $\norm{Z}\le D$ and $\norm{\mu}\le \gamma R$ give us
    \begin{align*}
        c +\beta \left[2L\norm{\mu_t}\norm{Z_t}+6\gamma \norm{\mu_t}^2\right]\le c+\beta \left[2L\gamma RD+6\gamma (\gamma R)^2\right]
    \end{align*}
    By the fact $\norm{\mu}\le \gamma \max\{1, \alpha^{-1}\}r$ and $\norm{Q}\le r$
    \begin{align*}
        &L\gamma^{-1}\left((1+\theta)\norm{Q}-\theta \norm{Z}\right)+16\beta\norm{\mu} \\
        &\le L\gamma^{-1}\norm{Q} +16\beta\norm{\mu}\\
        &\le L\gamma^{-1} r+16\beta\gamma \max\{1, \alpha^{-1}\}r
    \end{align*}
    As a result, by setting $A_0$, $A_1$, $A_2$ as in Eq. \eqref{eqn_A012}, we have 
    \begin{align*}
        K_t \le \left[A_0 f(r)+A_1 r_t f'_{-}(r_t)+ A_2 f''(r_t)\right]\left(1+\beta\norm{\mu}^2\right)
    \end{align*}
    When Eq. \eqref{eqn_cond_K1} is satisfied, $K_t\le 0$ in this region.

    \textbf{Region II \& III: } $\norm{\mu}\ge \gamma R+\epsilon$ (Synchronous coupling) (Corresponding to Condition Eq. \eqref{eqn_cond_K23})
    
    In the case, we have $r\ge R$, which means $f'_{-}(r)=f''(r)=0$.
    \begin{align*}
        K_t&=c f(r_t)\left(1+\beta\norm{\mu_t}^2\right)+\beta f(r_t)\left[-2\gamma\norm{\mu_t}^2+2L\norm{\mu_t}\norm{Z_t}\right]\\
        &\le f(r_t)\left[c -\beta(2\gamma-c-2LD\gamma^{-1}R^{-1})\norm{\mu_t}^2\right]
    \end{align*}
    which is non-negative given condition Eq. \eqref{eqn_cond_K23}. 
    
    \textbf{Region IV: } $\norm{\mu}\le \gamma R$ and $\norm{Q}<\epsilon$ or $\gamma R<\norm{\mu}< \gamma R+\epsilon$ (Mix between synchronous and reflection coupling) (Corresponding to Condition Eq. \eqref{eqn_cond_K4})

    $0\le \rcfunc\le 1$ in this case and we have
    \begin{align*}
        K_t&\le \rcfunc(Z_t, \mu_t)^2\left[A_0 f(r_t)+A_1 r_t f'_{-}(r_t)+ A_2 f''(r_t)\right]\left(1+\beta\norm{\mu_t}^2\right)\\
        &+(1-\rcfunc(Z_t, \mu_t)^2)\left[\left(c-2\beta\gamma\norm{\mu_t}^2+2\beta L\norm{\mu_t}\norm{Z_t}\right) f(r_t)+(\alpha\gamma\norm{Q_t}-\theta L\gamma^{-1}\norm{Z_t})  f'_{-}(r_t)\right]\left(1+\beta\norm{\mu_t}^2\right)\\
        &\le \left[\left(c-2\beta\gamma\norm{\mu_t}^2+2\beta L\norm{\mu_t}\norm{Z_t}\right) f(r_t)+(\alpha\gamma\norm{Q_t}-\theta L\gamma^{-1}\norm{Z_t})  f'_{-}(r_t)\right]\left(1+\beta\norm{\mu_t}^2\right)\\
        &\le \left[A_0 f(r_t)+(\alpha\gamma\epsilon-\theta L\gamma^{-1}\norm{Z_t})  f'_{-}(r_t)\right]\left(1+\beta\norm{\mu_t}^2\right)\\
        &\le \left[A_0 f(r_t)-\alpha^{-1}\theta L\gamma^{-1}r_t  f'_{-}(r_t)\right]\left(1+\beta\norm{\mu_t}^2\right)+\alpha\gamma\epsilon  f'_{-}(r_t)\left(1+\beta\norm{\mu_t}^2\right)
    \end{align*}
    Since $f'_{-}\le 1$ gives the desired result.
    
\end{proof}

\subsection{Choose the parameters}
\label{sec_choose_parameters}
In order to make sure the conditions in Lemma \ref{lemma_conditions} satisfied,  the parameters $\alpha$ ($\theta$), $\beta$, $R$, the contraction rate $c$ and the function $f$ are needed to be carefully selected. The idea of choosing parameters here is inspired by \cite{eberle2016reflection, eberle2019couplings}.
\subsubsection{Choose $f$}
\label{sec_choose_parameter_f}
At first glance, choosing $f$ seems to be the most difficult part since the condition Eq. \eqref{eqn_cond_K1} is not very intuitive and we do not have a parametrization for function $f$. So, we start by giving $f$ an explicit expression heuristically.

Denote $\varphi$ as 
\begin{align*}
    \varphi(s):=\exp\left(-\frac{A_1}{2A_2} s^2\right)
\end{align*}
i.e., $\varphi$ is the solution of ODE $A_1 r\varphi(r)+A_2\varphi'(r)=0$. An intuition for this equation on $\varphi$ is by imaging $\varphi=f'$ in  Eq. \eqref{eqn_cond_K1} with the term $A_0 f$ disregarded.

Since the ODE for $\varphi$ omitted the positive $A_0 f$ term, $f(r)$ cannot be simply set as $\int_0^r\varphi(s)\d s$. Instead, a correction term $\psi$ is introduced and $f$ has the following form:
\begin{align}
\label{eqn_f}
    f(r)=\int_0^{\min\{r, R\}}\varphi (s) \psi(s) \d s 
\end{align}
$\psi \le 1$ is a function to be determined later. By denoting 
\begin{equation}
\label{eqn_Phi}
    \Phi(r):=\int_0^r \varphi(s) ds
\end{equation}
we have $f\le \Phi$ and when $r\le R$,
\begin{align*}
    A_0 f+A_1 r f'_{-}+ A_2 f'' &=A_0 f+A_1 r (\varphi \psi)+ A_2(\varphi \psi'+\varphi' \psi)\\
    &=A_0 f+A_2\varphi \psi'\\
    &\le A_0 \Phi+A_2 \varphi \psi'
\end{align*}
This shows that $A_0 \Phi(r)+A_2 \varphi(r) \psi'(r)\le 0$ is a sufficient condition for Eq. \eqref{eqn_cond_K1}, which inspires us to choose $\psi$ as
\begin{align}
\label{eqn_psi}
    \psi(r)=1 -\frac {A_0}{A_2}\int_0^r\Phi (s)\varphi (s)^{-1} ds, \quad \forall r\in[0, R]
\end{align}
Now we have an expression for $f$, we need to find a set of parameters $\alpha$($\theta$), $\beta$, $R$ and $c$ s.t. Eq. \eqref{eqn_cond_K23}, \eqref{eqn_cond_K4} is satisfied.
\subsubsection{Choose $\alpha$($\theta$), $\beta$, $R$}
\label{sec_choose_parameter_theta_beta_R}
First, we focus on condition Eq. \eqref{eqn_cond_K23}. We set $\lambda c=\beta (6\gamma^3 R^2+2\gamma LRD)$ where $\lambda$ is a constant to be determined, i.e., $A_0$ is assumed to be of the same order as $c$. Under such assumption, Eq. \eqref{eqn_cond_K23} can be simplified as
\begin{align*}
    &c -\beta(2\gamma-c-2LD\gamma^{-1}R^{-1})\gamma^2R^2\le 0\\
    &\Leftrightarrow \beta (6\gamma^3 R^2+2\gamma LRD) \le \lambda\beta(2\gamma-c-2LD\gamma^{-1}R^{-1})\gamma^2R^2\\
    &\Leftrightarrow 6 \gamma^2R+2LD \le \lambda(2\gamma^2 R-c\gamma R-2LD)\\
    &\Leftrightarrow [2(\lambda-3)\gamma-c\lambda]\gamma R\ge 2(1+\lambda)LD
\end{align*}

The left-hand side must be positive, which means $\lambda>3$. After simply choosing $\lambda=4$ and assuming 
\begin{equation}
\label{eqn_c_assumption_1}
    c\le \gamma/4
\end{equation}
we can choose $R$ as
\begin{equation}
\label{eqn_choose_R}
    R=10 LD\gamma^{-2}
\end{equation}
to make Eq. \eqref{eqn_cond_K23} satisfied. After $R$ is chosen, solving $4c=\beta(6\gamma^3 R^2+2\gamma LRD)$ gives an explicit expression for $\beta$:
\begin{equation}
\label{eqn_choose_beta}
\beta=\frac{1}{155}c\gamma L^{-2} D^{-2}
\end{equation}

Before we proceed, we give $\varphi$ a lower bound by adding more constraints to the parameters. The benefit is that we will have a better estimate for the conditions depending on $\varphi$ ($\Phi$) and $f$, i.e., Eq. \eqref{eqn_cond_K4}. $\frac{A_1}{2A_2}=\frac{\alpha}{8}\gamma^2+2\beta\gamma^2\max\{1, \alpha^{-1}\}$ and we have 
\begin{align*}
    \varphi(s)=\exp\left(-\frac{Ls^2}{8}(1+\theta)-2\gamma^2\beta\max\{1, \alpha^{-1}\}s^2\right)
\end{align*}
And we want to make sure
\begin{equation}
\label{eqn_varphi_lower_bound}
    \varphi(s)\ge \exp\left(-2-\frac{Ls^2}{8}\right), \quad\forall s\in[0, R]
\end{equation}
Since this only need to be satisfied on $[0, R]$, it can be achieved by ensuring $2\beta \gamma^2 R^2\max\{1, \alpha^{-1}\}\le 1$ and $\theta L R^2/8\le 1$, which further leads to the following choice of $\theta$ 
\begin{equation}
\label{eqn_choose_theta}
\theta=8 L^{-1} R^{-2}=0.08L^{-3} D^{-2}\gamma^4
\end{equation}
and one more condition on $c$:
\begin{equation}
\label{eqn_c_assumption_2}
    c\le \frac{31}{40}\min\{\gamma, (1+\theta) L\gamma^{-1}\}
\end{equation}

To summarize, comparing with the conditions in Lemma \ref{lemma_conditions} (Eq. \eqref{eqn_cond_K1}, \eqref{eqn_cond_K23}, \eqref{eqn_cond_K4}), we made some artificial conditions and have a set of more explicit choice of constants (Eq. \eqref{eqn_choose_R}, \eqref{eqn_choose_beta}, \eqref{eqn_choose_theta}) with conditions (Eq. \eqref{eqn_cond_K4}, \eqref{eqn_c_assumption_1}, \eqref{eqn_c_assumption_2}). 
\subsubsection{Choose $c$}
Now, the last part is to choose a suitable parameter $c$. We use all the choice of parameters mentioned earlier in Sec. \ref{sec_choose_parameter_f} and \ref{sec_choose_parameter_theta_beta_R}. Given Lemma \ref{lemma_conditions}, we only need to check Eq. \eqref{eqn_cond_K1}, \eqref{eqn_cond_K23}, \eqref{eqn_cond_K4} are satisfied.
\begin{lemma}
\label{lemma_parameters}
By our choice of $f$ (Eq. \eqref{eqn_f}), $\theta$ (Eq. \eqref{eqn_choose_theta}), $\beta$ (Eq. \eqref{eqn_choose_beta}), $R$ (Eq. \eqref{eqn_choose_R}), we have conditions in Lemma \ref{lemma_conditions} (\eqref{eqn_cond_K1}, \eqref{eqn_cond_K23}, \eqref{eqn_cond_K4}) are satisfied for 
\begin{equation}
\label{eqn_c_upper_bound}
c\le c_*:=\min\left\{\frac{\gamma}{4}, \frac{31}{40}(1+\theta)L\gamma^{-1}, \frac{L^{3/2}R}{5\sqrt{2\pi}\gamma}e^{-\frac{LR^2}{8}}, \frac{\theta \gamma}{5(1+\theta)}\min\left\{e^{-\frac{5}{2}}, \frac{\sqrt{L}R}{\sqrt{2\pi}}e^{-\frac{LR^2}{8}}\right\}\right\}
\end{equation}
\end{lemma}

\begin{proof}[Proof of Lemma \ref{lemma_parameters}]
    Condition Eq. \eqref{eqn_cond_K1} is satisfied by our choice of $f$ in Eq. \eqref{eqn_f}. 
    
    Eq. \eqref{eqn_cond_K23} is satisfied by Eq. \eqref{eqn_c_assumption_1}, with extra condition Eq. \eqref{eqn_c_assumption_1}.

    The rest of the proof will be devoted to finding an appropriate $c$ for Eq. \eqref{eqn_cond_K4}.

    We first ensure a lower bound by $\psi\ge 1/2$. By the explicit expression of $\psi$ in Eq. \eqref{eqn_psi}, this is equal to
    \begin{equation}
    \label{eqn_cond_g}
        \frac{2A_2}{A_0}\le\int_0^{R}\Phi (s)\varphi (s)^{-1}\d s
    \end{equation}
    Our choice of $\beta$ gives $A_0=5c$. Consequently, Eq. \eqref{eqn_cond_g} is satisfied when
    \begin{equation*}
        c\le \frac{8}{5}\gamma^{-1}/\int_0^{R}\Phi (s)\varphi (s)^{-1}\d s
    \end{equation*}
    
    Since our the property Eq. \eqref{eqn_varphi_lower_bound} is guaranteed by our choice of $\beta$ and extra condition on $c$ in Eq. \eqref{eqn_c_assumption_2}, we have the following upper bound of $\Phi$ from the definition of $\Phi$ in Eq. \eqref{eqn_Phi}
    \begin{equation*}
        \Phi\le \int_0^\infty \exp(-2-Ls^2/8) \d s=e^{-2}\sqrt{2\pi/L}
    \end{equation*}
    This upper bound of $\Phi$ gives an upper bound of $\int_0^{R}\Phi (s)\varphi (s)^{-1}\d s$:
    \begin{align*}
        &\int_0^{R}\Phi (s)\varphi (s)^{-1}\d s\\
        &\le \Phi (R)\int_0^{R}\varphi (s)^{-1}\d s\\
        &\le \Phi (R)e^2\int_0^{R}\exp(Ls^2/8)\d s\\
        &\le \Phi (R)e^2\frac{8}{LR}\exp(LR^2/8)\\
        &\le \sqrt{2\pi/L} \frac{8}{LR}\exp(LR^2/8)
    \end{align*}
    
    In the end, an sufficient condition for Eq. \eqref{eqn_cond_g} is 
    \begin{equation}
    \label{eqn_c_assumption_3}
        c\le \frac{L^{3/2}R}{5\sqrt{2\pi}\gamma}\exp(-LR^2/8)
    \end{equation}
    By our choice of $\theta$ in Eq. \eqref{eqn_choose_theta} and apply condition Eq. \eqref{eqn_c_assumption_2}, we have Eq. \eqref{eqn_varphi_lower_bound}.

    Next, we are ready to make Eq. \eqref{eqn_cond_K4} satisfied. By the expression for $f$, we have $f\le \Phi$. Using the property $\psi\ge1/2$ guaranteed by Eq. \eqref{eqn_c_assumption_3}, we have $f'=\varphi \psi\ge \frac{1}{2}\varphi$. Together with $A_0=5c$ and Eq. \eqref{eqn_varphi_lower_bound}, we know Eq. \eqref{eqn_cond_K4} is satisfied given
    \begin{align*}
        \frac{5c(1+\theta)}{\theta \gamma}\le\inf_{r\in [0, R]} \frac{r\varphi(r)}{\Phi(r)}
    \end{align*}
    Now we estimate a lower bound for $\inf_{r\in [0, R]} \frac{r\varphi(r)}{\Phi(r)}$. First we notice that
    \begin{align*}
        \frac{r\varphi(r)}{\Phi(r)}\ge \frac{r\exp(-2-Lr^2/8)}{\int_0^r \exp(-2-Ls^2/8)ds}=\frac{r\exp(-Lr^2/8)}{\int_0^r \exp(-Ls^2/8) ds}
    \end{align*}
    For $r\ge \frac{2}{\sqrt{L}}$, we have $\frac{\d}{\d r}r\exp(-Lr^2/8)=(1-Lr^2/4)\exp(-Lr^2/8)\le 0 $ and $\frac{r\varphi(r)}{\Phi(r)}\ge \frac{R\exp(-LR^2/8)}{\int_0^R \exp(-Ls^2/8)ds}\ge R\exp(-LR^2/8)/\sqrt{2\pi/L}$.

    For $r\le \frac{2}{\sqrt{L}}$, we have $r/\Phi\le 1$ and $\frac{r\varphi(r)}{\Phi(r)}\ge \varphi(r)\ge \exp(-5/2)$.
    
    As a result, a sufficient condition for Eq. \eqref{eqn_cond_K4} is
    \begin{align}
    \label{eqn_c_assumption_4}
        c\le\frac{\theta \gamma}{5(1+\theta)}\min\left\{\exp(-5/2), R\exp(-LR^2/8)/\sqrt{2\pi/L}\right\}
    \end{align}

Since we added extra conditions on $c$ in Eq. \eqref{eqn_c_assumption_1}, \eqref{eqn_c_assumption_2}), a summary of all the conditions on $c$, Eq. \eqref{eqn_c_assumption_1}, \eqref{eqn_c_assumption_2}, \eqref{eqn_c_assumption_3}, \eqref{eqn_c_assumption_4} give us the desired result.
\end{proof}

\subsection{More discussion on the convergence rate $c$}
\label{sec_discussion_c}
Here we discuss the order of $c_*$, the upper bound of convergence rate defined in Eq. \eqref{eqn_c_upper_bound}. By denoting 
\begin{align}
\label{eqn_Lambda}
    \Lambda:=\frac{LR^2}{8}\sim L^3 D^2 \gamma^{-4}
\end{align}
We have an estimate of the order of $c_*$:
\begin{equation}
\label{eqn_c_order}
c_*\sim \min\left\{\gamma, L\gamma^{-1}, L\gamma^{-1}\sqrt{\Lambda}e^{-\Lambda}, \gamma\Lambda^{-1}, \gamma \Lambda^{-\frac{1}{2}} e^{-\Lambda}\right\}
\end{equation}
Notice that the inner product in Lemma \ref{lemma_ad_self_adjoint} can be rescaled $\ip{\cdot}{\cdot}\to a\ip{\cdot}{\cdot}$, and in this case $L\to a^{-2}L$ and $D\to aD$, but the value $LD^2$ will remain the same. As a result, we consider dependence of the convergence rate $c_*$ on $\gamma$ and $LD^2$ when $LD^2\to \infty$:
\begin{itemize}
    \item For the critical choice $\gamma\sim \sqrt{L}$, we have $c_*\sim D^{-\frac{1}{2}}e^{-LD^2}$. This is also the common choice in the  Euclidean case.
    \item For the underdamped choice $\gamma\lesssim \sqrt{L}$, we have $c_*\sim \gamma \min\left\{\Lambda^{-1}, \Lambda^{-\frac{1}{2}} e^{-\Lambda}\right\}$, which is slower than the kinetic choice when $LD^2\to \infty$.
    \item For the overdamped choice $\gamma\gtrsim \sqrt{L}$, it is hard to simplify the expression for $c_*$. One may suggest $\gamma\sim(L^3D^2)^{\frac{1}{4}}$ to make $\Lambda\sim 1$. In this case, we have $c_*\sim L\gamma^{-1}\sim D^{\frac{1}{2}}L^{-\frac{1}{4}}$. However, the research in Euclidean kinetic Langevin suggest this will lead to the requirement of small step size and will not remove the exponential decay of convergence rate when $D$ is large. 
\end{itemize}

In summary, our recommendation is $\gamma\sim\sqrt{L}$, and we have the convergence rate $c$ decays exponentially when $D$ (diameter of Lie group) grows. One may worry that our rate is slow, but it actually should be slow because we only assumed compactness and smooth log density, and metastability can exist: without convexity conditions on $U$, the potential barriers between the potential wells of the non-convex potential function $U$ stops the Langevin dynamics from mixing fast. We can see that by comparing, for example, our rate with some existing results by Gaussian mixture models (GMM).

We consider the mixture of 2 Gaussians $N(x, \sigma I)$ and $N(y, \sigma I)$ with equal weight in Euclidean space. \cite{schlichting2019poincare} provides an estimate of LSI constant $\varrho$ for this GMM model showing an exponential $\varrho\gtrsim \exp(-\norm{x-y}^2/\sigma)$, and \cite{ma2021there} showing a convergence rate $\varrho$. This means the result \cite{ma2021there} has a convergence rate exponentially decaying w.r.t. the distance between 2 Gaussian models. In other words, exponential long time is required to jump between potential well when they are far, and since we have no convexity assumption on potential function, this slow rate cannot be avoided.

Our choice of $\gamma$ and the quantification of convergence rate are consistent with the Euclidean case \citep[Sec. 2.6]{eberle2019couplings}.

\section{Proof for Thm. \ref{thm_contractivity_SDE_Wrho}}
\begin{proof}[Proof for Thm. \ref{thm_contractivity_SDE_Wrho}]
    By Cor. \ref{cor_absolute_continuous}, we have both the density of $(g_t, \xi_t)$ and $(\hat{g}_t, \hat{\xi}_t)$ are absolute continuous w.r.t. $\d g \d \xi$ for any $t\in [0, \infty)$, which gives us $\mathbb{E}_{(g_t, \xi_t, \hat{g}_t, \hat{\xi}_t)} \mathbbm{1}_{\{\hat{g}_t^{-1} g_t\in N\}}\delta_t\equiv 0, \forall t\in[0, \infty)$.

    As a result, by taking expectation to Eq. \eqref{eqn_decomposition_rho} and choosing the parameters $\alpha$, $\beta$ and $R$ as stated in Lemma \ref{lemma_parameters}, we have
    \begin{align*}
        \mathbb{E}\rho_t\le e^{-ct}\mathbb{E}\rho_0+\epsilon\alpha\gamma  \int_0^t\mathbb{E}\left(1+\beta \norm{\mu_s}^2\right) \d s
    \end{align*}
    By the fact that $\mathbb{E}\norm{\mu_s}^2$ is upper bounded by $4C_2$ (Lemma \ref{lemma_bound_xi_power}) since $\norm{\mu_s}^2\le 2\norm{\xi_s}^2+2\norm{\hat{\xi}_s}^2$, we can take the limit when $\epsilon\to 0$ and have
    \begin{align*}
        \mathbb{E}\rho_t\le e^{-ct}\mathbb{E}\rho_0
    \end{align*}
    By choosing the optimal joint distribution for initial distribution $\nu_0=\Law(g_0, \xi_0)$ and $\hat{\nu}_0=\Law(\hat{g}_0, \hat{\xi}_0)$, we have 
    \begin{align*}
        &W_\rho(\nu_0, \hat{\nu}_0)\\
        &=\mathbb{E}\rho((g_0, \xi_0), (\hat{g}_0, \hat{\xi}_0))\\
        &\ge e^{ct}\mathbb{E}\rho((g_t, \xi_t), (\hat{g}_t, \hat{\xi}_t))\\
        &\ge e^{ct}W_\rho(\Law(g_t, \xi_t), \Law(\hat{g}_t, \hat{\xi}_t))\\
        &= e^{ct} W_\rho(\nu_t, \hat{\nu}_t)
    \end{align*}
\end{proof}

\begin{proof}[Proof of Lemma \ref{lemma_equivalence_distance}]
We consider 2 cases separately: 1)$\norm{\xi-\hat{\xi}}\ge 2D$; 2) $\norm{\xi-\hat{\xi}}< 2D$, where $D$ is the diameter of the Lie group.

\textbf{Case 1}: $\norm{\xi-\hat{\xi}}\ge 2D$. In this case, we have 
\begin{align*}
    \norm{\xi-\hat{\xi}}^2\ge \frac{1}{2}\norm{\xi-\hat{\xi}}^2+\frac{1}{2}D^2\ge\frac{1}{2}\norm{\xi-\hat{\xi}}^2+\frac{1}{2}d^2(g, \hat{g})
\end{align*}
The definition of $r$ in Eq. \eqref{eqn_r} gives 
\begin{align*}
&r((g, \xi), (\hat{g}, \hat{\xi}))\\
&\ge  \min\{\alpha, 1\} \left(d(g,\hat{g})+ \norm{\log \hat{g}^{-1}g+ \gamma^{-1}(\hat{\xi}-\xi)}\right)\\
&\ge \min\{\alpha, 1\}\gamma^{-1}\norm{\hat{\xi}-\xi}\\
&\ge 2\min\{\alpha, 1\}\gamma^{-1}D
\end{align*}
which infers by the definition of $\rho$, 
\begin{align*}
&\rho((g, \xi), (\hat{g}, \hat{\xi}))\\
&\ge f(2\min\{\alpha, 1\}\gamma^{-1}D)(1+\beta d^2((g, \xi), (\hat{g}, \hat{\xi}))/2)\\
&\ge\frac{\beta}{2}f(2\min\{\alpha, 1\}\gamma^{-1}D)d^2((g, \xi), (\hat{g}, \hat{\xi}))
\end{align*}

\textbf{Case 2}: $\norm{\xi-\hat{\xi}}< 2D$. In this case, $r\le (\alpha+1+2\gamma^{-1})D$. By the concavity of $f$, we have 
\begin{align*}
f(r)\ge \frac{f((\alpha+1+2\gamma^{-1})D)}{(\alpha+1+2\gamma^{-1})D} r\ge \frac{f((\alpha+1+2\gamma^{-1})D)}{((\alpha+1+\gamma^{-1})^2D^2} r^2
\end{align*}

Since 
\begin{align*}
&r((g, \xi), (\hat{g}, \hat{\xi}))\\
&=\frac{\alpha}{2}d(g, \hat{g})+\frac{\alpha}{2}d(g, \hat{g})+\norm{\log \hat{g}^{-1}g+\gamma^{-1}(\hat{\xi}-\xi)}\\
&\ge \frac{\alpha}{2}d(g, \hat{g})+\gamma^{-1}\min\{1, \alpha/2\}\norm{\xi-\hat{\xi}}
\end{align*}
we have 
\begin{align*}
\rho\ge f(r)\ge \frac{f((\alpha+1+2\gamma^{-1})D)}{(\alpha+1+\gamma^{-1})^2D^2} r^2\ge \frac{f((\alpha+1+2\gamma^{-1})D)}{(\alpha+1+\gamma^{-1})^2D^2} \min\left\{\frac{\alpha^2}{4}, \frac{1}{\gamma^2}, \frac{\alpha^2}{4\gamma^2}\right\} d^2
\end{align*}

By defining $C_\rho$ as 
\begin{equation}
\label{eqn_C_rho}
    C_\rho=\min\left\{\frac{\beta}{2}f(2\min\{\alpha, 1\}\gamma^{-1}D), \frac{f((\alpha+1+2\gamma^{-1})D)}{(\alpha+1+\gamma^{-1})^2D^2} \min\left\{\frac{\alpha^2}{4}, \frac{1}{\gamma^2}, \frac{\alpha^2}{4\gamma^2}\right\}\right\}
\end{equation}
we have $C_\rho d^2\le \rho$.
\end{proof}

From Lemma \ref{lemma_equivalence_distance}, we directly have the following corollary, which bounds $W_2$ by $W_\rho$, and further gives us convergence under $W_2$ (Thm. \ref{thm_SDE_error_W2}).
\begin{corollary}[Control of $W_2$ by $W_\rho$]
\label{cor_equivalence_W2}
For any distributions $\nu$ and $\hat{\nu}$ on $\G\times \g$, we have
    \begin{align*}
        C_\rho W_2^2(\nu, \hat{\nu})\le W_\rho(\nu, \hat{\nu})
    \end{align*}
\end{corollary}
\begin{proof}[Proof of Cor. \ref{cor_equivalence_W2}]
    Suppose $\pi\in\Pi(\nu, \hat{\nu})$ is the optimal coupling between $\nu$ and $\hat{\nu}$ under $\rho$, i.e., $\int\rho(x,y)d\gamma=W_\rho(\nu, \hat{\nu})$.  Lemma \ref{lemma_equivalence_distance} gives
    \begin{align*}
        W_\rho(\nu, \hat{\nu})&=\int\rho\d\pi\\
        &\ge \int C_\rho d^2\d\pi\\
        &\ge C_\rho W_2^2(\nu, \hat{\nu})
    \end{align*}
\end{proof}

\begin{proof}[Proof of Thm .\ref{thm_SDE_error_W2}]
    This is a direct corollary for Thm. \ref{thm_contractivity_SDE_Wrho} and Cor. \ref{cor_equivalence_W2}.
\end{proof}

\section{More details about Sec. \ref{sec_convergence_splitting}}
\subsection{More discussion on the splitting discretization}
\label{sec_more_about_splitting_discretization}
Splitting discretization is less commonly used when discretizing kinetic Langevin SDE in the Euclidean space. This is because another discretization, which we will refer to as the `exponential integrator discretization' (Eq.\ref{eqn_exp_integrator_Euclidean}) is easier to analyze. However, that exponential integrator discretization works in Euclidean space but not on Lie groups.

More precisely, in the existing literature studying kinetic Langevin in Euclidean spaces, time discretization of SDE Eq. \eqref{eqn_sampling_SDE_Euclidean} is commonly done by the explicit solution of
    \begin{equation}
    \label{eqn_exp_integrator_Euclidean}
    \begin{cases}
        \dot{q}=p\\
        dp=-\gamma p dt-\nabla U(q_0)dt+\sqrt{2\gamma}dW
    \end{cases}
    \end{equation}
    i.e., by viewing $\nabla U$ as constant and using a closed-form solution for the rest. The only numerical error it introduces is the error in the gradient of the potential function. This special discretization also is of key importance in many proofs, e.g., \cite{ma2021there, altschuler2023faster}.

    However, although in the Euclidean case, it is a linear SDE and admits a closed-form solution, in the Lie group case it becomes
    \begin{equation*}
    \begin{cases}
        \dot{g}=\lte{g} \xi \\
        \d\xi=-\gamma \xi dt-\lt{g}\nabla U(g_0)dt+\sqrt{2\gamma}dW
    \end{cases}
    \end{equation*}
    It is no longer a linear SDE due to the involvement of the operator $\lte{g}$, and this SDE does not have a closed-form solution, hence not implementable as a numerical algorithm. 
    
    To see this more clearly, let's consider the matrix group case in Example \ref{example_matrix_lie_group}. In this case, $\lte{g} \xi=g\xi$, where both $g$ and $\xi$ are $n\times n$ matrices and $g\xi$ is the matrix multiplication. As a result, even though the $\xi$ equation can be analytically solved, resulting in a precise expression $\xi$ as a function of time, the equation $\dot{g}=g\xi$ is a non-autonomous linear system with coefficient ($\xi$) not necessarily commuting at different times, and hence $g$ does not admit a closed-form solution.
    
    Due to the nonlinearity of the $g\xi$ term, we use the splitting discretization Eq. \eqref{eqn_discrete_splitting} instead, where we step $g$ and $\xi$ separately. In that case, when we evolve $\dot{g}=g\xi$ in the matrix group, we view $\xi$ as constant and that is a linear equation in $g$, which admits a closed-form solution. However, extra numerical errors other than the error in the gradient of potential may be introduced compared to the exponential integrator discretization Eq. \eqref{eqn_exp_integrator_Euclidean}, which makes our analysis more non-trivial.

    At the end of this section, we give proof for the property of structure-preserving for splitting discretization Thm. \ref{thm_structure_preserving}.
    
\begin{proof}[of Thm. \ref{thm_structure_preserving}].
    Our numerical scheme is given by alternatively evolving the two SDEs. Suppose $(g_k, \xi_k)$ preserves the manifold structure, i.e., $g\in \G$ and $\xi\in \g$. For the first step when updating $\xi$ by evolving Eq.\ref{eqn_splitting_SDE_xi}, since it is an SDE in $\g$, which is a linear space, we have $\xi_{k+1}\in \g$. For the second step when updating $g$ by evolving Eq. \eqref{eqn_splitting_SDE_g}, since the group exponential is always in the group, i.e., $\exp(h\xi_{k+1})\in \G$, we have $g_{k+1}=g_k \exp(h\xi_{k+1})$ is in $\G$ since the group is closed under multiplication. 

    By induction, We have if $g_0\in \G$ and $\xi_0\in \g$, the group structure is preserved for the whole trajectory, i.e., $g_k\in \G$ and $\xi_k\in \g$ for any $k$.
\end{proof}

\begin{remark}[Example of structure preserving on $\mathsf{SO}(n)$]
\label{rmk_example_structure_preserving}
    When we use matrix representation for $\mathsf{SO}(n)$ (Example \ref{example_SOn}),  we have the Lie group structure is  $X^\top X = I$, $\xi + \xi^\top=0$. As a result, Thm. \ref{thm_structure_preserving} tell us $X_k^\top X_k = I$, $\xi_k + \xi_k^\top=0$ for all $k$ under any step size $h$.
\end{remark}

\begin{remark}[Comparison with \citep{cheng2022efficient}]
    In a brilliant paper by \cite{cheng2022efficient}, a quantification of the discretization of the Langevin on the manifold (without momentum) is proposed. However, our Lie group has the group structure and our error bound can be more explicit. What's more, our discretization is not simply geometric Euler-Maruyama discretization but a two-stage splitting. How to expand the bound for numerical error for exponential integrator in their work to splitting discretization is unclear yet.
\end{remark}

\subsection{Proof of Thm. \ref{thm_local_error_W2}}
The proof of Thm. \ref{thm_local_error_W2} will be in several steps: 1) we find a shadow SDE whose distribution at $t$ is the distribution of our numerical method evolve one step with step size $t$ (Lemma \ref{lemma_SDE_splitting}); 2) we upper bound some terms in Lemma \ref{lemma_d2exp_upper_bound}, \ref{lemma_d2_ito_correction_upper_bound}, \ref{lemma_dexp_upper_bound}, \ref{lemma_bound_xi_power}; 3) using those estimations, we derive an ODE whose solution at time $t$ is the upper bound of one step numerical error of our numerical scheme (Thm. \ref{thm_local_error_W2}).
\begin{lemma}[Exact SDE for splitting scheme]
\label{lemma_SDE_splitting}
Denote $(\tilde{g}_h, \tilde{\xi}_h)$ as the one-step evolution of our splitting numerical scheme Eq. \eqref{eqn_discrete_splitting} with step size $t$ starting from $(g_0, \xi_0)$. Then $(\tilde{g}_h, \tilde{\xi}_h)$ is the solution of the following SDE at time $h$:
\begin{equation}
\label{eqn_SDE_splitting}
\begin{cases}
    \d\tilde{g}_t=\lte{\tilde{g}_t}\left(\tilde{\xi}_t \d t+t\frac{1-e^{\ad_{t\tilde{\xi}_t}}}{\ad_{t\tilde{\xi}_t}} \left(-\gamma\tilde{\xi}_t \d t-\lt{g_0}\nabla U(\tilde{g}_0) \d t+\sqrt{2\gamma}\d W_t\right)+\gamma t^2\sum_{ij}\d^2\exp(\tilde{\xi}_t)(e_i, e_j)\d t\right)\\
    \d\tilde{\xi}_t=-\gamma\tilde{\xi}_t \d t-\lt{g_0}\nabla U(\tilde{g}_0)\d t+\sqrt{2\gamma} \d W_t
\end{cases}
\end{equation}
where $\d^2\exp :\g\times \g\rightarrow \g$ is defined as
\begin{align}
    \label{eqn_d2_exp}
        \d^2\exp(x)(\xi, \hat{\xi})&=\lim_{h\rightarrow 0}\ip{\d\exp(x+h\xi)}{\hat{\xi}}-\ip{\d\exp(x)}{\hat{\xi}}
    \end{align}
\end{lemma}
\begin{proof}
    The SDE for $\xi$ is from the definition of the numerical discretization directly. We focus on the SDE for $g$. The update of $g$ gives $\tilde{g}_t=g\exp(t\tilde{\xi}_t)$, and we will first write down the SDE $t\tilde{\xi}_t$ satisfied as the following:
    \begin{align*}
        \d t\tilde{\xi}_t=(1-t\gamma)\tilde{\xi}_t \d t-t \lt{g_0}\nabla U(\tilde{g}_0)\d t+t\sqrt{2\gamma} \d W_t
    \end{align*}
    By Ito's formula and the expression of $\d\exp$ in Eq. \eqref{eqn_d_exp}, 
    \begin{align*}
        \d \tilde{g}_t&=\d\exp(t\tilde{\xi}_t)\\
        &=\lte{\exp(t\tilde{\xi}_t)}\left(\d \exp_{t\tilde{\xi}_t}(dt\tilde{\xi}_t)+\gamma t^2\sum_{ij}\d^2\exp(t\tilde{\xi}_t)(e_i, e_j) \d t \right)
    \end{align*}
    where the second term in the bracket is Ito's correction drift term.
\end{proof}

\begin{lemma}[Upper bound for $\d^2\exp$]
\label{lemma_d2exp_upper_bound}
The map $\d^2\exp(X): \g\times \g\to \g$ defined in Eq. \eqref{eqn_d2_exp} has upper bound $C$, i.e., 
\begin{equation*}
    \norm{\d^2  \exp(X)}_{op}\le C
\end{equation*}
 where $C$ is defined in Eq. \eqref{eqn_C} and $\norm{\cdot}_{op}$ is the operator norm (maximum eigenvalue of a linear operator).
\end{lemma}

\begin{proof}[Proof of Lemma \ref{lemma_d2exp_upper_bound}]
    By the definition of $\d^2\exp$ given by Eq. \eqref{eqn_d2_exp} and $\d\exp$ in Eq. \ref{eqn_d_exp}, we have
    \begin{align*}
        \d^2\exp(X)(Y, Z)&=\lim_{h\rightarrow 0}\ip{\d\exp(X+hY)}{Z}-\ip{\d\exp(x)}{Z}\\
        &=\lim_{h\rightarrow 0}\frac{1-\exp(-\ad_{X+hY})}{\ad_{X+hY}}Z-\frac{1-\exp(-\ad_{X})}{\ad_{x}}Z
    \end{align*}

    Setting $Y=Z$,
    \begin{align*}
        \norm{\d^2\exp(X)(Y, Y)}\le \norm{\ad_Y}_{op} \norm{Y} \max\left\{\frac{\d}{\d x}\Re \frac{1-\exp(x i)}{x i}\right\}\le C\norm{Y}^2
    \end{align*}
\end{proof}

\begin{lemma}[Upper bound $\frac{\partial^2}{\partial e_i \partial e_j} d^2(\cdot, g)$]
\label{lemma_d2_ito_correction_upper_bound}
Define
\begin{align*}
\frac{\partial^2}{\partial e_i \partial e_j} d^2(\tilde{g}, g):=\lim_{h\to 0}\frac{d^2(\tilde{g}\exp(he_i)\exp(he_j), g\exp(he_i))-d^2(\tilde{g}\exp(he_j), g)}{h^2}
\end{align*}
Then we have
\begin{align*}
    \frac{\partial^2}{\partial e_i \partial e_j} d^2(\tilde{g}, g)\le C d(\tilde{g}, g)
\end{align*}
\end{lemma}
\begin{proof}[Proof of Lemma \ref{lemma_d2_ito_correction_upper_bound}]
    \begin{align*}
    \frac{\partial^2}{\partial e_i \partial e_j} d^2(\tilde{g}, g)&:=\lim_{h\to 0}\frac{d^2(\tilde{g}\exp(he_i)\exp(he_j), g\exp(he_i))-d^2(\tilde{g}\exp(he_j), g)}{h^2}\\
    &=\lim_{h\to 0}\frac{d^2(\tilde{g}\exp(he_i)\exp(he_j)\exp(-he_j), g)-d^2(\tilde{g}, g)}{h^2}\\
    &=\lim_{h\to 0}\frac{d^2(\tilde{g}\exp(h^2\lb{e_i}{e_j}), g)-d^2(\tilde{g}, g)}{h^2}\\
    &\le 2d(\tilde{g}, g)\norm{\lb{e_i}{e_j}}\\
    &\le 2C d(\tilde{g}, g)
\end{align*}
\end{proof}

\begin{lemma}
\label{lemma_dexp_upper_bound}
The operator norm for $\d\exp$ is bounded by
\begin{align*}
    \norm{\d\exp_X}_{op}\le 1
\end{align*}
\end{lemma}
\begin{proof}[Proof of Lemma \ref{lemma_dexp_upper_bound}]
    This can be derived from Eq.\eqref{eqn_d_exp}. Since $\norm{\frac{1-e^{xi}}{xi}}\le 1, \forall x$ and all the eigenvalue of $\ad_X$ is pure imaginary by the skew-symmetricity of $\ad_X$ operator, we have the desired result.
\end{proof}

\begin{lemma}
\label{lemma_bound_xi_power}
    Suppose we initialize $\xi(0)\equiv 0$ and we have $\norm{\nabla f}$ is upper bounded by $LD$. Then we have $\mathbb{E}\norm{\xi}^k< \infty$ for both SDE Eq. \eqref{eqn_sampling_SDE} and splitting scheme Eq. \eqref{eqn_discrete_splitting}.
\end{lemma}
\begin{proof}[Proof of Lemma \ref{lemma_bound_xi_power}]
    The drift term for momentum has magnitude $\le LD$ and by Tanaka formula, we have
    \begin{align*}
        \frac{\d}{\d t}\mathbb{E}\norm{\xi}&=-\gamma \mathbb{E}\norm{\xi}-\mathbb{E}\ip{\frac{\xi}{\norm{\xi}}}{\nabla f}\\
        &\le -\gamma \mathbb{E}\norm{\xi}+LD
    \end{align*}
    as a result, $\mathbb{E}\norm{\xi} \le \frac{LD}{\gamma}:=C_1$.
    \begin{align*}
        \frac{\d}{\d t}\mathbb{E}\norm{\xi}^2&=-2\gamma\mathbb{E}\norm{\xi}^2-2\mathbb{E}\ip{\xi}{\nabla f}+m\\
        &\le -2\gamma \mathbb{E}\norm{\xi}^2+2LDC_1+m
    \end{align*}
    As a result, 
    \begin{equation*}
        \mathbb{E}\norm{\xi}^2\le \frac{2LDC_1+m}{2\gamma}:=C_2
    \end{equation*}
    We prove by induction. Suppose for $i=1, \dots k-1$ $\mathbb{E}\norm{\xi}^i\le C_i$, we have for $k\ge 3$
    \begin{align*}
        \frac{\d}{\d t}\mathbb{E}\norm{\xi}^k&=-\gamma k\mathbb{E}\norm{\xi}^k-k\mathbb{E}\norm{\xi}^{k-2}\ip{\xi}{\nabla f}+\gamma\mathbb{E}\Delta\norm{\xi}^k\\
        &=-\gamma k\mathbb{E}\norm{\xi}^k-k\mathbb{E}\norm{\xi}^{k-2}\ip{\xi}{\nabla f}+\gamma k(k-1)\mathbb{E}\norm{\xi}^{k-2}\\
        &\le -\gamma k\mathbb{E}\norm{\xi}^k+kLDC_{k-1}+\gamma k (k-1) C_{k-2}
    \end{align*}
    Since we have $\xi(0)\equiv 0$, we have $\mathbb{E}\norm{\xi}^k\le \frac{LDC_{k-1}}{\gamma}+(k-1)C_{k-2}:=C_k<\infty$
\end{proof}

\begin{proof}[Proof of Thm. \ref{thm_local_error_W2}]
We consider the evolution of $\mathbb{E}d^2((g_t, \xi_t), (\tilde{g}_t, \tilde{\xi}_t))$, where $(g_t, \xi_t)$ and $(\tilde{g}_t, \tilde{\xi}_t)$ are 2 trajectories strating from the same point $(g_0, \xi_0)$.
\begin{align*}
    &\frac{\d}{\d t}\mathbb{E}d^2((g_t, \xi_t), (\tilde{g}_t, \tilde{\xi}_t))\\
    &=\frac{\d}{\d t}\mathbb{E}d^2(g_t, \tilde{g}_t)+\frac{\d}{\d t} \mathbb{E}d^2(\xi_t,\tilde{\xi}_t)
\end{align*}
We quantify $\frac{\d}{\d t}\mathbb{E}d^2(\xi_t,\tilde{\xi}_t)$ first:
\begin{align*}
    \frac{\d}{\d t}\mathbb{E}\norm{\tilde{\xi}_t- \xi_t}^2&\le 2\mathbb{E}\left[\norm{\tilde{\xi}_t- \xi_t}\norm{\nabla U(g_t)-\nabla(g_0)}\right]\\
    &\le 2L\mathbb{E}\left[\norm{\tilde{\xi}_t- \xi_t}d(g_t,g_0)\right]\\
    &\le 2L\mathbb{E}\left[\norm{\tilde{\xi}_t- \xi_t}\int_0^t \norm{\xi_t}\d t\right]\\
    &\le L\mathbb{E}\norm{\tilde{\xi}_t- \xi_t}^2+L\mathbb{E}\left(\int_0^t \norm{\xi_t}\d t\right)^2\\
    &\le L\mathbb{E}\norm{\tilde{\xi}_t- \xi_t}^2+t^2 L C_2
\end{align*}
where the last line is because Cauchy-Schwarz inequality and Lemma \ref{lemma_bound_xi_power} gives us
\begin{align*}
    &\mathbb{E}\left(\int_0^t \norm{\xi_t}\d t\right)^2\le t\int_0^t \mathbb{E}\norm{\xi_t}^2\d t\le t^2 C_2
\end{align*}

Gronwall's inequality gives $\mathbb{E}\norm{\tilde{\xi}_t- \xi_t}^2\le x$ where $x$ is the solution of ODE $\dot{x}=Lx+C_2 L t^2$ with initial condition $x(0)=0$.

Next, we bound $\mathbb{E}d^2(\tilde{g}_t, g_t)$. When we have $\tilde{g}_t^{-1} g_t \notin N$, Ito's formula gives
\begin{align*}
    &\d d^2(\tilde{g}_t, g_t)\\
    &= 2\ip{\log g_t^{-1}\tilde{g}_t}{\d\tilde{g}_t-\d g_t}\d t+\gamma t^2 \sum_{ijk}\ip{\frac{1-e^{\ad_{t\tilde{\xi}_t}}}{\ad_{t\tilde{\xi}_t}}e_k}{e_i}\ip{\frac{1-e^{\ad_{t\tilde{\xi}_t}}}{\ad_{t\tilde{\xi}_t}}e_k}{e_j}\frac{\partial^2}{\partial e_i \partial e_j} d^2(\tilde{g}_t, g_t) \d t
\end{align*}
where the second term is the Ito's correction. It can be bounded by Lemma \ref{lemma_d2_ito_correction_upper_bound} and \ref{lemma_dexp_upper_bound}.
\begin{align*}
    \sum_{ijk}\ip{\frac{1-e^{\ad_{t\tilde{\xi}_t}}}{\ad_{t\tilde{\xi}_t}}e_k}{e_i}\ip{\frac{1-e^{\ad_{t\tilde{\xi}_t}}}{\ad_{t\tilde{\xi}_t}}e_k}{e_j}\frac{\partial^2}{\partial e_i \partial e_j} d^2(\tilde{g}_t, g_t)\le 2C m^3 d(\tilde{g}_t, g_t)
\end{align*}

After taking expectations, we only need to consider $\tilde{g}_t^{-1}g_t \in \G\backslash N$ since $N$ is zero measured (Cor. \ref{cor_absolute_continuous}). Also, the local martingale part can be eliminated. Using Lemma \ref{lemma_d2exp_upper_bound} and \ref{lemma_bound_xi_power} we have
\begin{align*}
    &\frac{\d}{\d t}\mathbb{E}d^2(\tilde{g}_t, g_t)=\frac{\d}{\d t}\mathbb{E}_{\tilde{g}_t^{-1}g_t\notin N}d^2(\tilde{g}_t, g_t)\\
    &=2\mathbb{E}\ip{\log g_t^{-1}\tilde{g}_t}{\tilde{\xi}_t-\xi_t+t\frac{1-e^{\ad_{t\tilde{\xi}_t}}}{\ad_{t\tilde{\xi}_t}} \left(-\gamma\tilde{\xi}_t-\lt{g_0}\nabla U(\tilde{g}_0)\right)+\gamma t^2\sum_{ij}\d^2 \exp(t\tilde{\xi}_t)(e_i, e_j)}\\
    &+\gamma t^2 \sum_{ijk}\ip{\frac{1-e^{\ad_{t\tilde{\xi}_t}}}{\ad_{t\tilde{\xi}_t}}e_k}{e_i}\ip{\frac{1-e^{\ad_{t\tilde{\xi}_t}}}{\ad_{t\tilde{\xi}_t}}e_k}{e_j}\frac{\partial^2}{\partial e_i \partial e_j} d^2(\tilde{g}_t, g_t)\\
    &\le \mathbb{E}d^2(\tilde{g}_t, g_t)+\mathbb{E}\norm{\tilde{\xi}_t-\xi_t+t\frac{1-e^{\ad_{t\tilde{\xi}_t}}}{\ad_{t\tilde{\xi}_t}} \left(-\gamma\tilde{\xi}_t-\lt{g_0}\nabla U(\tilde{g}_0)\right)+\gamma t^2\sum_{ij}\d^2 \exp(t\tilde{\xi}_t)(e_i, e_j)}^2\\
    &+2\gamma t^2 Cm^3\mathbb{E}d(\tilde{g}_t, g_t)\\
    &\le \mathbb{E}d^2(\tilde{g}_t, g_t)+4\left(\mathbb{E}\norm{\tilde{\xi}_t-\xi_t}^2+\gamma^2t^2 \mathbb{E}\norm{\tilde{\xi}_t}^2 +t^2 \norm{\nabla U(\tilde{g}_0)}^2+\gamma^2 t^4 C^2 m^2\right)\\
    &+\gamma t^2 Cm^3 \mathbb{E}d^2(\tilde{g}_t, g_t)+\gamma t^2 Cm^3 \\
    &\le \mathbb{E}d^2(\tilde{g}_t, g_t)+4\left(\mathbb{E}\norm{\tilde{\xi}_t-\xi_t}^2+\gamma^2t^2 C_2 +t^2 LD+\gamma^2 t^4 C^2 m^2\right)+\gamma t^2 Cm^3 \mathbb{E}d^2(\tilde{g}_t, g_t)+\gamma t^2 Cm^3\\
    &= \left(1+\gamma Cm^3  t^2\right)\mathbb{E}d^2(\tilde{g}_t, g_t)+4\mathbb{E}\norm{\tilde{\xi}_t-\xi_t}^2+\left(4\gamma^2 C_2 + 4LD+\gamma  Cm^3\right) t^2+4\gamma^2  C^2 m^2 t^4
\end{align*}

Gronwall's inequality gives $\mathbb{E}d^2(\tilde{g}_t, g_t)\le y(t)$ where $y$ is the solution of ODE 
\begin{equation*}
    \dot{y}=\left(1+\gamma Cm^3  t^2\right)y+4x+\left(4\gamma^2 C_2 + 4LD+\gamma  Cm^3\right) t^2+4\gamma^2  C^2 m^2 t^4
\end{equation*}
with initial condition $y(0)=0$.

To summarize, the ODE for $x(t)$ and $t(t)$ are given by
\begin{equation}
\label{eqn_xy_ode}
    \begin{cases}
        \dot{x}=Lx+C_2 L t^2\\
        \dot{y}=\left(1+\gamma Cm^3  t^2\right)y+4x+\left(4\gamma^2 C_2 + 4LD+\gamma  Cm^3\right) t^2+4\gamma^2  C^2 m^2 t^4
    \end{cases}
    \begin{cases}
        x(0)=0\\
        y(0)=0
    \end{cases}
\end{equation}
Its solution can be analytically obtained, but omitted for the sake of length.

\end{proof}

\subsection{Proof of Thm. \ref{thm_error_propagation}}

\begin{proof}[Proof of Lemma \ref{lemma_triangular_ineq_rho}]
By the definition of semi-distance $\rho$ in Eq. \eqref{eqn_rho}, we have
    \begin{align*}
        &\rho((\hat{g}, \hat{\xi}), (\tilde{g}, \tilde{\xi}))-\rho((\hat{g}, \hat{\xi}), (g, \xi))\\
        &=f(r((\hat{g}, \hat{\xi}), (\tilde{g}, \tilde{\xi})))\left(1+\beta\norm{\hat{\xi}-\tilde{\xi}}^2\right)\\
        &-f(r((\hat{g}, \hat{\xi}), (g, \xi))\left(1+\beta\norm{\hat{\xi}-\xi}^2\right)\\
        &=\beta f(r((\hat{g}, \hat{\xi}), (\tilde{g}, \tilde{\xi})))\left(\norm{\hat{\xi}-\tilde{\xi}}^2-\norm{\hat{\xi}-\xi}^2\right)\\
        &+\left(f(r((\hat{g}, \hat{\xi}), (\tilde{g}, \tilde{\xi})))-f(r((\hat{g}, \hat{\xi}), (g, \xi)\right)\left(1+\beta\norm{\hat{\xi}-\xi}^2\right)
    \end{align*}
    We will bound the 2 terms in the last equation separately.
    
    The first term:
    \begin{align*}
    &\norm{\hat{\xi}-\tilde{\xi}}^2-\norm{\hat{\xi}-\xi}^2\\
    &=\norm{\tilde{\xi}-\xi}\norm{2\hat{\xi}-\tilde{\xi}-\xi}\\
    &\le 2\delta\norm{\hat{\xi}-\xi}+\delta^2 \\
    &\le \norm{\hat{\xi}-\xi}^2+2\delta^2
    \end{align*}
    and as a result,
    \begin{align*}
        &\beta f(r((\hat{g}, \hat{\xi}), (g, \xi)))\left(\norm{\hat{\xi}-\tilde{\xi}}^2-\norm{\hat{\xi}-\xi}^2\right)\\
        &\le \beta f(r((\hat{g}, \hat{\xi}), (g, \xi)))\left(\norm{\hat{\xi}-\xi}^2+2\delta^2\right) \\
        &\le \rho((\hat{g}, \hat{\xi}), (g, \xi))\delta+2\beta f(R) \delta^2
    \end{align*}

    The second term: Because $f'\le 1$
    \begin{align*}
        &(f(r((\hat{g}, \hat{\xi}), (\tilde{g}, \tilde{\xi})))-f(r((\hat{g}, \hat{\xi}), (g, \xi))))\left(1+\beta\norm{\hat{\xi}-\xi}^2\right)\\
        &\le (r((\hat{g}, \hat{\xi}), (\tilde{g}, \tilde{\xi})))-r((\hat{g}, \hat{\xi}), (g, \xi)))\left(1+\beta\norm{\hat{\xi}-\xi}^2\right)\\
        &\le\left(1+\alpha+\gamma^{-1}\right)\left(1+\beta\norm{\hat{\xi}-\xi}^2\right)\delta\\
        &\le\left(1+\alpha+\gamma^{-1}\right)\frac{f(r((\hat{g}, \hat{\xi}), (g, \xi))))}{f(R)}\left(1+\beta\norm{\hat{\xi}-\xi}^2\right)\delta\\
        &=\left(1+\alpha+\gamma^{-1}\right)\frac{1}{f(R)}\rho((\hat{g}, \hat{\xi}), (g, \xi))\delta
    \end{align*}
    
    Sum them up we have
    \begin{align*}
        &\rho((\hat{g}, \hat{\xi}), (\tilde{g}, \tilde{\xi}))-\rho((\hat{g}, \hat{\xi}), (g, \xi))\\
         &\le \left(1+\frac{1+\alpha+\gamma^{-1}}{f(R)} \right)\rho((\hat{g}, \hat{\xi}), (g, \xi))\delta+2\beta f(R)\delta^2
    \end{align*}
    We have the desired result by defining 
    \begin{equation}
    \label{eqn_A}
        \begin{cases}
            A_1:=1+\frac{1+\alpha+\gamma^{-1}}{f(R)}\\
            A_2:=2\beta f(R)
        \end{cases}
    \end{equation}
\end{proof}

\begin{proof}[Proof of Lemma \ref{lemma_order_W2_error}]
    We focus on $x$ first. Since $\dot{x}\le Lx+C_2 L t_0^2$, we have $x(t)\le C_2 t_0^2 (e^{Lt}-1)= \mathcal{O}(t^3)$.

    For $y$, we have 
    \begin{align*}
        \dot{y}&\le \left(1+\gamma Cm^3  t_0^2\right)y+4x+\left(4\gamma^2 C_2 + 4LD+\gamma  Cm^3\right) t_0^2+4\gamma^2  C^2 m^2 t_0^4\\
        &\le \left(1+\gamma Cm^3  t_0^2\right)y+4C_2 t_0^2 e^{Lt}+\left(4\gamma^2 C_2-4C_2+ 4LD+\gamma  Cm^3\right) t_0^2+4\gamma^2  C^2 m^2 t_0^4
    \end{align*}
    Solving the ODE explicitly, we have
    \begin{align*}
        y\le \frac{4\gamma^2 C_2-4C_2+ 4LD+\gamma  Cm^3+4\gamma^2  C^2 m^2 t_0^2}{1+\gamma Cm^3  t_0^2}t_0^2\left(e^{\left(1+\gamma Cm^3  t_0^2\right)t}-1\right)+\frac{e^{\left(1+\gamma Cm^3  t_0^2\right) t}-e^{Lt}}{1+\gamma Cm^3  t_0^2-L}4C_2 t_0^2
    \end{align*}
    Since we have $e^t-1= \mathcal{O}(t)$, we have $y(t)= \mathcal{O}(t^3)$.
\end{proof}

\begin{corollary}
\label{cor_rho_upper_bounded}
Suppose the Lie group $\G$ is compact with diameter $D$. For any $(g, \xi)$ and $(\hat{g}, \hat{\xi})$ in the following cases:
\begin{itemize}
    \item $(g, \xi)$ is the solution of sampling SDE Eq. \eqref{eqn_sampling_SDE} at time $t$ for any $t\ge 0$.\\
    $(\hat{g}, \hat{\xi})$ is the solution of sampling SDE Eq. \eqref{eqn_sampling_SDE} at time $s$ for any $s\ge 0$.
    \item $(g, \xi)$ is the solution of numerical discretization Eq. \eqref{eqn_discrete_splitting} at step $k$ for any $k\in \mathbb{N}$. \\
    $(\hat{g}, \hat{\xi})$ is the solution of numerical discretization Eq. \eqref{eqn_discrete_splitting} at step $l$ for any $l\in \mathbb{N}$.
    \item $(g, \xi)$ is the solution of sampling SDE Eq. \eqref{eqn_sampling_SDE} at time $t$ for any $t\ge 0$. \\
    $(\hat{g}, \hat{\xi})$ is the solution of numerical discretization Eq. \eqref{eqn_discrete_splitting}  at step $k$ for any $k\in \mathbb{N}$.
\end{itemize}
then we have 
    \begin{align*}
        \mathbb{E}\rho^2((\hat{g}, \hat{\xi}), (g, \xi))&\le f(R_1)^2(1+8C_2+16C_4)
    \end{align*}
    where $C_2$ and $C_4$ are given in Lemma \ref{lemma_bound_xi_power}.
\end{corollary}
\begin{proof}[Proof of Cor. \ref{cor_rho_upper_bounded}]
By our design that $f$ is upper bounded by $f(R)$, we have
    \begin{align*}
        \rho((g, \xi), (\hat{g}, \hat{\xi}))&\le f(R_1)(1+\norm{\xi-\hat{\xi}}^2)\\
        &\le f(R_1)(1+2\norm{\xi}^2+2\norm{\hat{\xi}}^2)
    \end{align*}
    and by Cauchy-Schwarz inequality
    \begin{align*}
        \rho^2 ((g, \xi), (\hat{g}, \hat{\xi}))&\le f(R_1)^2\left(1+4\norm{\xi}^2+4\norm{\hat{\xi}}^2+8\norm{\xi}^4+8\norm{\hat{\xi}}^4\right)
    \end{align*}
    By taking expectation and Lemma \ref{lemma_bound_xi_power}, we have the corollary proved.
\end{proof}

\begin{proof}[Proof of Thm. \ref{thm_error_propagation}]
In this proof, we denote $\delta:=d((\tilde{g}_t, \tilde{\xi}_t), (g_t, \xi_t))$. Using the modified triangle inequality (Lemma \ref{lemma_triangular_ineq_rho}), 
    \begin{align*}
        &\mathbb{E}\rho((\hat{g}_t, \hat{\xi}_t), (\tilde{g}_t, \tilde{\xi}_t))\\
        &\le\mathbb{E}\left[(1+A_1\delta)\rho((\hat{g}_t, \hat{\xi}_t), (g_t, \xi_t))+A_2\delta^2\right]\\
        &\le\exp(-ct)\mathbb{E}\rho((\hat{g}_0, \hat{\xi}_0), (g_0, \xi_0))+A_1\mathbb{E}\delta\rho((\hat{g}_t, \hat{\xi}_t), (g_t, \xi_t))+A_2\mathbb{E}\delta^2\\
        &\le\exp(-ct)\mathbb{E}\rho((\hat{g}_0, \hat{\xi}_0), (g_0, \xi_0))+\frac{A_1}{2}\mathbb{E}\left[\frac{\delta^2}{t^\frac{3}{2}}+t^\frac{3}{2}\rho^2((\hat{g}_t, \hat{\xi}_t), (g_t, \xi_t))\right]+A_2\mathbb{E}\delta^2
    \end{align*}
    where the last step is because of the Cauchy-Schwarz inequality. By Cor. \ref{cor_rho_upper_bounded}, we have $\mathbb{E}\rho$ and $\mathbb{E}\rho^2$ is bounded and
    \begin{align*}
        &\mathbb{E}\rho((\hat{g}_t, \hat{\xi}_t), (\tilde{g}_t, \tilde{\xi}_t))\\
        &\le\exp(-ct)\mathbb{E}\rho((\hat{g}_0, \hat{\xi}_0), (g_0, \xi_0))+\left(\frac{A_1}{2t^\frac{3}{2}}+A_2\right)\mathbb{E}\delta^2+\frac{A_1}{2}f(R)^2(1+8C_2+16C_4)t^\frac{3}{2}
    \end{align*}
\end{proof}

\subsection{Proof of Cor. \ref{cor_error_Erho}, Thm. \ref{thm_global_error_Wrho} and Thm. \ref{thm_global_error_W2}}
\begin{proof}[Proof of Cor. \ref{cor_error_Erho}]
    We apply Thm. \ref{thm_error_propagation} recurrently, which gives
    \begin{align*}
        &\mathbb{E}\rho((\hat{g}_{kh}, \hat{\xi}_{kh}), (\tilde{g}_k, \tilde{\xi}_k))\\
        &\le e^{-ch} \mathbb{E}\rho((\hat{g}_{(k-1)h}, \hat{\xi}_{(k-1)h}), (\tilde{g}_{(k-1)h}, \tilde{\xi}_{(k-1)h})), (\tilde{g}_{k-1}, \tilde{\xi}_{k-1}))+E(h)\\
        &\le e^{-2ch} \mathbb{E}\rho((\hat{g}_{(k-2)h}, \hat{\xi}_{(k-1)h}), (\tilde{g}_{(k-2)h}, \tilde{\xi}_{(k-2)h})), (\tilde{g}_{k-2}, \tilde{\xi}_{k-2}))+(1+e^{-ch})E(h)\\
        &\le \cdots\\
        &\le e^{-ckh} \mathbb{E}\rho((\hat{g}_{0}, \hat{\xi}_{0}), (\tilde{g}_{0}, \tilde{\xi}_{0})), (\tilde{g}_{0}, \tilde{\xi}_{0}))+\sum_{i=0}^k e^{-cih}E(h)\\
        &\le e^{-ckh} \mathbb{E}\rho((\hat{g}_{0}, \hat{\xi}_{0}), (\tilde{g}_{0}, \tilde{\xi}_{0})), (\tilde{g}_{0}, \tilde{\xi}_{0}))+\frac{E(h)}{1-\exp(-ch)}
    \end{align*}
\end{proof}

\begin{proof}[Proof of Thm. \ref{thm_global_error_Wrho}]
    Let $\Law(\hat{g}_0, \hat{\xi}_0)$ follows the target distribution and $(\hat{g}_t, \hat{\xi}_t)$ is the solution of the sampling dynamics \eqref{eqn_sampling_SDE}. As a result, $\Law(\hat{g}_t, \hat{\xi}_t)\sim\nu_*, \forall t$. We make $(\hat{g}_0, \hat{\xi}_0)$ and $(\tilde{g}_0, \tilde{\xi}_0)$ are paired in the optimal way under $\rho$, i.e., $\mathbb{E}\rho((\hat{g}_0, \hat{\xi}_0), (\tilde{g}_0, \tilde{\xi}_0))=W_\rho(\Law(\hat{g}_0, \hat{\xi}_0), \Law(\tilde{g}_0, \tilde{\xi}_0))$. 
    
    Cor. \ref{cor_error_Erho} gives
    \begin{align*}
        &W_\rho(\tilde{\nu}_k, \nu_*)\\
        &W_\rho(\Law(\hat{g}_k, \hat{\xi}_k), \Law(\tilde{g}_k, \tilde{\xi}_k))\\
        &\le \mathbb{E}\rho((\hat{g}_{kh}, \hat{\xi}_{kh}), (\tilde{g}_k, \tilde{\xi}_k))\\
        &\le e^{-ckh} \mathbb{E}\rho((\hat{g}_0, \hat{\xi}_0), (\tilde{g}_0, \tilde{\xi}_0))+\frac{E(h)}{1-\exp(-ch)}\\
        &= e^{-ckh} W_\rho(\Law(\hat{g}_0, \hat{\xi}_0), \Law(\tilde{g}_0, \tilde{\xi}_0))+\frac{E(h)}{1-\exp(-ch)}
    \end{align*}
\end{proof}

\begin{proof}[Proof of Thm. \ref{thm_global_error_W2}]
    This is a direct corollary of Lemma \ref{cor_equivalence_W2} and Thm. \ref{thm_global_error_Wrho}.
\end{proof}

\end{document}